\documentclass[12pt]{article}
\pdfoutput=1
\usepackage{amsmath,amssymb,amsthm,amscd}
\usepackage{graphicx}
\usepackage{tikz}
\usepackage{float}
\usetikzlibrary{intersections,calc,arrows.meta}

\newtheorem{theorem}{Theorem}[section]
\newtheorem{lemma}[theorem]{Lemma}

\newtheorem{conjecture}[theorem]{Conjecture}

\theoremstyle{definition}
\newtheorem{definition}[theorem]{Definition}

\theoremstyle{remark}
\newtheorem{remark}[theorem]{Remark}

\numberwithin{equation}{section}

\newcommand\C{\mathbb{C}}
\newcommand\F{\mathbb{F}}

\newcommand\Z{\mathbb{Z}}

\newcommand\N{\mathbb{N}}
\newcommand\T{\mathbb{T}}
\newcommand\cC{\mathcal{C}}

\newcommand\cG{\mathcal{G}}

\newcommand\fA{\mathfrak{A}}
\newcommand\fD{\mathfrak{D}}
\newcommand\fS{\mathfrak{S}}
\newcommand\bm{{\mathbf{m}}}

\newcommand\Aut{\operatorname{Aut}}
\newcommand\Out{\operatorname{Out}}
\newcommand\End{\operatorname{End}}
\newcommand\id{\mathrm{id}}
\newcommand\Ad{\mathrm{Ad}}
\newcommand\Irr{\operatorname{Irr}}
\newcommand\Ind{\operatorname{Ind}}
\newcommand\Rep{\operatorname{Rep}}

\newcommand\biota{\bar{\iota}}
\newcommand\bkappa{\bar{\kappa}}

\newcommand\bdelta{\bar{\delta}}
\newcommand\bepsilon{\bar{\epsilon}}
\newcommand\heta{\hat{\eta}}

\title{Group-subgroup subfactors revisited}
\author{Masaki Izumi
\thanks{Supported in part by JSPS KAKENHI Grant Number JP20H01805}\\
Graduate School of Science \\
Kyoto University \\
Sakyo-ku, Kyoto 606-8502, Japan \\
izumi@math.kyoto-u.ac.jp}

\date{}
\begin{document} 
\maketitle
\begin{center}\textit{In memory of Vaughan Jones}
\end{center}

\begin{abstract} For all Frobenius groups and a large class of finite multiply transitive permutation groups, 
we show that the corresponding group-subgroup subfactors are completely characterized by their principal graphs. 
The class includes all the sharply $k$-transitive permutation groups for $k=2,3,4$, and in particular the Mathieu group $M_{11}$ 
of degree 11.   
\end{abstract}

\section{Introduction}
The  classical Goldman's theorem \cite{G59} says, in modern term, that every index 2 inclusion $M\supset N$ of type II$_1$ factors 
is given by the crossed product $M=N\rtimes \Z_2$, where $\Z_2$ is the cyclic group of order 2. 
It is a famous story that this fact is one of the motivating examples when Vaughan Jones introduced his cerebrated 
notion of index for subfactors \cite{J83}. 
In the case of index 3, there are two different cases: their principal graphs are either the Coxeter graph $D_4$ or $A_5$ 
(see \cite{EK}, \cite{GHJ} for example). 
In the $D_4$ case, the subfactor is given by the crossed product $M=N\rtimes \Z_3$. 
In the $A_5$ case, we showed in \cite{I92} that there exists a unique subfactor $R\subset N$, up to inner conjugacy, such that 
$$M=R\rtimes \fS_3\supset N=R\rtimes \fS_2$$
holds where $\fS_n$ denotes the symmetric group of degree $n$.  
We call such a result \textit{Goldman-type theorem}, uniquely recovering the subfactor $R$ and a group action on it 
solely from one of the principal graphs of $M\supset N$. 
More Goldman-type theorems were obtained in \cite{HS96},\cite{H95}, and \cite{I97}, but here we should emphasize 
that only Frobenius groups had been treated until we recently showed a Goldman-type theorem for the alternating groups 
$\fA_5>\fA_4$ (\cite[Theorem A1]{IMPP15}). 

Let $G$ be a finite group, let $H$ be a subgroup of it, and let $\alpha$ be an outer action of $G$ on a factor $R$. 
Then the inclusion 
$$M=R\rtimes_\alpha G\supset N=R\rtimes_\alpha H$$
is called a \textit{group-subgroup subfactor}. 
Let $L$ be the kernel of the permutation representation of $G$ acting on $G/H$, which is the largest normal 
subgroup of $G$ contained in $H$. 
Then the inclusion $M\supset N$ remembers at most the information of $G/L>H/L$, and so whenever we discuss 
group-subrgroup subfactors, we always assume that $L$ is trivial, or more naturally, we treat $G$ as a transitive permutation group acting on a finite set and $H$ as a point stabilizer. 
A Frobenius group $G$ is a semi-direct product $K\rtimes H$ with a free $H$ action on $K\setminus \{e\}$. 
In this paper, we show Goldman-type theorems for all Frobenius groups and for a large class of multiply transitive permutation groups. 

One might suspect that every question about group-subgroup subfactors should be reduced to an easy exercise in either permutation group theory 
or representation theory, which turned out to be not always the case. 
Indeed, Kodiyalam-Sunder \cite{KS00} showed that two pairs of groups $\fS_4>\Z_4$ and $\fS_4>\Z_2\times \Z_2$ give isomorphic 
group-subgroup subfactors, which cannot be understood either in permutation group theory or representation theory. 
In \cite{I01}, we gave a complete characterization of two isomorphic group-subgroup subfactors coming from two different permutation groups 
in terms of fusion categories and group cohomology. 
To understand this kind of phenomenon, the representation category of a group should be treated as an abstract fusion category, and ordinary 
representation theory is not strong enough.

When I discussed the above result \cite{I01} with Vaughan more than 10 years ago, he asked me whether the Kodiyalam-Sunder-type phenomena occur for primitive permutation groups, or in other words, when $H$ is a maximal subgroup in $G$. 
Theorem 2.3 of \cite{I01} shows that the answer is `no', and when I told it to him, somehow he looked content. 
I guess Vaughan believed that one should assume primitivity of the permutation group $G$ to obtain reasonable results in group-subgroup subfactors. 
Probably he was right because the primitivity of $G$ is equivalent to the condition that the corresponding group-subgroup 
subfactor has no non-trivial intermediate subfactor, and such a subfactor is known to be very rigid.  
This assumption also rules out the following puzzling example: while the principal graph of the group-subgroup subfactor for 
$\fD_8=\Z_4\rtimes_{-1}\Z_2>\Z_2$ is the Coxeter graph $D_6^{(1)}$, 
there are 3 other subfactors sharing the same principal graph but they are not group-subgroup subfactors (\cite[Theorem 3.4]{IK92}). 
This means that a Goldman-type theorem never holds for $\fD_8>\Z_2$. 
Note that $\Z_2$ is not a maximal subgroup of $\fD_8$, and hence the $\fD_8$-action on $\fD_8/\Z_2$ is not primitive. 

Typical examples of primitive permutation groups are multiply transitive permutation groups, and we mainly work on 
Goldman-type theorems for them in this paper. 
We briefly recall the basic definitions related to them here.  
Let $G$ be a permutation group on a finite set $X$. 
For $k\in \N$, we denote by $X^{[k]}$ the set of all ordered tuples $(a_1,a_2,\ldots,a_k)$ consisting of distinct elements in $X$. 
The group $G$ acts on $X^{[k]}$ by $g\cdot(a_1,a_2,\ldots,a_k)=(ga_1,ga_2,\ldots,ga_k)$, and we always consider this action. 
For $x\in X$, we denote by $G_x$ the stabilizer of $x$ in $G$, and for $(x_1,x_2,\ldots,x_k)\in X^{[k]}$ we denote 
$$G_{x_1,x_2,\ldots,x_k}=\bigcap_{i=1}^kG_{x_i}. $$
We say that $G$ is $k$-transitive if the $G$-action on $X^{[k]}$ is transitive. 
This is equivalent to the condition that the $G_{x_1,x_2,\ldots,x_{k-1}}$-action on $X\setminus\{x_1,x_2,\ldots,x_{k-1}\}$ is transitive. 
We say that $G$ is regular if $G$ is free and transitive.  
A Goldman-type theorem for a regular permutation group is nothing but the characterization of crossed products 
(see \cite{PP86},\cite{K89}). 

As will be explained in Subsection 2.5 in detail, our strategy for proving a Goldman-type theorem for $G>G_{x_1}$ is an induction 
argument reducing it to that of $G_{x_1}>G_{x_1,x_2}$.  
Assume that $G$ is $k$-transitive but not $k+1$-transitive. 
Then the first step of the induction is a Goldman-type theorem for $G_{x_1,x_2,\ldots,x_{k-1}}>G_{x_1,x_2,\ldots,x_k}$, 
and we need a good assumption on the $G_{x_1,x_2,\ldots,x_{k-1}}$-action on $X\setminus\{x_1,x_2,\ldots,x_{k-1}\}$ to assure it.  
Therefore we will treat the following two cases in this paper: 
\begin{itemize}
\item[(i)] $G_{x_1,x_2,\ldots,x_{k-1}}$ is regular, 
\item[(ii)] $G_{x_1,x_2,\ldots,x_{k-1}}$ is a primitive Frobenius group. 
\end{itemize}
Permutation groups satisfying (i) are called sharply $k$-transitive, and their complete classification is known. 
Other than symmetric groups and alternating groups, the following list exhausts all of them (see \cite[Chapter XII]{HB}). 
\begin{itemize}
\item[(1)] We denote by $\F_q$ the finite field with $q$ elements. 
Every sharply 2-transitive group is either a group of transformations of the form $x\mapsto ax^\sigma+b$ of $\F_q$, 
where $a\in \F_q^\times$, $b\in \F_q$, and $\theta\in \Aut(\F_q)$, or one of the 7 exceptions. 
They are all Frobenius groups.  
\item[(2)] There exist exactly 2 infinite families of sharply 3-transitive permutation groups: 
$L(q)=PGL_2(q)$ acting on the projective geometry $PG_1(q)=(\F_q^2\setminus\{0\})/\F_q^\times$ over the finite field $\F_q$, 
and its variant $M(q)$ acting on $PG_1(q)$ with an involution of $\F_q$ when $q$ is an even power of an odd prime. 
When $q$ is odd, both of them contain $PSL_2(q)$ as an index 2 subgroup. 
\item[(3)] The Mathieu group $M_{11}$ of degree 11 is a sharply 4-transitive group, and the Mathieu group $M_{12}$ of degree 12 
is a sharply 5-transitive permutation group. 
\end{itemize}

\begin{conjecture} A Goldman-type theorem holds for every sharply $k$-transitive permutation group. 
\end{conjecture}

In Section 3, we show Goldman-type theorems for all Frobenius groups, and verify the conjecture for $k=2$ 
as a special case (Theorem \ref{GthF}).  
We also classify related fusion categories generalizing  Etingof-Gelaki-Ostrik's result \cite[Corollary 7.4]{EGO04} 
(Theorem \ref{classification}). 
We verify the conjecture for $k=3$ in Section 4 (Theorem \ref{LM}), and for $k=4$ in Section 6 (Theorem \ref{GTS5},\ref{GTA6}, \ref{GthM}). 
When $q$ is odd, the action of $PSL_2(q)$ on $PG_1(q)$ is 2-transitive and it satisfies the condition (ii) above. 
We will show a Goldman-type theorem for $PSL_2(q)$ acting on $PG_1(q)$ in Section 5 (Theorem 5.1). 

2-transitive extensions of Frobenius groups (with some condition) are called Zassenhaus groups (see \cite[Chapter XI]{HB} 
for the precise definition), and there are exactly 4 infinite families of them: $L(q)$, $M(q)$, $PSL_2(q)$ as above, and the Suzuki groups $Sz(2^{2n+1})$ 
of degree $2^{4n+2}$ for $n\geq 1$. 
One might hope that a Goldman-type theorem would hold for the Suzuki groups too. 
However, it is difficult to prove it with our technique now because the point stabilizers of the Suzuki groups are non-primitive Frobenius groups and the Frobenius kernels are non-commutative. 

\section{Preliminaries}
\subsection{Frobenius groups}
A transitive permutation group $G$ on a finite set $X$ is said to be a Frobenius group if it is not regular and every 
$g\in G\setminus \{e\}$ has at most one fixed point. 
Let $H=G_{x_1}$ be a point stabilizer.  
Then $G$ being Frobenius is equivalent to the condition that the $H$-action on $X\setminus \{x_1\}$ is free, and 
is further equivalent to the condition that $H\cap gHg^{-1}=\{e\}$ for all $g\in G\setminus H$. 

For a Frobenius group $G$, 
$$K=G\setminus \bigcup_{x\in X}G_x$$
is a normal subgroup of $G$, called the Frobenius kernel, and $G$ is a semi-direct product $K\rtimes H$ (see \cite[8.5.5]{R93}). 
The point stabilizer $H$ is called a Frobenius complement. 
Now the set $X$ is identified with $K$, and the $H$-action on $X\setminus \{x_1\}$ is identified with that on $K\setminus \{e\}$. 
It is known that $K$ is nilpotent (Thompson), and $H$ has periodic cohomology (Burnside) in the sense that  
the Sylow $p$-subgroups of $H$ are cyclic for odd $p$, and are either cyclic or generalized quaternion for $p=2$ (\cite[10.5.6]{R93}).  
We collect the following properties of Frobenius groups we will use later. 

Recall that a transitive permutation group is primitive if and only if its point stabilizer is maximal in $G$. 

\begin{lemma} \label{Fro} Let $G$ be a Frobenius group with the kernel $K$ and a complement $H$. 
Then the following hold:
\begin{itemize}
\item[$(1)$] $G$ is primitive if and only if $K$ is an elementary abelian $p$-group $\Z_p^l$ with a prime $p$ and there is no 
non-trivial $H$-invariant subgroup of $K$. 
\item[$(2)$] The Schur multiplier $H^2(H,\T)$ is trivial. 
\item[$(3)$] Every abelian subgroup of $H$ is cyclic. 
\end{itemize} 
\end{lemma}

\begin{proof} (1) Note that $G$ is primitive if and only if there is no non-trivial $H$-invariant subgroup of $K$. 
Assume that $G$ is primitive. 
Since $K$ is nilpotent, its center $Z(K)$ is not equal to $\{e\}$ and $H$-invariant, and so $K=Z(K)$. 
Let $p$ be a prime so that the $p$-component $K_p$ of $K$ is not $\{e\}$. 
Since $K_p$ is $H$-invariant, we get $K=K_p$. 
The same argument applied to $L=\{x\in K;\; x^p=0\}$ shows that $K$ is an elementary abelian $p$-group. 

(2) Since the Schur multiplier is trivial for every cyclic group and generalized quaternion 
(see for example \cite[Proposition 2.1.1, Example 2.4.8]{K}), the statement follows from  \cite[Theorem 10.3]{B}. 

(3) The statement follows from the fact that every abelian subgroup of a generalized quaternion group is cyclic. 
\end{proof}

\subsection{Sharply $k$-transitive permutation groups}
A transitive permutation group $G$ on a finite set $X$ is said to be sharply $k$-transitive permutation group if the $G$-action on 
$X^{[k]}$ is regular. 
If the degree of $G$ is $n$, a sharply $k$-permutation group has order $n(n-1)\cdots (n-k+1)$. 

For $n\in \N$, let $X_n=\{1,2,\ldots,n\}$. 
Since $X_n^{[n-1]}$ and $X_n^{[n]}$ are naturally identified, the defining action of $\fS_n$ on $X_n$ is 
both sharply $n-1$ and $n$-transitive. 
As this fact might cause confusion, we treat $\fS_n$ as a sharply $n-1$-transitive group in this paper. 
The natural action of $\fA_n$ on $X_n$ is sharply $n-2$-transitive.

Every sharply 2-transitive permutation group $G$ is known to be a Frobenius group, and hence of the form $G=\Z_p^k\rtimes H$ with a prime $p$ 
and with a Frobenius complement $H$ acting on $\Z_p^k\setminus\{0\}$ regularly.  
Let $q=p^k$, and let $T(q)=\F_q^\times \rtimes \Aut(\F_q)$, which acts on $\F_q$ as an additive group isomorphic to $(\Z/p\Z)^k$. 
Then the Zassenhaus theorem says that $H$ is either identified with a subgroup of $T(q)$ or one of the following exceptions: 
$SL_2(3)$ acting on $\Z_5^2$, $GL_2(3)$ acting on $\Z_7^2$, $SL_2(3)\times \Z_5$ acting on $\Z_{11}^2$, 
$SL_2(5)$ acting on $\Z_{11}^2$, $GL_2(3)\times \Z_{11}$ acting on $\Z_{23}^2$, 
$SL_2(5)\times \Z_7$ acting on $\Z_{29}^2$, and $SL_2(5)\times \Z_{29}$ acting on $\Z_{59}^2$. 
The reader is referred to \cite[Chapter XII, Section p]{HB} for this fact.  

There are two important families $H(q)$ and $S(q)$ of sharply 2-transitive permutation groups. 
If $G=\Z_p^k\rtimes H$ is a sharply 2-transitive group with an abelian Frobenius complement, it is necessarily of the form  
$G=\F_q\rtimes \F_q^\times$, which is denoted by $H(q)$. 
Assume now that $p$ is an odd prime and $q=p^{2l}$. 
Then the field $\F_q$ has an involution $x^\sigma=x^{p^l}$. 
The group $S(q)$ has a Frobenius complement $\F_q^\times$ as a set, but its action on $\F_q$ is given as follows:
$$a\cdot x=\left\{
\begin{array}{ll}
ax , &\quad \textrm{if } a\textrm{ is a square in }\F_q^\times , \\
ax^\sigma , &\quad\textrm{if } a\textrm{ is not a square in }\F_q^\times.
\end{array}
\right.
$$
For example, the group $S(3^2)$ is isomorphic $\Z_3^2\rtimes Q_8$.  
We have small order coincidences $\fS_3=H(3)$ and $\fA_4=H(2^2)$. 

There are exactly two families of sharply 3-transitive permutation groups $L(q)$ and $M(q)$, and they are transitive extensions of 
$H(q)$ and $S(q)$ respectively (see \cite[Chapter XI, Section 2]{HB}). 
To describe their actions, it is convenient to identify the projective geometry $PG_1(q)$ with $\F_q\sqcup \{\infty\}$. 
The 3-transitive action of $L(q)=PGL_2(q)$ is given as follows: 
$$\left[
\left(
\begin{array}{cc}
a &b  \\
c &d 
\end{array}
\right)
\right]\cdot x=\frac{ax+b}{cx+d}.$$
The group $M(q)$ is $PGL_2(q)$ as a set, but its action on $PG_1(q)$ is given by 
$$\left[
\left(
\begin{array}{cc}
a &b  \\
c &d 
\end{array}
\right)
\right]\cdot x=
\left\{
\begin{array}{ll}
\frac{ax+b}{cx+d} , &\quad\textrm{if }  ad-bc\textrm{ is a square in }\F_q^\times, \\
\frac{ax^\sigma+b}{cx^\sigma+d} , &\quad\quad\textrm{if } ad-bc\textrm{ is not a square in }\F_q^\times.
\end{array}
\right.
$$
We have small order coincidences $\fS_4=L(3)$ and $\fA_5=L(2^2)$. 

When $q$ is odd, the restriction of the $L(q)$-action on $PG_1(q)$ to $PSL_2(q)$ is two transitive, and its point stabilizer 
is isomorphic to $\Z_p^k\rtimes \Z_{(p^k-1)/2}$. 

Other than symmetric groups and alternating groups, the Mathieu groups $M_{11}$ and $M_{12}$ are the only sharply 
4 and 5-transitive permutation groups, and their degrees are 11 and 12 respectively (see \cite[Chapter XII, Section 3]{HB}). 
To show a Goldman-type theorem for the permutation group $M_{11}$ of degree 11, we do not really need its construction. 
Instead, we only need the fact that this action is a transitive extension of the sharply 3-transitive permutation group $M(3^2)$ 
on $PG_1(3^2)$ (see \cite[Chapter XII, Theorem 1.3]{HB}).

\subsection{Group-subgroup subfactors}
For a finite index inclusion $M\supset N$ of factors, we need to distinguish the two principal graphs of it, and symbols for them. 
Thus we mean by the principal graph of $M\supset N$ the induction-reduction graph between $N$-$N$ bimodules and $M$-$N$ bimodules 
arising from the inclusion, and denote it by $\cG_{M\supset N}$,  
while we mean by the dual principal graph the induction-reduction graph between $M$-$M$ bimodules and $M$-$N$ bimodules, 
and denote it by $\cG_{M\supset N}^d$. 
 
Let $G$ be a transitive permutation group on a finite set $X$, and let $H=G_{x_1}$ with $x_1\in X$. 
Let 
$$M=R\rtimes_\alpha G\supset N=R\rtimes_\alpha H,$$
be a group subgroup subfactor with an outer $G$-action on a factor $R$. 
The reader is referred to \cite{KY92} for the tensor category structure of the $M$-$M$, $M$-$N$, $N$-$M$, and $N$-$N$ bimodules 
arising from the group-subgroup subfactor $M\supset N$. 
The category of $M$-$M$ bimodules is equivalent to the representation category $\Rep(G)$ of $G$, and we use the symbol 
$\widehat{G}$ to parametrize the equivalence classes of irreducible $M$-$M$ bimodules. 
The set of equivalence classes of irreducible $M$-$N$ bimodules are parametrized by $\widehat{H}$, 
and $\cG_{G>H}^d$ is the induction-reduction graph between $\widehat{G}$ and $\widehat{H}$. 
For this reason, we denote by $\cG_H^G$ the dual principal graph $\cG_{M\supset N}^d$. 

The description of the category of $N$-$N$ bimodules is much more involved.  
We choose one point from each $G_{x_1}$-orbit in $X\setminus \{x_1\}$, and enumerate them as $x_2,x_3,\ldots,x_k$. 
Then the set of the equivalence classes of irreducible $N$-$N$ bimodules arising from $M\supset N$ 
is parametrized by the disjoint union 
$$\widehat{G_{x_1}}\sqcup \widehat{G_{x_1,x_2}}\sqcup \cdots\sqcup \widehat{G_{x_1,x_k}},$$
and the graph $\cG_{M\supset N}$ is the union of the induction-reduction graph between $\widehat{G_{x_1}}$ and 
$\widehat{G_{x_1,x_i}}$ over $1\leq i\leq k$ 
with convention $G_{x_1,x_1}=G_{x_1}$. 
The dimension of the irreducible object corresponding to $\pi\in \widehat{G_{x_1,x_2}}$ is $|G_{x_1}/G_{x_1,x_2}|\dim \pi$. 
We denote by $\cG_{(G,X)}$ or $\cG_{G>G_{x_1}}$ the principal graph $\cG_{M\supset N}$ depending on the situation.  

The category of $N$-$N$ bimodules for the inclusion $N\supset R$ is equivalent to $\Rep(H)$, and we denote the equivalence classes 
of irreducible objects of it by $\{[\beta_{\pi}]\}_{\pi\in \widehat{H}}$. 
Then the set $\{[\beta_{\pi}]\}_{\pi\in \widehat{H}}$ actually coincides with $\widehat{H}$ in $\cG_{G>H}$ 
as equivalence classes of $N$-$N$ bimodules (this fact is not usually emphasized but one can see it from \cite{KY92}). 
Let $\iota={}_MM_N$ be the basic bimodule. 
Then the set of equivalence classes of irreducible $M$-$N$ bimodules arising from $M\supset N$ is given by 
$\{[\iota\otimes_N\beta_\pi] \}_{\pi \in \widehat{H}}$. 

If $G$ is 2-transitive, we have $k=2$, and the graph $\cG_{(G,X)}$ can be obtained from $\cG^{G_{x_1}}_{G_{x_1,x_2}}$ by 
putting an edge of length one to each even vertex of $\cG^{G_{x_1}}_{G_{x_1,x_2}}$. 
More generally, for a bipartite graph $\cG$, we denote by $\widetilde{\cG}$ the graph obtained by putting 
an edge of length one to each even vertex of $\cG$. 
Then we have $\cG_{(G,X)}=\widetilde{\cG^{G_{x_1}}_{G_{x_1,x_2}}}$.

Let $\cG_n$ be a depth 2 graph without multi-edges and with $n$ even vertices. 
Assume that $\cG_n$ is the principal graph $\cG_{M\supset N}$ of a finite index inclusion $M\supset N$ of factors. 
Then the characterization of crossed products shows that $M=N\rtimes_\alpha G$, 
and the $G$-action is unique up to inner conjugacy. 
Thus a Goldman-type theorem holds for regular permutation groups, but in a weak sense because 
the graph $\cG_n$ determines only the order $n$ of $G$, and not the group structure unless $n$ is a prime.  
Even when we specify the dual principal graph of $M\supset N$, it does not distinguish the dihedral group 
$\fD_8$ of order 8 and the quaternion group $Q_8$.  
As this example suggests, we should clarify what we really mean by a Goldman-type theorem. 
\begin{figure}[H]
\centering
\begin{tikzpicture}
\draw (0,0)--(-1.2,1)(0,0)--(-0.4,1)(0,0)--(1.2,1);
\draw(-1.2,1)node{$*$};
\draw(-1.2,1)node[above]{1};
\draw(-0.4,1)node{$\bullet$};
\draw(-0.4,1)node[above]{2};
\draw(1.2,1)node{$\bullet$};
\draw(1.2,1)node[above]{$n$};
\draw(0,0)node{$\bullet$};
\draw(0.4,1)node{$\cdots$};
\end{tikzpicture}
\caption{$\cG_n$}
\end{figure}

\begin{definition} Let $\cG$ be a bipartite graph. 
\begin{itemize}
\item[(1)] We say that a strong Goldman-type theorem for $\cG$ 
(or for $(G,X)$ if $\cG=\cG_{(G,X)}$) if the following holds: 
there exists a unique transitive permutation group $G$ on a finite set, up to permutation conjugacy, 
such that whenever the principal graph of a finite index subfactor $M\supset N$ is $\cG$, 
there exists a unique subfactor $R$ of $N$, up to inner conjugacy in $N$, satisfying $M\cap R'=\C$ and 
$$M=R\rtimes_\alpha G\supset N=R\rtimes_\alpha H,$$
where $H$ is a point stabilizer of $G$. 
\item[(2)] We say that a weak Goldman-type theorem for $\cG$ if the following holds: 
whenever the principal graph of a finite index subfactor $M\supset N$ is $\cG$, 
there exists a unique subfactor $R$ of $N$, up to inner conjugacy in $N$, satisfying $M\cap R'=\C$ and 
$$M=R\rtimes_\alpha G\supset N=R\rtimes_\alpha H,$$
for some transitive permutation group $G$ on a finite set with a point stabilizer of $H$. 
\end{itemize}
\end{definition}
Note that the action $\alpha$ is automatically unique, up to inner conjugacy, thanks to the irreducibility of $R$ in $M$.  

We will show weak Goldman-type theorems for all Frobenius groups (including sharply 2-permutation groups), 
and strong ones for sharply 3 and 4-permutation groups and for $PSL_2(q)$ acting on $PG_1(q)$. 

\subsection{Intermediate subfactors}
In what follows, we use the sector notation for subfactors (see \cite[Section 2]{I98} or \cite[Subsction 2.1]{I18} for example), 
though all results are stated for general factors. 
The inclusion map $\iota:N \hookrightarrow M$ in the statements should be read as the basic bimodule $\iota={}_MM_N$ 
in the type II$_1$ case.  
In the proofs, we always assume that factors involved are either of type II$_\infty$ or type III without mentioning it. 
In the type II$_1$ case, this can be justified by either directly working on bimodules instead of sectors, 
or replacing $M\supset N$ with $M\otimes B(\ell^2)\supset N\otimes B(\ell^2)$. 
For example, assume that a statement insists existence of a subfactor $P\subset N$ with a certain property.  
In the latter case, after finding an appropriate subfactor $P\subset N\otimes B(\ell^2)$, we can pass to the corners 
$(1\otimes e)P(1\otimes e)\subset N\otimes \C e$ and the original statement can be recovered, 
where $e\in B(\ell^2)$ is a minimal projection  
(we may always assume $1\otimes e\in P$ up to inner conjugacy in $N$). 

We collect useful statements for our purpose in the next theorem concerning intermediate subfactors extracted from 
\cite[Corollary 3.10]{ILP98}. 
\ref{intermediate}
\begin{theorem}\label{intermediate} 
Let $M\supset N$ be an irreducible inclusion of factors with finite index, and 
let $\iota:N\hookrightarrow M$ be the inclusion map. 
Let 
$$[\iota\biota]=\bigoplus_{\xi\in \Lambda}n_\xi[\xi]$$
be the irreducible decomposition. 
\begin{itemize}
\item[$(1)$] Let $P$ be an intermediate subfactor between $M$ and $N$, and let $\kappa:P\hookrightarrow M$ be the inclusion map. 
If $\xi_1,\xi_2\in \Lambda$ are contained in $\kappa\bkappa$, and $\xi_3\in \Lambda$ is contained in $\xi_1\xi_2$, 
then $\xi_3$ is contained in $\kappa\bkappa$. 

\item[$(2)$] Assume that $P$ and $Q$  are intermediate subfactors between $M$ and $N$, and the inclusion maps  
$\kappa:P\hookrightarrow M$ and $\kappa_1:P_1\hookrightarrow M$ satisfy $[\kappa\bkappa]=[\kappa_1\bar{\kappa_1}]$. 
If for each $\xi\in \Lambda$ the multiplicity of $\xi$ in $\kappa\bkappa$ is either 0 or $n_\xi$, then $P=Q$. 

\item[$(3)$] Assume that $\Lambda_1$ is self-conjugate subset of $\Lambda$ such that whenever $\xi_3\in \Lambda$ 
is contained in $\xi_1\xi_2$ for some $\xi_1,\xi_2\in \Lambda_1$, we have $\xi_3\in \Lambda_1$. 
Then there exists a unique intermediate subfactor $P$ between $M$ and $N$ such that the inclusion map $\kappa :P\hookrightarrow M$ satisfies 
$$[\kappa\bkappa]=\bigoplus_{\xi\in \Lambda_1}n_\xi[\xi].$$
\end{itemize}
\end{theorem}

\subsection{The strategy of the proofs}

Let $\Gamma$ be a doubly transitive permutation group acting on a finite set $X$, and let $x_1,x_2\in X$ be distinct points. 
We further assume that the $\Gamma_{x_1,x_2}$-action on $X\setminus\{x_1,x_2\}$ has no orbit of length 1. 
Our basic strategy to prove a Goldman-type theorem for $\Gamma>\Gamma_{x_1}$ is to reduce it to that of 
$\Gamma_{x_1}>\Gamma_{x_1,x_2}$. 
To explain it, we first discuss the relationship between the group-subgroup subfactor of the former 
and that of the latter. 
We denote $G=\Gamma_{x_1}$ and $H=\Gamma_{x_1,x_2}$ for simplicity.

Assume that we are given an outer action $\alpha$ of $\Gamma$ on a factor $R$. 
We set $N=R\rtimes_\alpha H$, $M=R\rtimes_\alpha G$, and $L=R\rtimes_\alpha \Gamma$.  
We denote by $\iota_1:M\hookrightarrow L$, $\iota_2:N\hookrightarrow M$, and $\iota_3: R\hookrightarrow N$ 
the inclusion maps. 
Since the $\Gamma$-action on $X$ is doubly transitive, there exists $g_0\in \Gamma$ exchanging $x_1$ and $x_2$. 
Such $g_0$ normalizes $H$, and we get $\theta\in \Aut(N)$ extending $\alpha_{g_0}$, 
that is $\theta\iota_3=\iota_3\alpha_{g_0}$. 
Let 
$$[\iota_3\bar{\iota_3}]=\bigoplus_{\pi\in \hat{H}}d(\pi)[\beta_\pi]$$ 
be the irreducible decomposition. 
The automorphism $\theta$ as above is not unique, and there is always a freedom to replace $\theta$ with $\theta\beta_\pi$ 
with $d(\pi)=1$.  

Since 
$$[\iota_1\iota_2\theta\iota_3]=[\iota_1\iota_2\iota_3\alpha_{g_0}]=[\iota_1\iota_2\iota_3],$$
we have 
$$1=\dim(\iota_1\iota_2\theta\iota_3,\iota_1\iota_2\iota_3)=(\iota_2\theta\iota_3\bar{\iota_3}\bar{\iota_2},\bar{\iota_1}\iota_1)
=\sum_{\pi\in\widehat{H}}d(\pi)\dim(\iota_2\theta\beta_\pi\bar{\iota_2},\bar{\iota_1}\iota_1).$$
We claim $(\iota_2\theta\beta_\pi\bar{\iota_2},\id)=0$ for all $\pi$. 
Indeed, if it were not the case, we would have $\pi$ with $d(\beta\pi)=1$ satisfying  $[\iota_2\theta\beta_\pi]=[\iota_2]$ 
thanks to the Frobenius reciprocity. 
However, this implies that $\theta\beta_\pi$ would be contained in $\overline{\iota_2}\iota_2$. 
Since $d(\theta\beta_\pi)=1$, this contradicts the assumption that the $H$-action on $G/H\setminus H$ has no 
orbit of length 1. 

Since $\Gamma$ is doubly transitive, there exists irreducible $\tau$ with $d(\tau)=|X|-1$ satisfying 
$[\bar{\iota_1}\iota_1]=[\id]\oplus [\tau]$. 
On the other hand, we have $d(\iota_2\theta\beta_\pi\bar{\iota_2})=(|X|-1)d(\pi)$, which shows that 
there exists $\pi\in \widehat{H}$ with $d(\pi)=1$ satisfying $[\tau]=[\iota_2\theta\beta_\pi\bar{\iota_2}]$. 
This means that by replacing $\theta$ with $\theta\beta_\pi$ if necessary, we may always assume 
$$[\bar{\iota_1}\iota_1]=[\id]\oplus [\iota_2\theta\bar{\iota_2}].$$

Now forget about $R$, $\alpha$, $N$, and assume that we are just given an inclusion $L\supset M$ with 
$\cG_{L\supset M}=\cG_{\Gamma>G}$. 
We denote by $\iota_1:M\hookrightarrow L$ the inclusion map. 
We assume that a Goldman-type theorem is known for $G>H$. 
Our task is to recover $R$ and $\alpha$ from the inclusion $L\supset M$. 
Our strategy is divided into the following steps:

\begin{itemize}
\item[(1)] Find a fusion subcategory $\cC_1$ in the fusion category $\cC$ generated by $\bar{\iota_1}\iota_1$ 
that looks like the representation category of $G$.  
. 
\item[(2)] Show that the object in $\cC_1$ corresponding to the induced representation $\Ind_H^G1$ has a unique $Q$-system 
satisfying the following condition: 
if $N\subset M$ is the subfactor corresponding to the $Q$-system and $\iota_2:N\hookrightarrow M$ is the inclusion map, 
then there exists $\theta \in \Aut(N)$ satisfying 
$$[\bar{\iota_1}\iota_1]=[\id]\oplus [\iota_2\theta\bar{\iota_2}].$$
\item[(3)] Show $\cG_{M\supset N}=\cG_{G>H}$. 
\item[(4)] Apply the Goldman-type theorem for $G>H$ to $M\supset N$, and obtain a subfactor $R$ and an outer action 
$\gamma$ of $G$ on $R\subset N$ satisfying $M=R\rtimes_\gamma G$ and $N=R\rtimes_\gamma H$. 
Show that $R$ is irreducible in $L$. 
Let $\iota_3:R\hookrightarrow N$ be the inclusion map. 
\item[(5)] Show that $L\supset R$ is a depth 2 inclusion. 
\item[(6)] Show that there exists $\theta_1\in \Aut(R)$ satisfying $[\theta\iota_3]=[\iota_3\theta_1]$. 
\end{itemize}

\begin{lemma}\label{finishing} Assume that the above (1)-(6) are accomplished. 
Then there exist a finite group $\Gamma_0$ including $G$ as a subgroup of index $|X|$, and an outer action $\alpha$ 
of $\Gamma_0$ on $R$ such that $\alpha$ is an extension of $\gamma$ and $L=R\rtimes_\alpha \Gamma_0$. 
Moreover, the action of $\Gamma_0$ on $\Gamma_0/G$ is a doubly transitive extension of the $G$-action on 
$X\setminus \{x_0\}$.  
\end{lemma}

\begin{proof} By (2),
$$[\bar{\iota_3}\bar{\iota_2}\bar{\iota_1}\iota_1\iota_2\iota_3]
=[\bar{\iota_3}\bar{\iota_2}(\id\oplus \iota_2\theta\bar{\iota_2})\iota_2\iota_3]
=\bigoplus _{g\in G} [\gamma_g]\oplus [\bar{\iota_3}\bar{\iota_2}\iota_2\theta\bar{\iota_2}\iota_2\iota_3],$$
which contains 
$$\bigoplus _{g\in G} [\gamma_g]\oplus [\bar{\iota_3}\theta\iota_3]
=\bigoplus _{g\in G} [\gamma_g]\oplus [\bar{\iota_3}\iota_3\theta_1]
=\bigoplus _{g\in G} [\gamma_g]\oplus \bigoplus_{h\in H}[\gamma_h\theta_1].$$
by (6). 
Let $\Gamma_0$ be the group of 1-dimensional sectors contained in 
$[\bar{\iota_3}\bar{\iota_2}\bar{\iota_1}\iota_1\iota_2\iota_3]$. 
Then $\Gamma_0$ is strictly larger than $[\gamma_G]$, and $R\rtimes \Gamma_0$ is a subfactor of $L$ strictly larger than $M$. 
Thanks to Theorem \ref{intermediate}, there is no non-trivial intermediate subfactor 
between $L$ and $M$, and we conclude $L=R\rtimes \Gamma_0$. 
From the shape of the graph $\cG_{\Gamma>G}$, we can see that the $\Gamma_0$-action on $\Gamma_0/G$ is doubly transitive. 
\end{proof}

To identify $\Gamma_0$ with $\Gamma$, we will use the classification of doubly transitive permutation groups. 

In concrete examples treated in this paper, (1) and (3) are purely combinatorial arguments,  
(2) follows from Theorem \ref{intermediate}, (4) is an induction hypothesis, and 
(5) is a simple computation of dimensions.  
To deal with (6), we give useful criteria now. 

\begin{lemma}\label{inv1} Let $G$ be a transitive permutation group on a finite set with a point stabilizer $H$, and 
let $\alpha$ be an outer action of $G$ on a factor $R$.
Let $M=R\rtimes_\alpha G\supset N=R\rtimes_\alpha H$. 
Let $L$ be a factor including $M$ as an irreducible subfactor of index $|G/H|+1$. 
We denote by $\iota_1:M\hookrightarrow L$, $\iota_2:N\hookrightarrow M$, and $\iota_3: R\hookrightarrow N$ 
the inclusion maps. 
We assume the following two conditions: 
\begin{itemize}
\item[$(1)$] The inclusion $L\supset R$ is irreducible and of depth 2. 
\item[$(2)$] There exists $\theta\in \Aut(N)$ satisfying $[\overline{\iota_1}\iota_1]=[\id]\oplus [\iota_2\theta\bar{\iota_2}]$  
\end{itemize}
Then we have 
$$\dim(\theta\bar{\iota_2}\iota_2\iota_3\bar{\iota_3}\theta^{-1},\bar{\iota_2}\iota_2\iota_3\bar{\iota_3})=|H|.$$
\end{lemma}

\begin{proof} Since $[L:M]=(|G/H|+1)|G|$, the depth 2 condition implies  
\begin{align*}
(|G/H|+1)|G|&=\dim (\bar{\iota_3}\bar{\iota_2}\bar{\iota_1}\iota_1\iota_2\iota_3,\bar{\iota_3}\bar{\iota_2}\bar{\iota_1}\iota_1\iota_2\iota_3) \\
 &=\dim (\bar{\iota_3}\bar{\iota_2}(\id\oplus \iota_2\theta\bar{\iota_2})\iota_2\iota_3,\bar{\iota_3}\bar{\iota_2}(\id\oplus \iota_2\theta\bar{\iota_2}) \iota_2\iota_3) \\
 &=\dim(\bigoplus_{g\in G} \alpha_g\oplus \bar{\iota_3}\bar{\iota_2}\iota_2\theta\bar{\iota_2}\iota_2\iota_3,\bigoplus_{g\in G} \alpha_g\oplus \bar{\iota_3}\bar{\iota_2}\iota_2\theta\bar{\iota_2}\iota_2\iota_3)\\
 &=|G|+\dim(\bar{\iota_3}\bar{\iota_2}\iota_2\theta\bar{\iota_2}\iota_2\iota_3, \bar{\iota_3}\bar{\iota_2}\iota_2\theta\bar{\iota_2}\iota_2\iota_3),
\end{align*}
and 
$$|G/H||G|=\dim(\bar{\iota_3}\bar{\iota_2}\iota_2\theta\bar{\iota_2}\iota_2\iota_3, \bar{\iota_3}\bar{\iota_2}\iota_2\theta\bar{\iota_2}\iota_2\iota_3)
 =\dim(\theta\bar{\iota_2}\iota_2\iota_3\bar{\iota_3}\bar{\iota_2}\iota_2\theta^{-1},\bar{\iota_2}\iota_2\iota_3\bar{\iota_3}\bar{\iota_2}\iota_2),$$
by the Frobenius reciprocity. 
Thus to prove the statement, it suffices to show 
$$[\bar{\iota_2}\iota_2\iota_3\bar{\iota_3}\bar{\iota_2}\iota_2]=|G/H|[\bar{\iota_2}\iota_2\iota_3\bar{\iota_3}].$$
Indeed, note that $\iota_2\iota_3\bar{\iota_3}\bar{\iota_2}$ is an $M$-$M$ sector corresponding to the regular representation of $G$,  
and hence $(\iota_2\iota_3\bar{\iota_3}\bar{\iota_2})\iota_2$ is an $M$-$N$ sector corresponding to the restriction of the regular 
representation of $G$ to $H$, which is equivalent to $|G/H|$ copies of the regular representation of $H$. 
Since $\iota_3\bar{\iota_3}$ is an $N$-$N$ sector corresponding the regular representation of $H$, we get 
$$[(\iota_2\iota_3\bar{\iota_3}\bar{\iota_2})\iota_2]=|G/H|[\iota_2(\bar{\iota_3}\iota_3)],$$
which finishes the proof. 
\end{proof}

In concrete cases where Lemma \ref{inv1} is applied, we can further show 
$$\dim(\theta\iota_3\bar{\iota_3}\theta^{-1},\iota_3\bar{\iota_3})=|H|,$$
resulting in $[\theta\iota_3\bar{\iota_3}\theta^{-1}]=[\iota_3\bar{\iota_3}]$. 

From Theorem 3.3 and Lemma 4.1 in \cite{IK02}, we can show the following global invariance criterion.  

\begin{lemma}\label{inv2} Let $H$ be a finite group and let $\alpha$ be an outer action of $H$ on a factor $R$. 
Let $N=R\rtimes_\alpha H$, and let $\iota:R\hookrightarrow N$ be the inclusion map. 
We assume that there is no non-trivial abelian normal subgroup $K\triangleleft H$ with a non-degenerate 
cohomology class $\omega\in H^2(\widehat{K},\T)$ invariant under the $H$-action by conjugation. 
If $\theta\in \Aut(N)$ satisfies $[\theta\iota\biota\theta^{-1}]=[\iota\biota]$, there exists 
$\theta_1\in \Aut(R)$ satisfying $[\theta\iota]=[\iota\theta_1]$. 
\end{lemma}

Even when the cohomological assumption in Lemma \ref{inv2} is not fulfilled, we still have a chance to apply the following criterion. 
For an inclusion $N\supset R$ of factors, we denote by $\Aut(N,R)$ the set of automorphisms of $N$ globally preserving $R$. 

\begin{lemma}\label{inv3} 
Let $N\supset R$ be an irreducible inclusion of factors with finite index, and let $P$ be an intermediate subfactor between $N$ and $R$. 
We denote by $\iota:R\hookrightarrow N$ and $\kappa:P\hookrightarrow N$ the inclusion maps. 
Let  
$$[\iota\biota]=\bigoplus_{\xi\in \Lambda}n_\xi[\xi]$$
be the irreducible decomposition. 
We assume that for each $\xi\in \Lambda$ the multiplicity of $\xi$ in $\kappa\bkappa$ is either 0 or $n_\xi$. 
If $\theta\in \Aut(N,R)$ satisfies $[\theta\kappa\bkappa\theta^{-1}]=[\kappa\bkappa]$, then $\theta(P)=P$. 
\end{lemma}

\begin{proof} Let $Q=\theta(P)$, let $\varphi:P\to Q$ be the restriction of $\theta$ to $P$ regarded as an isomorphism from $P$ onto $Q$, 
and let $\kappa_1:Q\hookrightarrow N$ be the inclusion map. 
Then by definition, we have $\theta\circ \kappa=\kappa_1\circ \varphi$. 
Thus 
$$[\kappa_1\bar{\kappa_1}]=[\kappa_1\varphi\bar{\varphi}\bar{\kappa_1}]=[\theta][\kappa\bkappa][\theta^{-1}]=[\kappa\bkappa],$$
and the statement follows from Theorem \ref{intermediate}.  
\end{proof}


\section{Goldman-type theorems for Frobenius groups}
In this section, we establish weak Goldman-type theorems for all Frobenius groups generalizing results obtained in \cite{I97}. 

For a tuple of natural numbers $\bm=(m_0,m_1,\ldots,m_l)$ with $m_0=1$ and $l\geq 1$, and a natural number $n$, 
we assign a bipartite graph $\cG_{\bm,n}$ as follows. 
Let $I=\{0,1,\ldots,l\}$ and let $J$ be an index set with $|J|=n$. 
The set of even vertices is $\{v^0_i\}_{i\in I}\sqcup \{v^2_j\}_{j\in J}$ and the set of odd vertices is $\{v^1_i\}_{i\in I}$.  
The only non-zero entries of the adjacency matrix $\Delta$ of $\cG_{\bm,n}$ are 
$$\Delta(v^0_i,v^1_i)=\Delta(v^1_i,v^0_i)=1,\quad \forall i\in I,$$
$$\Delta(v^{1}_i,v^2_j)=\Delta(v^2_j,v^{1}_i)=m_i,\quad \forall i\in I,\;\forall j\in J.$$
The vertex $v_0^0$ is treated as a distinguished vertex $*$. 

\begin{figure}[H]
\centering
\begin{tikzpicture}
\draw (-2,0)--(-1,0)--(0,0);
\draw (0,0)--(1,1)--(2,2);
\draw (0,0)--(1,0.3)--(2,0.6);
\draw (0,0)--(1,-0.3)--(2,-0.6);
\draw (0,0)--(1,-1)--(2,-2);
\draw(-2,0)node{$*$};
\draw(-1,0)node{$\bullet$};
\draw(0,0)node{$\bullet$};
\draw(1,1)node{$\bullet$};
\draw(1,0.3)node{$\bullet$};
\draw(1,-0.3)node{$\bullet$};
\draw(1,-1)node{$\bullet$};
\draw(2,2)node{$\bullet$};
\draw(2,0.6)node{$\bullet$};
\draw(2,-0.6)node{$\bullet$};
\draw(2,-2)node{$\bullet$};
\draw(0.4,-0.6)node{2};
\end{tikzpicture}
\caption{$\cG_{(1^4,2),1}=\cG_{S(3^2)>Q_8}$}
\end{figure}
We use notation $k^a=\overbrace{k,k,\ldots,k}^{a}$ for short. 
With this convention, the graph $\cG_{m,n}$ considered in \cite{I97} is $\cG_{(1^m),n}$. 
An edge with a number $b$ means a multi-edge with multiplicity $b$. 

Let 
$$m:=\|\bm\|^2=\sum_{i=0}^lm_i^2.$$
Then the Perron-Frobenius eigenvalue of $\Delta$ is $\sqrt{1+mn}$. 
The Perron-Frobenius eigenvector $d$ with normalization $d(v_0^0)=1$ is 
$$d(v_i^0)=m_i,\quad d(v_i^1)=m_i\sqrt{1+mn},\quad d(v_j^2)=m.$$

Let $G=K\rtimes H$ be a Frobenius group with the Frobenius kernel $K$ and a Frobenius complement $H$. 
Then we have $\cG_{G>H}=\cG_{\bm,n}$ where $n$ is the number of $H$-orbits in $K\setminus\{e\}$, 
and $\bm$ is the ranks of the irreducible representations of $H$. 
Therefore we have $|H|=m$ and $|K|=1+mn$. 
If moreover $K$ is abelian, the graph $\cG_{H}^G$ is also $\cG_{\bm,n}$. 

Conversely, we can show the following theorem.

\begin{theorem}\label{GthF} Let $N\supset P$ be a finite index inclusion of factors with $\cG_{N\supset P}=\cG_{\bm,n}$. 
Then there exists a unique subfactor $R\subset P$, up to inner conjugacy, such that $N\cap R'=\C$ and there exists 
a Frobenius group $G=K\rtimes H$ with the Frobenius kernel $K$ and a Frobenius complement $H$ satisfying $|K|=1+mn$, 
$|H|=m$, the tuple $(m_0,m_1,\ldots,m_l)$ being the ranks of the irreducible representations of $H$, and 
$$N=R\rtimes G\supset P=R\rtimes H.$$ 
Moreover, 
\begin{itemize}
\item[$(1)$] If $n=1$, then $1+m$ is a prime power $p^k$ with a prime $p$ and  $K=\Z_p^k$. 
The $G$-action on $G/H$ is sharply 2-transitive. 
The dual principal graph is also $\cG_{\bm,1}$ in this case.  
\item[$(2)$] If $n=2$ or $n=3$, then $1+mn$ is a prime power $p^k$ with a prime $p$, and 
$G$ is a primitive Frobenius group with $K=\Z_p^k$. 
The dual principal graph is also $\cG_{\bm,n}$ in this case. 
\end{itemize}
\end{theorem}

We prove the theorem in several steps.  
Let 
$\iota:P\hookrightarrow N$ be the inclusion map. 
We denote by $\alpha_i$ the irreducible endomorphism of $N$ corresponding to $v_i^0$, and by 
$\rho_j$ the ones corresponding to $v_j^2$.  
Then $\iota\circ \alpha_i$ corresponds to $v_i^1$. 
From the graph $\cG_{\bm,n}$, we get the following fusion rules:
$$[\biota][\iota]=[\id]\oplus \bigoplus_{j\in J}[\rho_j],$$
$$[\iota][\rho_j]=\bigoplus_{i\in I}m_i[\iota\alpha_i],$$   
$$[\biota][\iota\alpha_i]=[\alpha_i]+m_i\bigoplus_{j\in J}[\rho_j],$$
$$d(\alpha_i)=m_i,\quad d(\iota)=\sqrt{1+mn},\quad d(\rho_j)=m.$$
\begin{figure}[H]
\centering
\begin{tikzpicture}
\draw (-1,2)--(-1,1)--(-0.5,0)--(0,1)--(0.5,0)--(1,1)--(-0.5,0);
\draw (-1,1)--(0.5,0);
\draw (0,2)--(0,1);
\draw (1,2)--(1,1);
\draw(-1,2)node{$*$};
\draw(-1,1)node{$\bullet$};
\draw(-0.5,0)node{$\bullet$};
\draw(0,2)node{$\bullet$};
\draw(0,1)node{$\bullet$};
\draw(0.5,0)node{$\bullet$};
\draw(1,1)node{$\bullet$};
\draw(1,2)node{$\bullet$};
\draw(-1,2)node[left]{$\id_P$};
\draw(0,2)node[left]{$\alpha_1$};
\draw(1,2)node[right]{$\alpha_2$};
\draw(-1,1)node[left]{$\iota$};
\draw(0,1)node[left]{$\iota\alpha_1$};
\draw(1,1)node[right]{$\iota\alpha_2$};
\draw(-0.5,0)node[below]{$\rho_1$};
\draw(0.5,0)node[below]{$\rho_2$};
\end{tikzpicture}
\caption{$\cG_{(1^3),2}=\cG_{\Z_7\rtimes \Z_3>\Z_3}$}
\end{figure}Let $\cC$ be the fusion category generated by $\biota\iota$. 
Then since $d(\alpha_{i_1}\alpha_{i_2})$ is smaller than $m=d(\rho_j)$, 
we have a fusion subcategory $\cC_0$ with the set (of equivalence classes) of simple objects 
$\Irr(\cC_0)=\{\alpha_i\}_{i\in I}$. 

We introduce involutions of $I$ and $J$ by $[\overline{\alpha_i}]=[\alpha_{\bar{i}}]$ and 
$[\overline{\rho_j}]=[\rho_{\overline{j}}]$. 
Note that $\rho_j\rho_{\bar{j}}$ contains $\alpha_i$ at most $d(\alpha_i)=m_i$ times (see \cite[p.39]{ILP98}). 
Since it contains $\id$, dimension counting shows that it contains $\alpha_i$ with full multiplicity $m_i$. 
Thus the Frobenius reciprocity implies 
$$[\alpha_i\rho_j]=m_i[\rho_j].$$

\begin{lemma}\label{factorization} Let the notation be as above. 
There exist a unique intermediate subfactor $P\supset R_j\supset \rho_j(P)$ and an isomorphism 
$\theta_j:R_{\bar{j}}\to R_j$ for each $j\in J$ 
such that if $\kappa_j:R_j \hookrightarrow P$ is the inclusion map, 
$$[\rho_j]=[\kappa_j\theta_j\overline{\kappa_{\bar{j}}}],$$ 
$$[\kappa_j\overline{\kappa_j}]=\bigoplus_{i\in I}m_i[\alpha_i].$$
Moreover $P\supset R_j$ is a depth 2 inclusion of index $m$.  
\end{lemma}

\begin{proof}
Theorem \ref{intermediate} shows that there exists a unique intermediate subfactor $P\supset R_j\supset \rho_j(P)$ 
such that if $\kappa_j:R_j\hookrightarrow P$ is the inclusion map, we have 
$$[\kappa_j\overline{\kappa_j}]=\bigoplus_{i\in I} m_i[\alpha_i].$$
Since $m_i=d(\alpha_i)$, Frobenius reciprocity implies 
$$[\alpha_i][\kappa_i]=m_i[\kappa_i],$$
and $P\supset R_j$ is a depth 2 inclusion of index $m$. 

Let $\sigma_j$ be $\rho_j$ regarded as a map from $P$ to $R_j$.  
By definition, we have  $\rho_j=\kappa_j\circ \sigma_j$, and since $d(\rho_j)=m$ and $d(\kappa_j)=\sqrt{m}$, 
we get $d(\sigma_j)=\sqrt{m}$. 
Taking conjugate, we get $[\rho_{\bar{j}}]=[\overline{\sigma_j}][\overline{\kappa_j}]$. 
Perturbing $\overline{\sigma_j}$ by an inner automorphism if necessary, we may and do assume 
$\rho_{\bar{j}}=\overline{\sigma_j}\circ \overline{\kappa_j}$.  
Since $[\overline{\sigma_j}\sigma_j]$ contains $\id$ and is contained in $\rho_{\bar{j}}\rho_j$, 
dimension counting shows 
$$[\overline{\sigma_j}\sigma_j]=\bigoplus_{i\in  I}m_i[\alpha_i],$$
and Theorem \ref{intermediate} implies $\overline{\sigma_j}(R_j)=R_{\bar{j}}$. 
Let $\theta_j$ be the inverse of $\overline{\sigma_j}$, which is an isomorphism from $R_{\bar{j}}$ onto $R_j$. 
Then we get $\rho_{\bar{j}}=\kappa_{\bar{j}}\circ \theta_j^{-1}\circ \overline{\kappa_j}$, and 
$$[\rho_j]=[\kappa_j\theta_j\overline{\kappa_{\bar{j}}}].$$ 
\end{proof}

\begin{lemma} \label{1dim}
With the above notation $\bar{\kappa}_j\kappa_k$ is decomposed into 1-dimensional sectors for all $j,k\in J$. 
\end{lemma} 

\begin{proof} 
Let 
$$[\overline{\kappa_j}\kappa_k]=\bigoplus_{a\in \Lambda_{j,k}}n_{jk}^a[\xi_{jk}^a]$$
be the irreducible decomposition. 
Since 
$$[\overline{\kappa_j}\kappa_k\overline{\kappa_k}\kappa_l]=\bigoplus_{i\in I}m_i[\overline{\kappa_j}\alpha_i\kappa_l]
=\bigoplus_{i\in I}m_i^2[\overline{\kappa_j}\kappa_l]
=m[\overline{\kappa_j}\kappa_l],$$
the product $\xi_{jk}^a\xi_{kl}^b$ is a direct sum of irreducibles from $\{\xi_{jl}^c\}_{c\in \Lambda_{j,l}}$. 
Since $[\overline{\overline{\kappa_k}\kappa_{j}}]=[\overline{\kappa_j}\kappa_k]$, we can arrange the index sets so that 
for any $a\in \Lambda_{j,k}$ there exists $\overline{a}\in \Lambda_{k,j}$ satisfying $[\overline{\xi_{jk}^a}]=[\xi_{kj}^{\overline{a}}]$. 

Since 
$$\delta_{j,k}=\dim(\rho_j,\rho_k)
=\dim (\overline{\kappa_{\overline{j}}}\kappa_{\overline{k}},\theta_j^{-1}\overline{\kappa_j}\kappa_k\theta_k),$$
we have 
\begin{equation}\label{intersection1}
\{[\theta_j^{-1}][\xi_{jj}^a][\theta_j]\}_{a\in \Lambda_{j,j}}\cap \{[\xi_{\bar{j}\bar{j}}^b]\}_{b\in \Lambda_{\bar{j},\bar{j}}}
=[\id],\end{equation}
and for $j\neq k$, 
\begin{equation}\label{intersection2}
\{[\theta_j^{-1}][\xi_{jk}^a][\theta_k]\}_{a\in \Lambda_{j,k}}\cap \{[\xi_{\bar{j}\bar{k}}^b]\}_{b\in \Lambda_{\bar{j},\bar{k}}}
=\emptyset.\end{equation}

Assume we have $\xi_{jk}^a$ with $d(\xi_{jk}^a)>1$. 
Since $\kappa_{\bar{j}}\theta_j^{-1}\xi_{jk}^a\theta_k\overline{\kappa_{\bar{k}}}$ is contained in $\rho_{\bar{j}}\rho_{k}$, 
the former contains either $\alpha_i$ with $i\in I$ or $\rho_l$ with $l\in J$.  
The first case never occurs because 
$$\dim (\kappa_{\bar{j}}\theta_j^{-1}\xi_{jk}^a\theta_k\overline{\kappa_{\bar{k}}},\alpha_i)
=\dim (\theta_j^{-1}\xi_{jk}^a\theta_k,\overline{\kappa_{\bar{j}}}\alpha_i\kappa_{\bar{k}})=
m_i\dim (\theta_j^{-1}\xi_{jk}^a\theta_k,\overline{\kappa_{\bar{j}}}\kappa_{\bar{k}})=0.$$
Thus 
$$0\neq  \dim(\kappa_{\bar{j}}\theta_j^{-1}\xi_{jk}^a\theta_k\overline{\kappa_{\bar{k}}},\rho_l)
=\dim(\theta_j^{-1}\xi_{jk}^a\theta_k,\overline{\kappa_{\bar{j}}}\kappa_l\theta_l\overline{\kappa_{\bar{l}}}\kappa_{\bar{k}}),$$
and there exist $\xi_{\bar{j}l}^b$ and $\xi_{\bar{l}\bar{k}}^c$ such that 
$\theta_j^{-1}\xi_{jk}^a\theta_k$ is contained in 
$\xi_{\bar{j}l}^b\theta_l\xi_{\bar{l}\bar{k}}^c$. 
In fact, the latter is irreducible because of  
$$\dim(\xi_{\bar{j}l}^b\theta_l\xi_{\bar{l}\; \bar{k}}^c,\xi_{\bar{j}l}^b\theta_l\xi_{\bar{l} \;\bar{k}}^c)
=(\theta_l^{-1}\xi_{l\bar{j}}^{\bar{b}}\xi_{\bar{j}l}^b\theta_l,\xi_{\bar{l} \;\bar{k}}^c\xi_{\bar{k} \bar{l}}^{\bar{c}}),$$
and Eq.(\ref{intersection1}). 
Therefore we get 
\begin{equation}\label{txt=xtx}
[\theta_j^{-1}\xi_{jk}^a\theta_k]=[\xi_{\bar{j}l}^b\theta_l\xi_{\bar{l}\;\bar{k}}^c].
\end{equation} 

Since $d(\xi_{jk}^a)>1$, we have either $d(\xi_{\bar{j}l}^b)>1$ or $d(\xi_{\bar{l}\;\bar{k}}^c)>1$. 
We first assume $d(\xi_{\bar{l}\;\bar{k}}^c)>1$. 
We have $[\xi_{jk}^a\theta_k]=[\theta_j\xi_{\bar{j}l}^b\theta_l\xi_{\bar{l}\;\bar{k}}^c]$.
Since $\kappa_j\theta_j\xi_{\bar{j}l}^b\theta_l\overline{\kappa_{\bar{l}}}$ is contained in $\rho_j\rho_l$, 
the former contains either $\alpha_i$ with $i\in I$ or $\rho_{r}$ with $r\in J$. 
In the first case, we have 
$$0\neq \dim(\kappa_j\theta_j\xi_{\bar{j}l}^b\theta_l\overline{\kappa_{\bar{l}}},\alpha_i)
=\dim(\theta_j\xi_{\bar{j}l}^b\theta_l,\overline{\kappa_j}\alpha_i\kappa_{\bar{l}})
=m_i\dim(\theta_j\xi_{\bar{j}l}^b\theta_l,\overline{\kappa_j}\kappa_{\bar{l}}),$$
and there exists $\xi_{j\bar{l}}^d$ satisfying $[\theta_j\xi_{\bar{j}l}^b\theta_l]=[\xi_{j\bar{l}}^d]$, and 
$[\xi_{jk}^a\theta_k]=[\xi_{j\bar{l}}^d\xi_{\bar{l}\bar{k}}^c]$. 
By the Frobenius reciprocity, there exists $\xi_{k\bar{k}}^e$ satisfying $[\theta_k]=[\xi_{k\bar{k}}^e]$.
Since 
$$[\kappa_k\overline{\kappa_k}\kappa_{\bar{k}}]=\bigoplus _{i_1\in I}m_{i_1}[\alpha_g\kappa_{\bar{k}}]=m[\kappa_{\bar{k}}],$$
we get $[\kappa_{k}\xi_{k\bar{k}}^e]=[\kappa_{\bar{k}}]$, 
and 
$$[\rho_k]=[\kappa_k\theta_k\overline{\kappa_{\bar{k}}}]=[\kappa_k\xi_{k\bar{k}}^e\overline{\kappa_{\bar{k}}}]=
[\kappa_{\bar{k}}\overline{\kappa_{\bar{k}}}]=\bigoplus_{i_1\in I}[\alpha_{i_1}],$$
which is contradiction. 
Thus we are left with 
$$0\neq \dim(\kappa_j\theta_j\xi_{\bar{j}l}^b\theta_l\overline{\kappa_{\bar{l}}},\rho_r)=
\dim(\theta_j\xi_{\bar{j}l}^b\theta_l,\overline{\kappa_j}\kappa_r\theta_r\overline{\kappa_{\bar{r}}}\kappa_{\bar{l}}),$$
which shows that there exist $\xi_{jr}^e$ and $\xi_{\bar{r}\bar{l}}^f$ satisfying 
$$\dim(\theta_j\xi_{\bar{j}l}^b\theta_l,\xi_{jr}^e\theta_r\xi_{\bar{r}\bar{l}}^f)\neq 0.$$
As before, the right-hand side is irreducible, and we get 
$[\theta_j\xi_{\bar{j}l}^b\theta_l]=[\xi_{jr}^e\theta_r\xi_{\bar{r}\bar{l}}^f]$, and 
$[\xi_{jk}^a\theta_k]=[\xi_{jr}^e\theta_r\xi_{\bar{r}\bar{l}}^f\xi_{\bar{l}\bar{k}}^c]$. 
Since the left-hand side is irreducible, so is $\xi_{\bar{r}\bar{l}}^f\xi_{\bar{l}\bar{k}}^c$, and there exists 
$\xi_{\bar{r}\bar{k}}^s$ satisfying $[\xi_{\bar{r}\bar{l}}^f\xi_{\bar{l}\;\bar{k}}^c]=[\xi_{\bar{r}\bar{k}}^s]$, 
and $[\xi_{jk}^a\theta_k]=[\xi_{jr}^e\theta_r\xi_{\bar{r}\;\bar{k}}^s]$. 
Note that we have $d(\xi_{\bar{r}\bar{k}}^s)>1$. 
By the Frobenius reciprocity,
$$1=\dim(\xi_{jk}^a\theta_k,\xi_{jr}^e\theta_r\xi_{\bar{r}\bar{k}}^s)
=\dim(\theta_r^{-1}\xi_{rj}^{\bar{e}}\xi_{jk}^a\theta_k,\xi_{\bar{r}\bar{k}}^s),$$
and there exists $\xi_{rk}^t$ satisfying $[\theta_r^{-1}\xi_{rk}^t\theta_k]=[\xi_{\bar{r}\bar{k}}^s]$, 
which contradicts Eq.(\ref{intersection1}). 

Now the only possibility is $d(\xi_{\bar{j}l}^b)>1$. 
Taking conjugate of Eq.(\ref{txt=xtx}), we get 
$[\theta_k^{-1}\xi_{kj}^{\bar{a}}\theta_j]=[\xi_{\bar{k}\bar{l}}^{\bar{c}}\theta_l^{-1}\xi_{l\bar{j}}^{\bar{b}}]$, 
and $[\xi_{kj}^{\bar{a}}\theta_j]=[\theta_k\xi_{\bar{k}\bar{l}}^{\bar{c}}\theta_l^{-1}\xi_{l\bar{j}}^{\bar{b}}]$. 
Since $\kappa_k\theta_k\xi_{\bar{k}\bar{l}}^{\bar{c}}\theta_l^{-1}\overline{\kappa_l}$ is contained in $\rho_k\rho_{\bar{l}}$, 
a similar argument as above works, and we get contradiction again. 
Therefore $d(\xi_{jk}^a)=1$ for all $j,k,a$. 
\end{proof}

\begin{proof}[Proof of Theorem \ref{GthF}] 
We fix $j_0\in J$. 
Since $\overline{\kappa_{j_0}}\kappa_k$ contains an isomorphism $\varphi_j:R_j\to R_{j_0}$, 
by the Frobenius reciprocity, we get $[\kappa_j]=[\kappa_{j_0}\varphi_j]$. 
Thus there exists a unitary $u_j\in P$ satisfying $\Ad u_j\circ \kappa_j=\kappa_{j_0}\circ \varphi_j$, 
which means that for every $x\in R_j$, 
$$u_jxu_j^*=\varphi_j(x).$$
This implies $u_jR_ju_j^*=R_{j_0}.$
By replacing $\rho_j$ with $\Ad u_j\circ \rho_j$ if necessary, we may assume $R_j=R_{j_0}$ for all $j\in J$. 
We denote $R=R_{j_0}$ and $\kappa=\kappa_{j_0}$ for simplicity. 
Now we have $\theta_j\in \Aut(R)$ and $[\rho_j]=[\kappa\theta_j\overline{\kappa}]$. 

Since $\overline{\kappa}\kappa$ is decomposed into 1-dimensional sectors, 
the inclusion $P\supset R$ is a crossed product by a finite group of order $m$, say $H$, 
and there exists an outer action $\beta$ of $H$ on $R$ such that $P=R\rtimes_\beta H$, and 
$$[\overline{\kappa}\kappa]=\bigoplus_{h\in H}[\beta_h].$$

Note that $N\supset R$ is irreducible because 
$$\dim(\iota\kappa,\iota\kappa)=\dim(\biota\iota,\kappa\overline{\kappa})=1.$$
Now we have 
$$[\overline{(\iota\kappa)}\iota\kappa]=[\bkappa\biota\iota\kappa]=[\bkappa\kappa]\oplus \bigoplus_{j\in J}[\bkappa\rho_j\kappa]
=\bigoplus_{h\in H}[\beta_h]\oplus \bigoplus_{j\in J,\;h_1,h_2\in H}
[\beta_{h_1}\theta_j\beta_{h_2}].$$
This shows that there exists a finite group $G$ including $H$, and its outer action $\gamma$ on $R$ 
extending $\beta$ satisfying $N=R\rtimes _\gamma G$. 
Moreover, 
$$\bigoplus_{g\in G}[\gamma_g]=\bigoplus_{h\in H}[\beta_h]\oplus 
\bigoplus_{j\in J,\;h_1,h_2\in H}[\beta_{h_1}\theta_j\beta_{h_2}],$$
holds, which shows that every $(H,H)$-double coset except for $H$ has 
size $|H|^2$. 
Therefore $G$ is a Frobenius group with a Frobenius complement $H$, and it is of the form $K\rtimes H$ 
with the Frobenius kernel $K$. 
Since $|K|=[N:P]$, we get $|K|=1+mn$. 

When $n=1$, we have $|K|=|H|+1$, and $G$ acting on $G/H$ is a sharply 2-transitive permutation group. 

For (2), it suffices to show that $H$ is maximal in $G$.  
For this, it suffices to show that there is no non-trivial intermediate subfactor between $N$ and $P$. 
Assume $n=2$ first. 
Suppose $Q$ is a non-trivial intermediate subfactor and let $\iota_1:P\hookrightarrow Q$ be the inclusion map. 
Since $[\biota\iota]=[\id]\oplus [\rho_1]\oplus [\rho_2]$, we have either $[\overline{\iota_1}\iota_1]=[\id]\oplus [\rho_1]$ 
or $[\overline{\iota_1}\iota_1]=[\id]\oplus [\rho_2]$. 
In any case, we get $[Q:P]=1+m$, and 
$$[N:Q]=\frac{[N:P]}{[Q:P]}=\frac{1+2m}{1+m}=2-\frac{1}{1+m},$$
which is forbidden by the Jones theorem. 

The case $n=3$ can be treated in a similar way. 
\end{proof}

\begin{remark}\label{uniqueness} The above theorem together with the classification of sharply 2-transitive permutation groups with abelian 
point stabilizers shows that the graph $\cG_{(1^m),1}$ uniquely characterizes the group-subgroup subfactor for 
$H(q)=\F_q\rtimes \F_q^\times>\F_q^\times$ with $q=m+1$. 
In the case of non-commutative $H$, probably the graph $\cG_{\bm,1}$ does not uniquely determine the group $K\rtimes H$ in general. 
However, \cite[ Chapter XII, Theorem 9.7]{HB} shows that possibilities of $H\rtimes K$ for a given $q=m+1$ are very much restricted. 
For example, the graph $\cG_{(1^4,2),1}$ uniquely characterizes $S(3^2)>Q_8$. 
\end{remark}

In the rest of this section, we classify related fusion categories, which is a generalization of \cite[(7.1)]{EGO04}. 

Let $\cC_0$ be a $C^*$-fusion category with the set of (equivalence classes of) simple objects $\mathrm{Irr}(\cC)=\{\alpha_i\}_{i\in I}$. 
We may assume $0\in I$ and $\alpha_0=1$. 
Let $\cC$ be a fusion category containing $\cC_0$ with $\mathrm{Irr}(\cC)=\{\alpha_i\}_{i\in I}\cup \{\rho\}$. 
Then we have $\alpha_i\otimes \rho\cong \rho\otimes \alpha_i=d(\alpha_i)\rho$. 
Indeed, if $\alpha_i\otimes \rho$ contained $\alpha_j$, the Frobenius reciprocity implies that $\alpha_{\overline{i}}\otimes\alpha_j$ 
would contain $\rho$, which is impossible, and the claim holds. 
In particular $m_i=d(\alpha_i)$ is an integer. 
By the Frobenius reciprocity again we get 
$$\rho\otimes \rho\cong \bigoplus_{i\in I}m_i\alpha_i\oplus k\rho,$$
where $k$ is a non-negative integer. 
We now consider the case with $k=m-1$, where 
$$m=\sum_{i\in I}m_i^2.$$
Then $d(\rho)=m$.

\begin{theorem}\label{classification} Let $\cC$ be as above. 
Then there exists a sharply 2-transitive permutation group $G=K\rtimes H$ with the Frobenius kernel $K$ 
and a Frobenius complement $H$ such that $\cC_0$ is equivalent to the representation category of $H$. 
In particular, the number $m+1$ is a prime power $p^k$.  
The category $\cC$ is classified by 
$$\{\omega\in H^3(K\rtimes H,\T)\:|\: \omega|_H=0\}/\Aut(K\rtimes H,H),$$
(or equivalently by $H^3(K,\T)^H/N_{\Aut(K)}(H)$).  
\end{theorem}

\begin{proof} 
For the proof of Theorem \ref{classification}, we may assume that the category $\cC$ is embedded in 
$\End(P)$ for a type III factor $P$. 

In the same way as in the proofs of Theorem \ref{GthF}, 
there exist a unique subfactor $R\subset P$ , up to inner conjugacy, a unique finite group $H$ of order $m$, 
$\theta\in \Aut(R)$, and an outer action $\beta$ of $H$ on $R$ such that 
$$P=R\rtimes_\beta H,$$
and if $\kappa:\hookrightarrow P$ is the inclusion map, 
$$[\kappa\overline{\kappa}]=\bigoplus_{i\in I}m_i[\alpha_i],$$
$$[\overline{\kappa}\kappa]=\bigoplus_{h\in H}[\beta_h],$$
$$[\rho]=[\kappa\theta\overline{\kappa}].$$

Let $G$ be the group generated by $[\beta_H]=\{[\beta_h]\}_{h\in H}$ and $[\theta]$ in $\Out(R)$. 
We will show 
$$G=[\beta_H]\sqcup [\beta_H][\theta][\beta_H],$$ 
whose order is $m(m+1)$, and it is a Frobenius group with 
a Frobenius complement $[\beta_H]$. 

The proof of Lemma \ref{1dim} shows  $[\theta]\notin [\beta_H]$,  
$$[\theta][\beta_H][\theta^{-1}]\cap [\beta_H]=[\id],$$ 
and $|[\beta_H][\theta][\beta_H]|=m^2$.  
Let $G_0=[\beta_H]\cup [\beta_H][\theta][\beta_H]$, which is a subset of $G$ with $|G_0|=m(m+1)$. 
To prove that $G_0$ coincides with $G$, it suffices to show $[\theta][\beta_H][\theta]\subset G_0$ and 
$[\theta^{-1}]\in [\beta_H][\theta][\beta_H]$. 

Let $h\in H$. 
Since $\kappa\theta\beta_h\theta\bar{\kappa}$ is contained in $\rho^2$, it contains either $\alpha_i$ with $i\in I$ or $\rho$. 
If it contains $\alpha_i$, we have 
\begin{align*}
0 &\neq  \dim(\kappa\theta\beta_h\theta\bar{\kappa},\alpha_i)=\dim(\theta\beta_h\theta,\bar\kappa\alpha_i\kappa)
=m_i\dim(\theta\beta_h\theta,\bar\kappa\kappa)\\
&=m_i\sum_{k\in H}\dim(\theta\beta_h\theta,\beta_k),
\end{align*}
which shows $[\theta\beta_h\theta]\in [\beta_H]$. 
If it contains $\rho$, 
$$0\neq \dim(\kappa\theta\beta_h\theta\bar{\kappa},\kappa\theta\bar{\kappa})
=\dim(\theta\beta_h\theta,\bar{\kappa}\kappa\theta\bar{\kappa}\kappa)
=\sum_{k,l}\dim(\theta\beta_h\theta,\beta_k\theta\beta_l),$$
which shows $[\theta\beta_h\theta]\in [\beta_H][\theta][\beta_H]$. 
Therefore we get $[\theta][\beta_H][\theta]\subset G_0$. 

Since $\rho$ is self-conjugate, we have 
$$1=\dim(\bar{\rho},\rho)=\dim(\kappa\theta^{-1}\bar{\kappa},\kappa\theta \bar{\kappa})
=\dim(\theta^{-1},\bar{\kappa}\kappa\theta\bar{\kappa}\kappa)
=\sum_{h,k\in H}\dim(\theta^{-1},\beta_h\theta\beta_k),$$
which shows $[\theta^{-1}]\in [\beta_H][\theta][\beta_H]$. 
Therefore we get $G=G_0$. 

Since $G$ has only two $(H,H)$-double cosets, and the size of $[\beta_H][\theta][\beta_H]$ is $|H|^2$, 
the group $G$ is a Frobenius group with a Frobenius complement $[\beta_{H}]$. 
Moreover, the $G$ action on $G/H$ is sharply 2-transitive. 

For the classification of the category $\cC$, we may assume that $R$ is the injective type III$_1$ factor. 
Then the conjugacy class of $G$ in $\Out(R)$ is completely determined by its obstruction class 
$\omega\in H^3(G,\T)$. 
Since $H$ has a lifting $\beta_H\subset \Aut(R)$, the restriction of $\omega$ to $H$ is trivial. 
Since $H$ is a Frobenius complement, the Schur multiplier $H^2(H,\T)$ is trivial, and 
the lifting is unique, up to cocycle conjugacy, and one can uniquely recover 
$P$ from $R$ and $[\beta_H]$. 
This means that the generator $\rho$ of the category $\tilde{\cC}$ is uniquely determined by $\omega$. 
On the other hand, there always exists a $G$-kernel in $\Out(R)$ for a given $\omega\in H^3(G,\T)$, 
which shows the existence part of the statement.  

Finally, since $|K|$ and $|H|$ are relatively prime, we have 
$$E^{p,q}_2=H^p(H,H^q(K,\T))=0,$$
for $p,q\geq 1$ in the  Lindon/Hochschild-Serre spectral sequence for $G=K\rtimes H$.
Thus the group 
$$\{\omega\in H^3(G,\T)\:|\: \omega|_H=0\},$$
is isomorphic to $H^3(K,\T)^H$. 
\end{proof}

When $H$ is abelian (in fact cyclic in this case), the group $H^3(K,\T)^H$ is explicitly computed in \cite[Corollary 7.4]{EGO04}.


\section{Goldman-type theorems for sharply 3-transitive permutation groups}
Let $\bm$, $n$, $I$, and $m$ be as in the previous section. 
Now we consider the graph $\widetilde{\cG_{\bm,1}}$ (see Subsection 2.3 for the definition of $\widetilde{\cG}$ for a given $\cG$), 
which is described as follows. 
The set of even vertices of $\widetilde{\cG_{\bm,1}}$ is 
$$\{v_i^0\}_{i\in I}\sqcup\{v^2_i\}_{\in I}\sqcup\{v^4\},$$
the set of odd vertices is 
$$\{v^1_i\}_{i\in I}\sqcup\{v^3\}.$$
The only non-zero entries of the adjacency matrix $\Delta$ of $\widetilde{\cG_{\bm,1}}$ are 
$$\Delta(v^0_i,v^1_i)=\Delta(v^1_i,v^0_i)=1,\quad \forall i\in I,$$
$$\Delta(v^1_i,v^2_i)=\Delta(v^2_i,v^1_i)=1,\quad \forall i\in I,$$
$$\Delta(v^{2}_i,v^3)=\Delta(v^3, v^{2}_i)=m_i,\quad \forall i\in I,$$
$$\Delta(v^3,v^4)=\Delta(v^4,v^3)=1.$$
The vertex $v_0^0$ is treated as a distinguished vertex $*$. 
The Perron-Frobenius eigenvalue of $\Delta$ is $\sqrt{2+m}$. 
The Perron-Frobenius eigenvector $d$ normalized as $d(v_0^0)=1$ is 
$$d(v_i^0)=m_i,\quad d(v_i^2)=m_i(1+mn),\quad d(v_4)=m.$$
$$d(v_i^1)=m_i\sqrt{2+mn},\quad d(v_3)=m\sqrt{2+mn}.$$

In \cite{IMPP15}, we showed that a strong Goldman-type theorem for 
$\widetilde{\cG_{(1^3),1}}$.
Now we show it for general sharply 3-transitive permutation groups. 
\begin{figure}[H]
\centering
\begin{tikzpicture}
\draw (-3,0)--(-2,0)--(-1,0)--(0,0)--(1,0)--(2,0)--(3,0);
\draw (0,1)--(0,0)--(0,-1)--(0,-2)--(0,-3);
\draw(-3,0)node{$*$};
\draw(-2,0)node{$\bullet$};
\draw(-1,0)node{$\bullet$};
\draw(0,0)node{$\bullet$};
\draw(1,0)node{$\bullet$};
\draw(2,0)node{$\bullet$};
\draw(3,0)node{$\bullet$};
\draw(0,1)node{$\bullet$};
\draw(0,-1)node{$\bullet$};
\draw(0,-2)node{$\bullet$};
\draw(0,-3)node{$\bullet$};
\end{tikzpicture}
\caption{$\widetilde{\cG_{(1^3),1}}=\cG_{(L(2^2),PG_1(2^2))}=\cG_{(\fA_5,X_5)}$}
\end{figure}

Although we excluded the case $\bm=(1)$ in the definition of $\cG_{\bm,1}$ in Section 3, 
the graph itself makes sense for $\bm=(1)$, and we include this case in the next theorem. 

\begin{theorem}\label{LM} Let $M\supset N$ be a finite index subfactor with $\cG_{M\supset N}=\widetilde{\cG_{\bm,1}}$. 
Then $q=1+m$ is a prime power, and there exists a unique subfactor $R\subset N$ that is irreducible in $M$ such that 
if $\bm=1^m$,   
$$M=R\rtimes L(q)\supset N=R\rtimes H(q),$$
and otherwise 
$$M=R\rtimes M(q)\supset N=R\rtimes S(q).$$
\begin{figure}[H]
\centering
\begin{tikzpicture}
\draw (-3,0)--(-2,0)--(-1,0)--(0,0)--(1,0)--(2,0)--(3,0);
\draw (0,0)--(1,1)--(2,2)--(3,3);
\draw (0,0)--(1,-1)--(2,-2)--(3,-3);
\draw (0,1)--(0,0)--(0,-1)--(0,-2)--(0,-3);
\draw(-3,0)node{$*$};
\draw(-2,0)node{$\bullet$};
\draw(-1,0)node{$\bullet$};
\draw(0,0)node{$\bullet$};
\draw(1,0)node{$\bullet$};
\draw(2,0)node{$\bullet$};
\draw(3,0)node{$\bullet$};
\draw(1,1)node{$\bullet$};
\draw(2,2)node{$\bullet$};
\draw(3,3)node{$\bullet$};
\draw(1,-1)node{$\bullet$};
\draw(2,-2)node{$\bullet$};
\draw(3,-3)node{$\bullet$};
\draw(0,1)node{$\bullet$};
\draw(0,-1)node{$\bullet$};
\draw(0,-2)node{$\bullet$};
\draw(0,-3)node{$\bullet$};
\draw(-0.1,-0.5)node{2};
\draw(-3,0)node[above]{$\id_N$};
\draw(-2,0)node[above]{$\epsilon$};
\draw(-1,0)node[above]{$\sigma$};
\draw(-0.3,0.2)node{$\epsilon\rho'$};
\draw(0,1)node[above]{$\rho'$};
\draw(0.8,1.3)node{$\sigma\alpha_1'$};
\draw(1.8,2.3)node{$\epsilon\alpha_1'$};
\draw(2.8,3.3)node{$\alpha_1'$};
\draw(1,0)node[above]{$\sigma\alpha_2'$};
\draw(2,0)node[above]{$\epsilon\alpha_2'$};
\draw(3,0)node[above]{$\alpha_2'$};
\draw(0.8,-1.3)node{$\sigma\alpha_3'$};
\draw(1.8,-2.3)node{$\epsilon\alpha_3'$};
\draw(2.8,-3.3)node{$\alpha_3'$};
\draw(0,-1)node[left]{$\sigma\alpha_4'$};
\draw(0,-2)node[left]{$\epsilon\alpha_4'$};
\draw(0,-3)node[left]{$\alpha_4'$};
\end{tikzpicture}
\caption{\label{M(9)}$\widetilde{\cG_{(1^4,2),1}}=\cG_{(M(3^2),PG_1(3^2))}$}
\end{figure}
\end{theorem}

\begin{proof} If $\bm=(1)$, the graph $\widetilde{\cG_{(1),1}}$ is nothing but the Coxeter graph $A_5$, and the 
statement follows from \cite{I92} as $(\fS_3,X_3)\cong (PGL_2(2),PG_1(2))$. 
We assume $\bm\neq (1)$ in what follows. 

We follow the strategy described in Subsection 2.5 taking the 6 steps. 

(1)  Let $\epsilon:N \hookrightarrow M$ be the inclusion map, and let $\cC$ be the fusion category generated by $\bepsilon\epsilon$. 
We first parametrize $\Irr(\cC)$. 
Let $[\bepsilon\epsilon]=[\id]\oplus [\sigma]$ be the irreducible decomposition, which means that 
$\sigma$ corresponds to the vertex $v_0^2$. 
We denote by $\alpha_i'$ and $\rho'$ the endomorphisms of $N$ corresponding to $v_i^0$ and $v^4$ respectively. 
Then $\epsilon\alpha_i'$, $\sigma\alpha'_i$, and $\epsilon\rho'$ are irreducible, and they correspond 
to $v_i^1$, $v_i^2$, $v^3$ respectively. 
Thus 
$$\Irr(\cC)=\{\alpha'_i\}_{i\in I}\sqcup \{\sigma\alpha'_i\}_{i\in I}\sqcup \{\rho'\}.$$
We have 
$$d(\alpha_i)=m_i,\quad d(\epsilon)=\sqrt{2+m},\quad d(\sigma)=1+m, \quad d(\rho')=m.$$
Two endomorphisms $\sigma$ and $\rho'$ are self-conjugate. 
We introduce two involutions of $I$ by $[\overline{\alpha'_i}]=[\alpha'_{\bar{i}}]$ and  
$[\overline{\sigma\alpha'_i}]=[\sigma\alpha'_{i^*}]$. 
Then they are related by $[\sigma\alpha'_{i^*}]=[\alpha'_{\bar{i}}\sigma]$. 

By dimension counting, we see that there exists a fusion subcategory $\cC_0$ of $\cC$ with 
$$\Irr(\cC_0)=\{\alpha'_i\}_{i\in I}.$$
We claim that there exists another fusion subcategory $\cC_1$ of $\cC$ with 
$$\Irr(\cC_1)=\{\alpha'_i\}_{i\in I}\sqcup \{\rho'\}.$$
Indeed, if $\rho'\alpha_i$ contained $\sigma\alpha'_{i_1}$, the Frobenius reciprocity implies that 
$\sigma\alpha'_{i_1}\alpha'_{\bar{i}}$ would contain $\rho'$, and hence $\sigma\alpha'_{i_2}$ would 
contain $\rho'$ for some $i_2$, which is contradiction. 
Thus $\rho\alpha_i$ is decomposed into a direct sum of sectors in $\{\alpha'_{i_1}\}_{i_1\in I}\cup\{\rho'\}$, 
and dimension counting shows 
\begin{equation}
[\rho'\alpha'_i]=m_i[\rho'], \quad [\alpha'_i\rho]=m_i[\rho],
\end{equation}
where the second equality follows from the first one by conjugation. 

From the shape of the graph $\widetilde{\cG_{\bm,1}}$, we can see
\begin{equation}\label{sigma2}
[\sigma^2]=[\id]\oplus [\rho']\oplus \bigoplus_{i\in I}m_i[\sigma\alpha'_i],
\end{equation}
\begin{equation}
[\sigma\rho']=\bigoplus_{i\in I}m_i[\sigma\alpha'_i]. 
\end{equation}
Using these and associativity, we have 
\begin{align*}
[\sigma][\sigma\rho']&=\bigoplus_{i_1\in I}m_i[\sigma][\sigma\alpha'_{i}] \\
 &=\bigoplus_{i\in I}m_{i}([\id]\oplus [\rho']\oplus \bigoplus_{i'\in I}m_{i'}[\sigma\alpha'_{i'}])[\alpha'_{i}] \\
 &=\bigoplus_{i\in I}m_{i}[\alpha'_i]\oplus m[\rho']\oplus\bigoplus_{i,i'}m_{i'}[\sigma\alpha'_{i'}\alpha'_{i}].
\end{align*}
On the other hand,
\begin{align*}
\lefteqn{[\sigma][\sigma\rho']=[\sigma^2][\rho']=([\id]\oplus [\rho']\oplus \bigoplus_{i\in I}m_i[\sigma\alpha'_i])[\rho'] } \\
 &=[\rho']\oplus [{\rho'}^2]\oplus  \bigoplus_{i\in I}m_i^2[\sigma\rho']=[\rho']\oplus [{\rho'}^2]
 \oplus  m\bigoplus_{i\in I}m_i[\sigma\alpha'_i]. 
\end{align*}
Since $\sigma\alpha'_{i'}\alpha'_{i}$ is a direct sum of irreducibles of the form $\sigma\alpha'_{i''}$, 
the endomorphism $\rho^2$ contains 
$$\bigoplus_{i\in I}m_i[\alpha'_i]\oplus (m-1)[\rho'],$$
and comparing dimensions, we get 
\begin{equation}
[{\rho'}^2]=\bigoplus_{i\in I}m_{i}[\alpha'_i]\oplus (m-1)[\rho'].
\end{equation}
Therefore the claim is shown. 

(2) 
Form Eq.(\ref{sigma2}) and Theorem \ref{intermediate}, there exists a unique intermediate subfactor $N\supset P\supset \sigma(N)$ such that 
if $\iota :P\hookrightarrow N$ is the inclusion map, we have $[\iota\biota]=[\id]\oplus [\rho']$. 
Let $\cC_2$ be the fusion category generated by $\biota\iota$. 
As in the proof of Lemma \ref{factorization}, there exists $\tau\in \Aut(P)$ satisfying 
\begin{equation}
[\sigma]=[\iota\tau\biota].
\end{equation}

(3) From the fusion rules of $\cC_1$, we can see that the dual principal graph $\cG_{N\supset P}^d$ is $\cG_{\bm,1}$,  
and Theorem \ref{GthF},(1) shows that so is the principal graph $\cG_{N>P}$ too. 
Therefore we can arrange the labeling of irreducibles of $\cC_2$ so that 
$$\Irr(\cC_2)=\{\alpha_i\}\sqcup \{\rho_i\},$$
and $[\alpha_i'\iota]=[\iota\alpha_i]$ and $[\biota\iota]=[\id]\oplus [\rho]$. 

(4) Now we apply Theorem \ref{GthF}, and we get a unique subfactor $R\subset P$, up to inner conjugacy such that $R'\cap P=\C$ 
and there exists an outer action $\beta$ of a Frobenius group $K\rtimes H$ satisfying 
$$N=R\rtimes_\beta (H\rtimes K)\supset P=R\rtimes_\beta H.$$
Moreover the $K\rtimes H$-action on $(K\rtimes H)/H$ is sharply 2-transitive. 
We denote by $\kappa:R\hookrightarrow P$ the inclusion map. 
Then we have 
$$[\iota\kappa\bkappa\biota]=\bigoplus_{i\in I}m_i[\iota\alpha_i\biota]=\bigoplus_{i\in I}m_i[\alpha'_i\iota\biota]
=\bigoplus_{i\in I}m_i[\alpha'_i]([\id]\oplus \rho')=\bigoplus_{i\in I}m_i[\alpha'_i]\oplus m[\rho'],$$
which shows 
$$\dim(\epsilon\iota\kappa,\epsilon\iota\kappa)=\dim(\bepsilon\epsilon,\iota\kappa\bkappa\biota)=1,$$
and $R$ is irreducible in $M$. 

(5) Since 
$$[M:P]=[M:N][N:P][P:R]=(m+2)(m+1)m,$$ 
to prove that the inclusion $L\supset R$ is of depth 2, it suffices to show that the number $(m+2)(m+1)m$ coincides with 
the following dimension: 
$$\dim(\epsilon\iota\kappa\overline{(\epsilon\iota\kappa)},\epsilon\iota\kappa\overline{(\epsilon\iota\kappa)})
=\dim(\bepsilon\epsilon\iota\kappa\bkappa\biota,\iota\kappa\bkappa\biota\bepsilon\epsilon )
=\dim ((\id\oplus \sigma)\iota\kappa\bkappa\biota,\iota\kappa\bkappa\biota(\id\oplus \sigma)).$$
Note that $[\sigma]$ commutes with $[\rho']$ and 
$$\bigoplus_{i\in I}m_i[\alpha'_i],$$
and hence with $[\iota\kappa\bkappa\biota]$. 
Thus this number is equal to 
$$=\dim ((\id\oplus \sigma)\iota\kappa\bkappa\biota,(\id\oplus \sigma)\iota\kappa\bkappa\biota)
=\dim ((\id\oplus \sigma)^2, (\iota\kappa\bkappa\biota)^2).$$
Since the fusion category generated by $\iota\kappa\bkappa\biota$ is equivalent to the representation category 
$\Rep(K\rtimes H)$, and $\iota\kappa\bkappa\biota$ corresponds to the regular representation of $K\rtimes H$, we get 
$$[(\iota\kappa\bkappa\biota)^2]=m(m+1)[\iota\kappa\bkappa\biota].$$
Thus \begin{align*}
\lefteqn{\dim ((\id\oplus \sigma)^2, (\iota\kappa\bkappa\biota)^2)
=m(m+1)(\id\oplus 2\sigma\oplus \sigma^2, \iota\kappa\bkappa\biota)} \\
 &= m(m+1)\dim(2\id \oplus \rho \oplus 2\sigma\oplus \bigoplus_{i\in I}m_i\sigma\alpha'_i, 
 \bigoplus_{i\in I}m_i\alpha'_i\oplus m\rho')\\
 &=m(m+1)(m+2),
\end{align*}
and the inclusion $M\supset R$ is of depth 2.

(6) Now Lemma \ref{inv1} shows that we have 
$$m=\dim(\tau\biota\iota\kappa\bkappa\tau^{-1},\biota\iota\kappa\bkappa),$$
and 
$$[\biota\iota\kappa\bkappa]=[(\id\oplus\rho)\bigoplus_{i\in I}m_i\alpha_i]=\bigoplus_{i\in I}m_i[\alpha_i]
\oplus m[\rho].$$
Dimension counting implies 
$$m=\dim(\bigoplus_{i\in I}m_i\tau\alpha_i\tau^{-1},\bigoplus_{i\in I}m_{i}\alpha_{i}),$$
and this is possible only if $[\tau\kappa\bkappa\tau^{-1}]=[\kappa\bkappa]$. 
Since $H$ is a Frobenius complement, every abelian subgroup of $H$ is cyclic, and Lemma \ref{inv2} implies 
there exists $\tau_1\in \Aut(R)$ satisfying $[\tau\kappa]=[\kappa][\tau_1]$. 

Now Lemma \ref{finishing} shows that there exists a group $G$ including $K\rtimes H$, and outer $G$-action on 
$R$ extending $\beta$ satisfying $M=R\rtimes_\gamma G$. 
The principal graph $\cG_{M\supset N}$ shows that the $G$-action on $G/(K\rtimes H)$ is 3-transitive. 
Since $|G/(K\rtimes H)|=m+2$, and $|G|=m(m+1)(m+2)$, the permutation group $G$ is sharply 3-transitive. 
Now the statement follows from the classification of sharply 3-transitive permutation groups.  
\end{proof}

We devote the rest of this section to a preparation of the Goldman-type theorem for the Mathieu groups $M_{11}$. 
Since $M(3^2)$ and $S(3^2)$ are a point stabilizer and a two point stabilizer of the sharply 4-transitive action of 
$M_{11}$, we denote $M(3^2)=M_{10}$ and $S(3^2)=M_9$.  
We first determine the dual principal graph $\cG_N^M$ in the case of $M_{10}>M_9$. 
Since this graph is the induction-reduction graph $\cG_{M_9}^{M_{10}}$, the irreducible $M$-$M$ sectors are parametrized by 
the irreducible representations of $M_{10}$, whose ranks are $1,1,9,9,10,10,10,16$ (see \cite[Table 8]{GH21}).  

We parametrize the irreducible $N$-$N$ and $M$-$N$ sectors as in the above proof and Figure \ref{M(9)}. 
Theorem \ref{GthF},(1) shows that $N\supset P$ and its dual inclusion 
are isomorphic subfactor associated with $(S(3^2),\F_{3^2})$ (see Remark \ref{uniqueness}), 
and the two fusion categories $\cC_1$ and $\cC_2$ are equivalent. 
On the other hand, the fusion subcategory generated by $\kappa\bkappa$ in $\cC_2$ is equivalent to $\Rep(Q_8)$. 
Thus the fusion category $\cC_0$ is equivalent to $\Rep(Q_8)$. 
In particular, we have $\bar{i}=i$ for all $i$. 
Since at least one of $\{1,2,3\}$ is fixed by the other involution $i\mapsto i^*$, we may and do assume $1^*=1$, 
and $\sigma\alpha_1$ is self-conjugate. 
Since $d(\alpha'_4)=2$, the two sectors $\alpha'_4$ and $\sigma\alpha'_4$ are self-conjugate. 

Let $[\epsilon\bepsilon]=[\id]\oplus [\pi]$ be the irreducible decomposition. 
Then $d(\pi)=9$. 
Since 
$$\dim(\epsilon\alpha'_i\bepsilon,\epsilon\alpha'_i\bepsilon)=\dim(\bepsilon\epsilon\alpha'_i,\alpha'_i\bepsilon\epsilon)
=\dim ((\id\oplus \sigma)\alpha'_i,\alpha'_i(\id\oplus \sigma))=1+\dim(\sigma\alpha'_{i},\sigma\alpha'_{i^*}),$$
if $i^*=i$, the endomorphism $\epsilon\alpha'_i\bepsilon$ is decomposed into two irreducibles, and otherwise it is irreducible. 
Thus $\epsilon\alpha'_1\bepsilon$ is decomposed into two irreducibles. 
Since $d(\epsilon\alpha'_1\bepsilon)=10$, it is a direct sum of a 1-dimensional representation and a 9-dimensional representation,  
and we denote the former by $\chi$. 
Then the Frobenius reciprocity implies $[\chi\epsilon]=[\epsilon\alpha'_1]$, and  
$$[\epsilon\alpha'_1\bepsilon]=[\chi\epsilon\bepsilon]=[\chi]\oplus [\chi\pi].$$

Since $\epsilon\alpha'_i\bepsilon$ for $i=2,3,$ cannot contain a 1-dimensional representation, we have $2^*=3$, and 
$\xi:=\epsilon\alpha'_2\bepsilon$ is irreducible. 
By
\begin{align*}
\lefteqn{[\epsilon\alpha'_i\bepsilon][\epsilon]=[\epsilon\alpha'_i(\id\oplus \sigma)]
=[\epsilon\alpha'_i]\oplus [\epsilon\alpha'_i\sigma]} \\
 &=[\epsilon\alpha'_i]\oplus [\epsilon\sigma\alpha'_{i^*}]=
[\epsilon\alpha'_i]\oplus [\epsilon\alpha'_{i^*}]\oplus d(\alpha'_i)[\epsilon\rho'],
\end{align*} 
and the Frobenius reciprocity, we also have $[\epsilon\alpha'_2\bepsilon]=[\xi],$ and 
$$[\xi\epsilon]= [\epsilon\alpha'_2]\oplus [\epsilon\alpha'_{3}]\oplus [\epsilon\rho'].$$

Since $d(\epsilon\alpha'_4\bepsilon)=20$, we have 
$$[\epsilon\alpha'_4\bepsilon]=[\eta_1]\oplus [\eta_2],$$ 
with $d(\eta_1)=d(\eta_2)=10$, and  
$$[\eta_1\epsilon]=[\epsilon\alpha'_4]\oplus [\epsilon\rho'].$$
$$[\eta_2\epsilon]=[\epsilon\alpha'_4]\oplus [\epsilon\rho'].$$

There is one irreducible representation of $M_{10}$ missing, which we denote by $\zeta$. 
By the Frobenius reciprocity and $d(\zeta)=16$, we get 
$$[\epsilon\rho'\bepsilon]=[\pi]\oplus [\chi\pi]\oplus [\xi]\oplus [\eta_1]\oplus [\eta_2]\oplus 2[\zeta].$$
Thus the graph $\cG_{M_{9}}^{M_{10}}$ is as follows. 
\begin{figure}[H]
\centering
\begin{tikzpicture}
\draw (-3,1.5)--(-2,1)--(-1,0.5)--(0,0)--(-1,-0.5)--(-2,-1)--(-3,-1.5);
\draw (0,0)--(0,1)--(-1,2);
\draw (0,1)--(1,2);
\draw (0,0)--(1,0.5)--(2,0)--(1,-0.5)--(0,0);
\draw (0,0)--(0,-1);
\draw(0.1,-0.5)node{2};
\draw(-3,1.5)node{$*$};
\draw(-2,1)node{$\bullet$};
\draw(-1,0.5)node{$\bullet$};
\draw(0,0)node{$\bullet$};
\draw(-1,-0.5)node{$\bullet$};
\draw(-2,-1)node{$\bullet$};
\draw(-3,-1.5)node{$\bullet$};
\draw(0,1)node{$\bullet$};
\draw(-1,2)node{$\bullet$};
\draw(1,2)node{$\bullet$};
\draw(1,0.5)node{$\bullet$};
\draw(1,-0.5)node{$\bullet$};
\draw(2,0)node{$\bullet$};
\draw(0,-1)node{$\bullet$};
\draw(-3.1,1.3)node{$\id_M$};
\draw(-2.1,0.8)node{$\epsilon$};
\draw(-1.1,0.3)node{$\pi$};
\draw(-3.2,-1.3)node{$\chi$};
\draw(-2.1,-0.8)node{$\chi\epsilon$};
\draw(-1.1,-0.3)node{$\chi\pi$};
\draw(0.2,0.3)node{$\epsilon_0$};
\draw(0,1)node[right]{$\xi$};
\draw(-1.1,1.7)node{$\epsilon_2$};
\draw(1.2,1.7)node{$\epsilon_3$};
\draw(1,0.5)node[above]{$\eta_1$};
\draw(1,-0.5)node[below]{$\eta_2$};
\draw(2,0)node[right]{$\epsilon_4$};
\draw(0,-1)node[below]{$\zeta$};
\end{tikzpicture}
\caption{\label{M10-M9}$\cG_{M_9}^{M_{10}}$}
\end{figure}

\begin{theorem}\label{dualM10-M9} Let $M\supset N$ be a finite index subfactor with $\cG_{M\supset N}^d=\cG_{M_9}^{M_{10}}$. 
Then we have $\cG_{M\supset N}=\cG_{M_{10}>M_9}$. 
In consequence, there exists a unique subfactor $R\subset N$ up to inner conjugacy, that is irreducible in $M$ such that 
$$M=R\rtimes M_{10}\supset N=R\rtimes M_9.$$
\end{theorem}

We divide the proof in a few steps. 
We parametrize the $M$-$M$ sectors and $M$-$N$ sectors as in Figure \ref{M10-M9}.
Then
$$d(\chi)=1,\quad d(\pi)=9,\quad d(\xi)=d(\eta_1)=d(\eta_2)=10,\quad d(\zeta)=16,$$
$$d(\epsilon)=d(\epsilon_2)=d(\epsilon_3)=\sqrt{10},\quad d(\epsilon_4)=2\sqrt{10},\quad d(\epsilon_0)=8\sqrt{10}.$$
From the graph, we can see that $\pi$, $\chi\pi$, $\chi$, $\zeta$ are self-conjugate, 
$$\{[\overline{\xi}],[\overline{\eta_1}],[\overline{\eta_1}]\}=\{[\xi],[\eta_1],[\eta_2]\},$$
and this with the graph symmetry implies  
$$[\chi^2]=[\id],\quad [\chi\pi]=[\pi\chi], \quad [\chi\zeta]=[\zeta\chi]=[\zeta],\quad  [\chi\xi]=[\xi],\quad $$
$$\{[\chi\eta_1],[\chi\eta_2]\}=\{[\eta_1],[\eta_2]\}.$$
The basic fusion rules coming from the graph are:
$$[\pi\epsilon]=[\epsilon]\oplus[\epsilon_0],\quad [\zeta\epsilon]=2[\epsilon_0],\quad [\xi\epsilon]=[\epsilon_2]\oplus [\epsilon_3]\oplus [\epsilon_0],$$
$$[\eta_1\epsilon]=[\eta_2\epsilon]=[\epsilon_4]\oplus [\epsilon_0],$$
$$[\epsilon\bepsilon]=[\id]\oplus[\pi],\quad [\epsilon_2\bepsilon]=[\epsilon_3\bepsilon]=[\xi],\quad [\epsilon_4\bepsilon]=[\eta_1]\oplus [\eta_2],$$
$$[\epsilon_0\bepsilon]=[\pi]\oplus [\chi\pi]\oplus [\xi]\oplus [\eta_1]\oplus[\eta_2]\oplus 2[\zeta].$$
We denote the last sector by $\Sigma$ for simplicity. 
Then we have $\overline{\Sigma}=\Sigma$, and associativity implies 
$$[\pi^2]=[\id]\oplus \Sigma,\quad [\xi\pi]=[\xi]\oplus \Sigma,\quad [\eta_1\pi]=[\eta_2]\oplus \Sigma,\quad [\eta_2\pi]=[\eta_1]\oplus \Sigma,$$
$$[\zeta\pi]\oplus [\zeta]=2\Sigma. $$
The Frobenius reciprocity implies 
\begin{equation}\label{pi1}
\dim(\bar{\xi}\xi,\pi)=\dim (\overline{\eta_2}\eta_1,\pi)=\dim (\overline{\eta_1}\eta_2,\pi)=2,
\end{equation}
\begin{equation}\label{pi2}
\dim(\overline{\eta_i}\xi,\pi)=\dim(\bar{\xi}\eta_i,\pi)=\dim(\overline{\eta_i}\eta_i,\pi)=1. 
\end{equation}
\begin{equation}\label{pi3}
\dim(\overline{\xi}\zeta,\pi)=\dim(\overline{\eta_i}\zeta,\pi)=2,
\end{equation}
\begin{equation}\label{pi4}
\dim(\overline{\zeta}\zeta,\pi)=3.
\end{equation}

\begin{lemma} With the above notation, we have $[\overline{\xi}]=[\xi]$ and $[\chi\eta_1]=[\eta_1\chi]=[\eta_2]$. 
\end{lemma}

\begin{proof} Note that we have $[\chi\xi]=[\xi]$. 
First we claim $[\xi\chi]=[\xi]$. 
Indeed, assume that it is not the case. 
Then we may assume $[\xi\chi]=[\eta_1]$, which implies 
$$[\chi\eta_1]=[\chi\xi\chi]=[\xi\chi]=[\eta_1],$$
and so $[\chi\eta_2]=[\eta_2]$. 
Since $\{[\overline{\xi}],[\overline{\eta_1}],[\overline{\eta_2}]\}=\{[\xi],[\eta_1],[\eta_2]\},$
we get contradiction, and the claim holds.  
 
Now to prove the statement, it suffices to show $[\eta_2\chi]=[\eta_3]$. 
For this, we assume $[\eta_1\chi]=[\eta_1]$ (and consequently $[\eta_2\chi]=[\eta_2]$), and will deduce contradiction. 
Taking conjugate, we also have $[\chi\eta_1]=[\eta_1]$ and $[\chi\eta_2]=[\eta_2]$ in this case. 
Then since $[\overline{\xi}\xi]$ contains $\pi$ with multiplicity 2 and $[\chi\overline{\xi}]=[\overline{\xi}]$, 
it contains $[\chi\pi]$ with multiplicity 2, and so dimension counting shows 
\begin{equation}\label{bxx}
[\overline{\xi}\xi]=[\id]\oplus [\chi]\oplus 2[\pi]\oplus 2[\chi\pi]\oplus 2[\zeta]\oplus  3\times 10\textrm{ dim},
\end{equation}
where $3\times 10\textrm{ dim}$ means a direct sum of 3 elements from $\{\xi,\eta_1,\eta_2\}$. 
In the same way, we get 
$$[\overline{\eta_1}\eta_1]=[\id]\oplus [\chi]\oplus [\pi]\oplus [\chi\pi]\oplus 80\textrm{ dim},$$
where the last part is decomposed as either $80=5\times 16$ or $80=8\times 10$. 
Also, we get 
$$[\overline{\eta_1}\eta_2]=[\overline{\eta_2}\eta_1]=2[\pi]\oplus 2[\chi\pi]\oplus 4[\zeta].$$
This implies 
\begin{align}\label{eex}
\lefteqn{0=\dim(\overline{\eta_1}\eta_2,\xi)=\dim(\overline{\eta_1}\eta_2,\eta_1)=\dim(\overline{\eta_1}\eta_2,\eta_2)} \\
 &=\dim(\overline{\eta_2}\eta_1,\xi)=\dim(\overline{\eta_2}\eta_1,\eta_1)=\dim(\overline{\eta_2}\eta_1,\eta_2). \nonumber
\end{align}
Also, the Frobenius reciprocity implies 
$$d(\eta_1\zeta,\eta_2)=4.$$
Since $[\zeta\chi]=[\zeta]$, Eq.(\ref{pi3}) shows 
$$\dim(\eta_1\zeta,\chi\pi)=\dim(\eta_1\zeta,\pi\chi)=\dim(\eta_1\zeta,\pi)=2,$$ 
and 
$$[\eta_1\zeta]=2[\pi]\oplus 2[\chi\pi]\oplus 4[\zeta]\oplus 4[\eta_2]\oplus 2\times 10\textrm{ dim}.$$
Since $\dim (\eta_1,\eta_1\zeta)=\dim(\overline{\eta_1}\eta_1,\zeta)$ is either $5$ or $0$, we get 
$$[\eta_1\zeta]=2[\pi]\oplus 2[\chi\pi]\oplus 4[\zeta]\oplus 4[\eta_2]\oplus 2[\xi],$$
and
\begin{equation}\label{bee}
[\overline{\eta_1}\eta_1]=[\id]\oplus [\chi]\oplus [\pi]\oplus [\chi\pi]\oplus 8\times 10\textrm{ dim}.
\end{equation}
A similar reasoning shows 
 \begin{equation}\label{bex}
[\overline{\eta_i}\xi]=[\pi]\oplus [\chi\pi]\oplus 2[\zeta]\oplus 5\times 10 \textrm{ dim},
\end{equation}

For the contragredient map, we have the following 3 possibilities up to relabeling $\eta_1$ and $\eta_2$: 
\begin{itemize}
\item[(i)] $[\overline{\xi}]=[\xi]$, $[\overline{\eta_1}]=[\eta_1]$, $[\overline{\eta_2}]=[\eta_2]$, 
\item[(ii)] $[\overline{\xi}]=[\xi]$, $[\overline{\eta_1}]=[\eta_2]$, $[\overline{\eta_2}]=[\eta_1]$, 
\item[(iii)] $[\overline{\xi}]=[\eta_1]$, $[\overline{\eta_1}]=[\xi]$, $[\overline{\eta_2}]=[\eta_2]$, 
\end{itemize}
However, direct computation shows that there are no fusion rules consistent with Eq.(\ref{bxx}),(\ref{eex}),(\ref{bee}), 
and (\ref{bex}) in each case. 
\end{proof}

\begin{lemma}\label{e2be2} With the above notation, 
$$[\chi\epsilon_2]=[\epsilon_3],$$
$$[\epsilon_2\overline{\epsilon_2}]=[\id]\oplus [\pi],$$
$$[\pi\epsilon_2]=[\epsilon_0]\oplus [\epsilon_2],\quad [\pi\epsilon_3]=[\epsilon_0]\oplus [\epsilon_3],$$
$$[\pi\epsilon_4]=2[\epsilon_0]\oplus [\epsilon_4],$$
$$[\pi\epsilon_0]=[\epsilon]\oplus[\chi\epsilon]\oplus [\epsilon_2]\oplus[\epsilon_3]\oplus 2[\epsilon_4]\oplus 8[\epsilon_0].$$
\end{lemma}

\begin{proof} Since $d(\epsilon_2\overline{\epsilon_2})=10$, and $\epsilon_2\overline{\epsilon_2}$ contains $\id$, we have only 
the following two possibilities:
$$[\epsilon_2\overline{\epsilon_2}]=[\id]\oplus[\pi],$$
$$[\epsilon_2\overline{\epsilon_2}]=[\id]\oplus[\chi\pi].$$
Since $\epsilon_2\overline{\epsilon_2}$ does not contain $\chi$ in any case, we have $[\chi\epsilon_2]\neq [\epsilon_2]$, 
and so $[\chi\epsilon_2]=[\epsilon_3]$. 

Assume that $[\epsilon_2\overline{\epsilon_2}]=[\id]\oplus[\chi\pi]$ holds. 
Then 
$$\dim(\eta_1\epsilon_2,\eta_1\epsilon_2)=\dim(\eta_1,\eta_1\epsilon_2\overline{\epsilon_2})=\dim(\eta_1,\eta_1(\id\oplus\chi\pi))
=1+\dim(\eta_1,\eta_2\pi)=3.$$
Since $d(\eta_1\epsilon_2)=10\sqrt{10}$, we have 
$$[\eta_1\epsilon_2]=[\epsilon_0]\oplus 2\times \sqrt{10} \textrm{ dim}.$$
However, we have 
$$\dim(\eta_1\epsilon_2,\epsilon)=\dim(\eta_1,\epsilon\overline{\epsilon_2})=\dim(\eta_1,\overline{\xi})=\dim(\eta_1,\xi)=0,$$
$$\dim(\eta_1\epsilon_2,\chi\epsilon)=\dim(\eta_2\epsilon_2,\epsilon)=\dim(\eta_2,\epsilon\overline{\epsilon_2})=\dim(\eta_2,\overline{\xi})
=\dim(\eta_2,\xi)=0,$$
$$\dim(\eta_1\epsilon_2,\epsilon_2)=\dim(\eta_1,\id\oplus \chi\pi)=0.$$
$$\dim(\eta_1\epsilon_2,\epsilon_3)=\dim(\eta_1\epsilon_2,\chi\epsilon_2)=\dim(\eta_1,\chi\oplus \pi)=0,$$
and we get contradiction. 
Therefore we get $[\epsilon_2\overline{\epsilon_2}]=[\id]\oplus[\pi]$. 

The Frobenius reciprocity implies $\dim(\pi\epsilon_2,\epsilon_2)=1$. 
Since $d(\pi\epsilon_2)=9\sqrt{10}$, we get $[\pi\epsilon_2]=[\epsilon_2]\oplus [\epsilon_0]$, and  
$[\pi\epsilon_3]=[\epsilon_3]\oplus [\epsilon_0]$ in the same way.  

By associativity,
\begin{align*}
\lefteqn{2[\pi\epsilon_0]=[\pi\zeta\epsilon]=[\overline{\zeta\pi}\epsilon]} \\
 &=[(2\pi\oplus 2\chi\pi\oplus 2\xi\oplus 2\eta_1\oplus 2\eta_2\oplus 3\zeta)\epsilon] \\
 &=2([\epsilon]\oplus [\epsilon_0])\oplus 2([\chi\epsilon]\oplus [\epsilon_0])
 \oplus 2([\epsilon_2]\oplus [\epsilon_3]\oplus [\epsilon_0])\\
 &\oplus 2([\epsilon_4]\oplus [\epsilon_0])\oplus 2([\epsilon_4]\oplus [\epsilon_0]) \oplus 6[\epsilon_0],\\
\end{align*}
which shows the last equation. 
The Frobenius reciprocity together with the equations obtained so far implies the fourth one. 
\end{proof}

\begin{proof}[Proof of Theorem \ref{dualM10-M9}] 
It suffices to show $\cG_{M\supset N}=\cG_{M_{10}>M_9}$ (which is $\widetilde{\cG_{(1^42),1})}$. 
Let $[\bepsilon\epsilon]=[\id]\oplus [\sigma]$ be the irreducible decomposition. 
Since 
$$[\epsilon\bepsilon\epsilon]=[(\id\oplus \pi)\epsilon]=2[\epsilon]\oplus [\epsilon_0],$$
we get $[\epsilon\sigma]=[\epsilon]\oplus [\epsilon_0]$. 

Since 
$$\dim(\bepsilon\chi\epsilon,\bepsilon\chi\epsilon)=\dim (\epsilon\bepsilon\chi,\chi\epsilon\bepsilon)
=\dim (\chi\oplus  \pi\chi,\chi\oplus \chi\pi)=2,$$
the sector $\bepsilon\chi\epsilon$ is decomposed into two distinct irreducibles. 
Since $d(\bepsilon\chi\epsilon)=10$ and 
$$[\epsilon\bepsilon\chi\epsilon]=[(\id\oplus \pi)\chi\epsilon]=[\chi\epsilon]\oplus [\chi][\pi\epsilon]=
2[\chi\epsilon]\oplus [\mu_0],$$
one of the irreducible components of $\bepsilon\chi\epsilon$ is an automorphism of $N$, say $\alpha_1$, 
and the Frobenius reciprocity implies $[\chi\epsilon]=[\epsilon\alpha_1]$. 
Thus 
$$[\bepsilon\chi\epsilon]=[\sigma\alpha_1]\oplus [\alpha_1],$$
and $[\epsilon\sigma\alpha_1]=[\chi\epsilon]\oplus [\epsilon_0]$. 
Since 
$$[\epsilon\sigma\alpha_1]=[(\epsilon\oplus \epsilon_0)\alpha_1],$$
we get $[\epsilon_0][\alpha_1]=[\epsilon_0]$.

In the same way, Lemma \ref{e2be2} implies 
$$\dim(\bepsilon\epsilon_2,\bepsilon\epsilon_2)=(\epsilon\bepsilon,\epsilon_2\overline{\epsilon_2})
=\dim(\id\oplus \pi,\id\oplus \pi)=2,$$
and there exists $\alpha_2\in \Aut(N)$ satisfying $[\epsilon_2]=[\epsilon\alpha_2]$, and 
$$[\bepsilon\epsilon_2]=[\sigma\alpha_2]\oplus [\alpha_2].$$
Letting $[\alpha_3]=[\alpha_1\alpha_2]$, we get 
$$[\epsilon_3]=[\chi\epsilon_2]=[\chi\epsilon\alpha_2]=[\epsilon\alpha_1\alpha_2]=[\epsilon\alpha_3],$$
and 
$$[\bepsilon\epsilon_3]=[\sigma\epsilon_3]\oplus [\alpha_3].$$
Since 
$$[\epsilon\bepsilon\epsilon_2]=[(\id\oplus \pi)\epsilon_2]=2[\epsilon_2]\oplus [\epsilon_0],$$
we get $[\epsilon\sigma\alpha_2]=[\epsilon_2]\oplus [\epsilon_0]$. 
Since 
$$[\epsilon\sigma\alpha_2]=[(\epsilon\oplus\epsilon_0)\alpha_2]=[\epsilon\alpha_2]\oplus [\epsilon_0\alpha_2],$$
we get $[\epsilon_0\alpha_2]=[\epsilon_0]$, 
and  $[\epsilon_0\alpha_3]=[\epsilon_0]$ too. 

Lemma \ref{e2be2} implies
$$\dim(\bepsilon\epsilon_4,\bepsilon\epsilon_4)=\dim(\epsilon_4,\epsilon\bepsilon\epsilon_4)=(\epsilon_4,(\id\oplus \pi)\epsilon_4)
=1+(\epsilon_4,\pi\epsilon_4)=2,$$ 
and $\bepsilon\epsilon_4$ is decomposed into two distinct irreducibles, say $\heta_1$ and $\heta_2$.  
On the other hand, we have
$$[\epsilon\bepsilon\epsilon_4]=[(\id\oplus \pi)\epsilon_4]=2[\epsilon_4]\oplus 2[\epsilon_0].$$
Thus there are the following two possibilities: 
\begin{itemize}
\item[(i)] $[\epsilon\heta_1]=[\epsilon\heta_2]=[\epsilon_4]\oplus [\epsilon_0]$. 
\item[(ii)] $[\epsilon\heta_1]=[\epsilon_4]\oplus 2[\epsilon_0]$ and $[\epsilon\heta_2]=[\epsilon_4]$. 
\end{itemize}

Assume that the case (i) occurs. 
Then $d(\heta_1)=d(\heta_2)=10$. 
Lemma \ref{e2be2} implies 
$$\dim(\bepsilon\epsilon_0,\bepsilon\epsilon_0)=(\epsilon_0,\epsilon\bepsilon\epsilon_0)=1+\dim(\epsilon_0,\pi\epsilon_0)=9.$$
Thus the Frobenius reciprocity together with the fusion rules obtained so far shows that there exists distinct irreducibles 
$\rho_1,\rho_2,\rho_3$  with $d(\rho_1)=d(\rho_2)=d(\rho_3)=8$ satisfying 
$$[\bepsilon\epsilon_0]=\bigoplus_{i=0}^3[\sigma\alpha_i]\oplus 
[\heta_1]\oplus[\heta_2]\oplus [\rho_1]\oplus [\rho_2]\oplus [\rho_3],$$
$$[\epsilon\rho_1]=[\epsilon\rho_2]=[\epsilon\rho_3]=[\epsilon_0],$$
where $\alpha_0=\id$. 
For the fusion category $\cC$ generated by $\bepsilon\epsilon$, we have 
$$\Irr(\cC)=\{[\alpha_i]\}_{i=0}^4\sqcup\{[\sigma\alpha_i]\}_{i=0}^3\sqcup\{[\heta_1],[\heta_2],[\rho_1],[\rho_2],[\rho_3]\}.$$
Let $\Lambda=\{[\alpha_i]\}_{i=0}^4$, which forms a group of order 4. 
Then the $\Lambda$-action on $\{[\rho_1],[\rho_2],[\rho_3]\}$ by left multiplication has a fixed point, 
and we may assume that it is $[\rho_1]$. 
Thus there exists an intermediate subfactor of index 4 between $N\supset \rho_1(N)$, and $\rho_1$ factorizes as 
$\rho_1=\mu_1\mu_2$ with $d(\mu_1)=2$, $d(\mu_2)=4$. 
Since $\overline{\mu_2}\mu_2$ is contained in $\overline{\rho_1}\rho_1$, it belongs to $\cC$. 
However, we have $d(\bar{\mu}\mu)=16$, and $\bar{\mu}\mu$ contains either 1,2 or 4 automorphisms, which is impossible
because $d(\sigma\alpha_i)=9$, $d(\heta_i)=10$, and $d(\rho_i)=8$. 
Therefore (i) never occurs. 

Now we are left with the case (ii). 
In this case, we have $d(\heta_1)=2$, and 
$$[\bepsilon\epsilon_4]=[\bepsilon\epsilon\heta_2]=[(\id\oplus \sigma)\heta_2],$$
implies $[\heta_1]=[\sigma\heta_2]$. 
The Frobenius reciprocity and $\dim(\bepsilon\epsilon_0,\bepsilon\epsilon_0)=9$ imply that 
there exists an irreducible $\rho$ satisfying 
$$[\bepsilon\epsilon_0]=\bigoplus_{i=0}^3[\sigma\alpha_i]\oplus 
2[\sigma\heta_2]\oplus [\rho],$$
$$[\epsilon\rho]=[\epsilon_0],$$
which shows $\cG_{M\supset N}=\cG_{(1^42),1}$. 
\end{proof}
\section{Goldman-type theorems for $(PSL_2(q),PG_1(q))$}

\begin{theorem} Let $M\supset N$ be a finite index subfactor with $\cG_{M\supset N}=\widetilde{\cG_{(1^m),2}}$. 
Then $q=1+2m$ is an odd prime power and there exists a subfactor $R\subset N$ up to inner conjugacy such that 
$R$ is irreducible in $M$ and 
$$M=R\rtimes PSL_2(q)\supset N=R\rtimes \Lambda,$$
where 
$$\Lambda=\{\left(
\begin{array}{cc}
a &b  \\
0 &a^{-1} 
\end{array}
\right)
;\; a\in \F_q^\times,\; b\in \F_q\}/\{\pm 1\}.$$

\begin{figure}[H]
\centering
\begin{tikzpicture}
\draw (-2,1)--(-1,1)--(0,1)--(1,0.5)--(2,0.5);
\draw (-2,-1)--(-1,-1)--(0,-1)--(1,-0.5)--(2,-0.5);
\draw (-2,0)--(-1,0)--(0,0)--(1,0.5)--(0,-1);
\draw (0,0)--(1,-0.5)--(0,1);
\draw(-2,1)node{$*$};
\draw(-1,1)node{$\bullet$};
\draw(0,1)node{$\bullet$};
\draw(1,0.5)node{$\bullet$};
\draw(2,0.5)node{$\bullet$};
\draw(-2,-1)node{$\bullet$};
\draw(-1,-1)node{$\bullet$};
\draw(0,-1)node{$\bullet$};
\draw(1,-0.5)node{$\bullet$};
\draw(2,-0.5)node{$\bullet$};
\draw(-2,0)node{$\bullet$};
\draw(-1,0)node{$\bullet$};
\draw(0,0)node{$\bullet$};
\draw(-2,1)node[above]{$\id_N$};
\draw(-1,1)node[above]{$\epsilon$};
\draw(0,1)node[above]{$\sigma$};
\draw(1,0.5)node[above]{$\epsilon\rho_1'$};
\draw(2,0.5)node[above]{$\rho_1'$};
\draw(-2,-1)node[below]{$\alpha'_2$};
\draw(-1,-1)node[below]{$\epsilon\alpha'_2$};
\draw(0,-1)node[below]{$\sigma\alpha'_2$};
\draw(1,-0.5)node[below]{$\epsilon\rho'_2$};
\draw(2,-0.5)node[below]{$\rho'_2$};
\draw(-2,0)node[above]{$\alpha'_1$};
\draw(-1,0)node[above]{$\epsilon\alpha'_1$};
\draw(0,0)node[above]{$\sigma\alpha'_1$};
\end{tikzpicture}
\caption{$\widetilde{\cG_{(1^3),2}}$}
\end{figure}

\end{theorem}

\begin{proof} Note that if $m=1$, we have $\widetilde{\cG_{(1),2}}=E_6^{(1)}=\cG_{(1^3),1}$, 
and the statement follows from \cite{H95} (or Theorem \ref{GthF}) as we have $(\fA_4,X_4)\cong (PSL_2(3),PG_1(3))$. 
We assume $m> 1$ in what follows. 

Let $\epsilon:N\hookrightarrow M$ be the inclusion map, and let $[\bepsilon\epsilon]=[\id]\oplus[\sigma]$ be
the irreducible decomposition. 
Let $\cC$ be the fusion category generated by $\bepsilon\epsilon$, and let $I$ be the group of (the equivalence classes of) 
the invertible objects in $\cC$.  
Then $|I|=m$. 

We can make the following parametrization of irreducible $N$-$N$ and $M$-$N$ sectors respectively:
$$\{\alpha'_i\}_{i\in I}\sqcup \{\sigma\alpha'_i\}_{i\in I}\sqcup\{\rho'_1,\rho'_2\},$$
$$\{\epsilon\alpha'_i\}_{i\in I}\sqcup \{\epsilon\rho'_1,\epsilon\rho'_2\},$$
with properties: 
$$d(\alpha'_i)=1,\quad d(\epsilon)=\sqrt{2+2m},\quad d(\sigma)=1+2m,\quad d(\rho'_1)=d(\rho'_2)=m,$$
$$[\bepsilon\epsilon]=[\id]\oplus[\sigma],$$
$$[\alpha'_{i_1}\alpha'_{i_2}]=[\alpha'_{i_1i_2}],$$
$$[\epsilon\sigma]=[\epsilon]\oplus [\epsilon\rho_1]\oplus [\epsilon\rho_2],$$
\begin{equation}\label{sr12}
[\sigma\rho'_1]=[\sigma\rho'_2]=\bigoplus_{i\in I}[\sigma\alpha'_i],
\end{equation}
\begin{equation}\label{ssr12}
[\sigma^2]=[\id]\oplus [\rho'_1]\oplus [\rho'_2]\oplus 2\bigoplus_{i\in I}[\sigma\alpha'_i].
\end{equation}
By definition of $I$, we have $[\overline{\alpha'_i}]=[\alpha'_{i^{-1}}]$. 
We can introduce another involution in $I$ by $[\overline{(\sigma\alpha'_i)}]=[\sigma\alpha'_{i^*}]$. 
We also introduce an involution in $\{1,2\}$ by $[\overline{\rho'_j}]=[\rho'_{\bar{j}}]$. 
Taking conjugation of Eq.(\ref{sr12}), we also have 
$$[\rho'_1\sigma]=[\rho'_2\sigma]=\bigoplus_{i\in I}[\sigma\alpha'_i].$$

We claim that there exists a fusion subcategory $\cC_1$ of $\cC$ satisfying 
$$\Irr(\cC_1)=\{\alpha'_i\}_{i\in I}\sqcup\{\rho_1,\rho_2\}.$$
Indeed, let 
$$I_j=\{i\in I;\;[\alpha'_i][\rho'_j]=[\rho'_j]\},$$
$$I_j'=\{i\in I;\;[\rho'_j][\alpha'_i]=[\rho'_j]\}.$$
Since the group $I$ acts on the 2 point set $\{[\rho'_1],[\rho'_2]\}$ by left (and also right) multiplication, 
we have the following two cases. 
\begin{itemize}
\item[(i)] $I_1=I_2=I$. In this case, we also have $I'_1=I'_2=I$ as $\{[\overline{\rho'_1}],\{[\overline{\rho'_2}]\}\}
=\{[\rho'_1],[\rho'_2]\}$.  
\item[(ii)] $|I_1|=|I_2|=m/2$. In this case, we also have $|I'_1|=|I'_2|=m/2$. 
\end{itemize}
Assume (i) occurs first. 
Then the Fobenius reciprocity implies 
$$[\rho'_j\overline{\rho'_j}]=\bigoplus_{i\in I}[\alpha'_i]\oplus a_{j1}[\rho'_1]\oplus a_{j2}[\rho'_2]\oplus 
\bigoplus_{i\in I}b_{ji}[\sigma\alpha'_i].$$
Let 
$$b_j=\sum_{i\in I}b_{ji}.$$
Then
$$m^2=m+(a_{j1}+a_{j2})m+b_j(2m+1),$$
and we see that $m$ divides $b_j$. 
If $b_j\geq m$, we would have $m\geq 2m+1$, which is contradiction. 
Thus $b_{ji}=0$ for all $i,j$. 
The Frobenius reciprocity shows that neither$[\rho'_1\overline{\rho'_2}]$ nor $[\rho'_2\overline{\rho'_1}]$ contain 
any automorphism, and a similar argument as above shows that $\rho'_1\overline{\rho'_2}$ and $\rho'_2\overline{\rho'_1}$ 
are also direct sums of sectors in $\{\alpha'_i\}\sqcup\{\rho'_1,\rho'_2\}$. 
This proves the claim in the case (i). 

Assume (ii) occurs now. 
Then $l=m/2$ is a natural number. 
A similar argument as above shows that for
$$a_j=\dim (\rho'_j\overline{\rho'_j},\rho'_1)+\dim (\rho'_j\overline{\rho'_j},\rho'_2),$$
$$b_j=\sum_{i\in I}\dim(\rho'_j\overline{\rho'_j},\sigma\alpha'_i),$$
we have 
$$4l^2=l+2a_jl+b_j(4l+1).$$ 
This shows that $l$ divides $b_j$, and so $b_j=0$. 
Note that there exists $i_0\in I$ satisfying $[\rho'_1]=[\alpha'_{i_0}\rho'_2]$, which implies 
$$[\rho'_1\overline{\rho'_2}]=[\rho'_1\overline{\rho'_1}{\alpha'_{i_0}}],\quad 
[\rho'_2\overline{\rho'_1}]=[\alpha'_{i_0^{-1}}\rho'_1\overline{\rho'_1}].$$
Therefore $\rho'_{j_1}\overline{\rho'_{j_2}}$, $1\leq j_1,j_2\leq 2$ are direct sums of sectors in 
$\{\alpha'_i\}\sqcup\{\rho'_1,\rho'_2\}$, 
which shows the claim in the case (ii). 

The rest of the proof is very much similar to that of Theorem \ref{LM}, and we briefly address it 
except for the last part deciding the group structure of $\Gamma$. 
Theorem \ref{intermediate} and Eq.(\ref{ssr12}) show that there exists a unique intermediate subfactor 
$P$ between $N$ and $\sigma(N)$ such that if we $\iota:P\hookrightarrow N$ denotes the inclusion map, we have 
$$[\iota\biota]=[\id]\oplus [\rho_1]\oplus [\rho_2].$$
Moreover, there exists $\tau\in \Aut(P)$ satisfying $[\sigma]=[\iota\tau\biota]$. 
The fusion rules of $\cC_1$ tell that the dual principal graph $\cG^d_{N\supset M}$ is $\cG_{(1^m),2}$, 
and Theorem \ref{GthF} shows that $\cG_{M\supset N}$ is also $\cG_{(1^m),2}$. 
The group $I$ is the cyclic group $\Z_m$ now. 
Let $\cC_2$ be the fusion category generated by $\biota\iota$. 
Then we can parametrize $\Irr(\cC_2)$ so that 
$$\Irr(\cC_2)=\{[\alpha_i]\}_{i\in I}\sqcup \{[\rho_1],[\rho_2]\},$$
$$[\iota\alpha_i]=[\alpha'_i\iota],$$
$$[\biota\iota]=[\id]\oplus [\rho_1]\oplus [\rho_2].$$

Applying Theorem \ref{GthF}, we see that there exists a unique subfactor $R\subset P$, up to inner conjugacy, that is 
irreducible in $M$ such that there exists a primitive Frobenius group $K\rtimes H$ with $|H|=m$, $|K|=1+2m$ and 
an outer action $\beta$ of it on $R$ satisfying 
$$N=R\rtimes_\beta(K\rtimes H)\supset P=R\rtimes_\beta H.$$
Note that the number $q=1+2m$ is an odd prime power $p^k$ and $K=\Z_p^k$, $H=\Z_m$. 
Moreover, there exists a group $\Gamma$ including $K\rtimes H$ such that $\beta$ extends to an outer action $\gamma$ 
of $\Gamma$ satisfying $M=R\rtimes_\gamma \Gamma$. 

From the graph $\cG_{M\supset N}$, we can see that the $\Gamma$-action on $\Gamma/(K\rtimes H)$ is a 2-transitive, 
but not 3-transitive, extension of the Frobenius group $K\rtimes H$ acting on $(K\rtimes H)/H$. 
Note that $|\Gamma|=(2m+2)(2m+1)m$. 
Thus \cite[Chapter XI, Theorem 1.1]{HB} shows that $\Gamma$ is a Zassenhaus group. 
The order of $\Gamma$ shows that it is not one of the Suzuki groups.  
Since $\Gamma$ is not 3-transitive, we conclude from \cite[Chapter XI, Theorem 11.16]{HB} that $\Gamma=PSL_2(q)$. 
\end{proof}
\section{Goldman-type theorems for sharply 4-transitive permutation groups}
Since the finite depth subfactors of index 5 are completely classified in \cite{IMPP15}, 
the only point of the following theorem is to see how to find a subfactor $R$ and an $\fS_5$-action on it from the principal graph.  

\begin{theorem}\label{GTS5}
Let $L\supset M$ be a finite index inclusion of factors with $\cG_{L\supset M}=\cG_{(\fS_5,X_5)}$. 
Then there exists a unique subfactor $R\subset M$, up to inner conjugacy, such that $R'\cap L=\C$ and there exists 
an outer action $\gamma$ of $\fS_5$ on $R$ satisfying 
$$L=R\rtimes_\gamma\fS_5\supset M=R\rtimes_\gamma\fS_4.$$ 
\begin{figure}[H]
\centering
\begin{tikzpicture}
\draw (-4,0)--(-3,0)--(-2,0)--(-1,0)--(0,0)--(1,0)--(2,0)--(3,0)--(4,0);
\draw (-1,1)--(-1,0);
\draw (1,1)--(1,0);
\draw (0,2)--(0,1)--(0,0);
\draw(-4,0)node{$*$};
\draw(-3,0)node{$\bullet$};
\draw(-2,0)node{$\bullet$};
\draw(-1,0)node{$\bullet$};
\draw(0,0)node{$\bullet$};
\draw(1,0)node{$\bullet$};
\draw(2,0)node{$\bullet$};
\draw(3,0)node{$\bullet$};
\draw(4,0)node{$\bullet$};
\draw(-1,1)node{$\bullet$};
\draw(1,1)node{$\bullet$};
\draw(0,2)node{$\bullet$};
\draw(0,1)node{$\bullet$};
\draw(-4,0)node[below]{$\id_M$};
\draw(-3,0)node[below]{$\delta$};
\draw(-2,0)node[below]{$\lambda$};
\draw(-1,0)node[below]{$\delta\pi$};
\draw(-1,1)node[above]{$\pi$};
\draw(4,0)node[below]{$\chi$};
\draw(3,0)node[below]{$\delta\chi$};
\draw(2,0)node[below]{$\lambda\chi$};
\draw(1,0)node[below]{$\delta\pi\chi$};
\draw(1,1)node[above]{$\pi\chi$};
\draw(0,0)node[below]{$\lambda\xi$};
\draw(0,1)node[right]{$\delta\xi$};
\draw(0,2)node[above]{$\xi$};
\end{tikzpicture}
\caption{\label{S5-S4}$\cG_{(\fS_5,X_5)}$}
\end{figure}
\end{theorem}

\begin{proof} We follow the strategy described in Subsection 2.5.

(1) Let $\delta:M\hookrightarrow L$ be the inclusion map, and let $[\bdelta\delta]=[\id]\oplus [\lambda]$ 
be the irreducible decomposition. 
We parametrize the irreducible $M$-$M$ sectors and the $L$-$M$ sectors generated by $\delta$ as in Figure \ref{S5-S4}. 
Then we have  
$$d(\lambda)=4,\quad d(\pi)=3,\quad d(\xi)=2,\quad d(\chi)=1,\quad d(\delta)=\sqrt{5}.$$
From the graph, we can see that all the $M$-$M$ sectors are self-conjugate, which implies $[\chi\lambda]=[\lambda\chi]$, 
$[\chi\pi]=[\pi\chi]$. 
The graph symmetry implies $[\xi\chi]=[\xi]$, and since $\xi$ is self-conjugate, we get 
$$[\xi^2]=[\id]\oplus [\chi]\oplus[\xi]$$
by dimension counting. 

The basic fusion rules coming from the graph are 
\begin{equation}\label{fuS5}
[\lambda^2]=[\id]\oplus [\lambda]\oplus [\pi]\oplus [\lambda\xi],
\end{equation}
$$[\lambda\pi]=[\lambda]\oplus [\lambda\xi],$$
$$[\lambda(\lambda\xi)]=[\lambda]\oplus [\lambda\chi]\oplus [\pi]\oplus[\pi\chi]\oplus[\xi]\oplus 2[\lambda\xi].$$
Taking conjugate, we also have 
$$[\pi\lambda]=[\lambda]\oplus [\lambda\xi].$$

Now direct computation using the Frobenius reciprocity and associativity shows the following fusion rules: 
$$[\pi^2]=[\id]\oplus [\pi]\oplus [\pi\chi]\oplus [\xi]$$
$$[\pi\xi]=[\xi\pi]=[\pi]\oplus[\pi\chi].$$

Let $\cC$ be the fusion category generated by $\bdelta\delta$. 
Then the above fusion rules show that there exists a fusion subcategory $\cC_1$ of $\cC$ with 
$$\Irr(\cC_1)=\{\id,\chi,\xi,\pi,\pi\chi\}.$$

(2) Theorem \ref{intermediate} and Eq.(\ref{fuS5}) imply that there exists a unique intermediate subfactor $N$ between 
$M$ and $\lambda(M)$ such that if $\epsilon:N\hookrightarrow M$ is the inclusion map, we have 
$$[\epsilon\bepsilon]=[\id]\oplus [\pi].$$ 
In the same way as in the proof of Lemma \ref{factorization}, there exists $\varphi\in \Aut(N)$ satisfying 
$[\lambda]=[\epsilon\varphi\bepsilon]$. 

(3) Note that we have $[M:N]=1+d(\pi)=4$. 
Thanks to the classification of subfactors of index 4 (see \cite[Subsection 3.2]{JMS14}) and $\Irr(\cC_1)$, we can see that 
$\cG_{M\supset N}^d$ is the Coxeter graph $E_7^{(1)}$, and so is $\cG_{M\supset N}$ too. 
Note that we have $E_7^{(1)}=\widetilde{\cG_{(1^2),1}}$, and $(L(3),PG_1(3))\cong (\fS_4,X_4)$. 
Let $\cC_2$ be the fusion category generated by $\bepsilon\epsilon$. 
As in Theorem \ref{LM}, we can parametrize $\Irr(\cC_2)$ as 
$$\Irr(\cC_2)=\{\id,\alpha',\rho',\sigma,\sigma\alpha'\},$$
with the following properties: 
$$d(\alpha')=1,\quad d(\rho')=2,\quad d(\sigma)=3,$$
$$[{\alpha'}^2]=[\id],$$
$$[\alpha'\rho]=[\rho'\alpha']=[\rho'],$$
$$[{\rho'}^2]=[\id]\oplus [\alpha']\oplus [\rho'],$$
$$[\sigma^2]=[\id]\oplus[\sigma]\oplus[\rho']\oplus [\sigma\alpha'],$$
$$[\alpha'\sigma]=[\sigma\alpha'],$$
$$[\sigma\rho']=[\rho'\sigma]=[\sigma]\oplus [\sigma\alpha'],$$
$$[\bepsilon\epsilon]=[\id]\oplus [\sigma].$$

(4)
Theorem \ref{LM} shows that there exists a unique subfactor $R\subset N$, up to inner conjugacy such that 
$R'\cap M=\C$ and there exists an outer action $\beta$ of $\fS_4$ on $R$ satisfying 
$$M=R\rtimes_\beta \fS_4\supset N=R\rtimes_\beta \fS_3.$$ 
To use notation consistent with that in Theorem \ref{GthF} and Theorem \ref{LM}, we let $P=R\rtimes_\beta\fS_3\subset N$ 
and we let $\iota:P\hookrightarrow N$ and $\kappa:R\hookrightarrow P$ be the inclusion maps. 
Let $\epsilon_1=\epsilon\iota\kappa$. 
Then $\epsilon_1\overline{\epsilon_1}$ corresponds to the regular representation of $\fS_4$, and
$$[\epsilon_1\overline{\epsilon_1}]=[\id]\oplus [\chi]\oplus 2[\xi]\oplus 3[\pi]\oplus 3[\pi\chi].$$
Thus since $[\bdelta\delta]=[\id]\oplus [\lambda]$,
$$\dim(\delta\epsilon_1,\delta\epsilon_1)=\dim(\bdelta\delta,\epsilon_1\overline{\epsilon_1})=1,$$
and $L\supset R$ is irreducible. 

(5) Note that we have $[L:R]=120$. 
On the other hand, 
$$\dim(\delta\epsilon_1\overline{(\delta\epsilon_1)},\delta\epsilon_1\overline{(\delta\epsilon_1)})
=\dim(\bdelta\delta\epsilon_1\overline{\epsilon_1},\epsilon_1\overline{\epsilon_1}\bdelta\delta).$$
Note that $[\lambda]$ commutes with $[\epsilon_1\overline{\epsilon_1}]$, and $[(\epsilon_1\overline{\epsilon_1})^2]
=|\fS_4|[\epsilon_1\overline{\epsilon_1}]$. 
Thus
\begin{align*}
\lefteqn{\dim(\bdelta\delta\epsilon_1\overline{\epsilon_1},\epsilon_1\overline{\epsilon_1}\bdelta\delta)
=\dim(\bdelta\delta\epsilon_1\overline{\epsilon_1},\bdelta\delta\epsilon_1\overline{\epsilon_1})} \\
 &= \dim((\bdelta\delta)^2,(\epsilon_1\overline{\epsilon_1})^2)=24\dim((\id\oplus \lambda)^2,\epsilon_1\overline{\epsilon_1})=120.
\end{align*}
Thus the inclusion $L\supset R$ is of depth 2. 

(6) We denote $\iota_3=\iota\kappa$. 
By Lemma \ref{inv1}, we get
$$\dim(\varphi\bepsilon\epsilon\iota_3\overline{\iota_3}\varphi^{-1},\bepsilon\epsilon\iota_3\overline{\iota_3})=|\fS_3|=6.$$
Note that $\iota_3\overline{\iota_3}$ corresponds to the regular representation in $\Rep(\fS_3)$, and 
$$[\iota_3\overline{\iota_3}]=[\id]\oplus [\alpha']\oplus 2[\rho'].$$
Thus 
$$[\bepsilon\epsilon\iota_3\overline{\iota_3}]=[(\id\oplus \sigma)(\id\oplus \alpha'\oplus 2\rho')]=[\id]\oplus[\alpha']\oplus 2[\rho']
\oplus 3[\sigma]\oplus 3[\sigma\alpha'].$$
Dimension counting implies 
$$\dim(\varphi(\id\oplus \alpha'\oplus 2\rho')\varphi^{-1},\id\oplus \alpha'\oplus 2\rho')=6,$$
and $[\varphi\iota_3\overline{\iota_3}\varphi^{-1}]=[\iota_3\overline{\iota_3}]$. 

Now we can apply Lemma \ref{inv2} to $\fS_3$, and we obtain $\varphi_1\in \Aut(R)$ satisfying $[\varphi\epsilon_1]=[\epsilon_1\varphi_1]$. 
Lemma \ref{finishing} implies that there exists a group $\Gamma$ including $\fS_4$ such that $\beta$ extends to an outer action 
$\gamma$ of $\Gamma$ satisfying $L=R\rtimes_\gamma\Gamma$. 
Note that $|\Gamma|=[L:R]=120$. 
Since the graph $\cG_{(\fS_5,X_5)}$ shows that the $\Gamma$-action on $\Gamma/\fS_4$ 
is a 3-transitive extension of $(\fS_4,X_4)$, we conclude $\Gamma=\fS_5$. 
\end{proof}

The remaining two cases are the most subtle because we cannot apply Lemma \ref{inv2} to either 
$\fA_4=H(2^2)=\Z_2^2\rtimes \Z_3$ or $M_9=S(3^2)=\Z_3^3\rtimes Q_8$ in the step (6). 

Since $\cG_{(\fA_6,X_6)}=\widetilde{\cG_{\fA_4}^{\fA_5}}$, we can obtain it from the induction-reduction graph 
$\cG_{\fA_4}^{\fA_5}$ between $\fA_5$ and $\fA_4$ (see, for example, \cite{IMPP15} for the latter). 

\begin{theorem}\label{GTA6}
Let $L\supset M$ be a finite index inclusion of factors with $\cG_{L\supset M}=\cG_{(\fA_6,X_6)}$. 
Then there exists a unique subfactor $R\subset M$, up to inner conjugacy, such that $R'\cap L=\C$ and there exists 
an outer action $\gamma$ of $\fA_6$ on $R$ satisfying 
$$L=R\rtimes_\gamma\fA_6\supset M=R\rtimes_\gamma\fA_5.$$ 
\begin{figure}[H]
\centering
\begin{tikzpicture}
\draw (-4,0)--(-3,0)--(-2,0)--(-1,0)--(0,0)--(1,0.5)--(2,1);
\draw (0,0)--(1,-0.5)--(2,-1);
\draw (0,0)--(0,1)--(-1,2);
\draw (0,1)--(1,2);
\draw (0,1)--(0,2);
\draw (-1,0)--(-1,1);
\draw(-4,0)node{$*$};
\draw(-3,0)node{$\bullet$};
\draw(-2,0)node{$\bullet$};
\draw(-1,0)node{$\bullet$};
\draw(0,0)node{$\bullet$};
\draw(1,0.5)node{$\bullet$};
\draw(2,1)node{$\bullet$};
\draw(1,-0.5)node{$\bullet$};
\draw(2,-1)node{$\bullet$};
\draw(0,1)node{$\bullet$};
\draw(-1,1)node{$\bullet$};
\draw(1,2)node{$\bullet$};
\draw(0,2)node{$\bullet$};
\draw(-1,2)node{$\bullet$};
\draw(-4,0)node[below]{$\id_M$};
\draw(-3,0)node[below]{$\delta$};
\draw(-2,0)node[below]{$\lambda$};
\draw(-1,0)node[below]{$\delta\pi$};
\draw(-1,1)node[above]{$\pi$};
\draw(0,0)node[below]{$\mu$};
\draw(0,1)node[right]{$\delta\xi_1$};
\draw(-1,2)node[above]{$\xi_1$};
\draw(0,2)node[above]{$\xi_2$};
\draw(1,2)node[above]{$\xi_3$};
\draw(1,0.5)node[below]{$\delta\eta_1$};
\draw(1,-0.5)node[below]{$\delta\eta_2$};
\draw(2,1)node[right]{$\eta_1$};
\draw(2,-1)node[right]{$\eta_2$};
\end{tikzpicture}
\caption{\label{A6-A5} $\cG_{(\fA_6,X_6)}$}
\end{figure}
\end{theorem}

\begin{proof}
(1) Let $\delta:M\hookrightarrow L$ be the inclusion map, and let $[\bdelta\delta]=[\id]\oplus [\lambda]$ 
be the irreducible decomposition. 
We parametrize the irreducible $M$-$M$ sectors and the $L$-$M$ sectors generated by $\delta$ as in Figure \ref{A6-A5}. 
Then we have  
$$d(\lambda)=d(\xi_1)=d(\xi_2)=d(\xi_3)=5,\quad d(\pi)=4,\quad d(\mu)=15,\quad d(\eta_1)=d(\eta_2)=3,$$
$$d(\delta)=\sqrt{6}.$$
From the graph, we can see that $\lambda$, $\pi$ and $\mu$ are self-conjugate, and 
$$\{[\overline{\xi_1}],[\overline{\xi_2}],[\overline{\xi_3}]\}=\{[\xi_1],[\xi_2],[\xi_3]\},\quad 
\{[\overline{\eta_1}],[\overline{\eta_2}]\}=\{[\eta_1],[\eta_2]\}.$$
We use the notation $[\overline{\xi_i}]=[\xi_{\bar{i}}]$ and $[\overline{\eta_j}]=[\eta_{\bar{j}}]$ for simplicity. 

The basic fusion rules coming from the graph and their conjugate are 
\begin{equation}\label{A1}
[\lambda^2]=[\id]\oplus [\lambda]\oplus [\pi]\oplus [\mu],
\end{equation}
\begin{equation}\label{A2}
[\lambda\pi]=[\pi\lambda]=[\lambda]\oplus [\mu],
\end{equation}
\begin{equation}\label{A3}
[\lambda\mu]=[\mu\lambda]=[\lambda]\oplus [\pi]\oplus[\xi_1]\oplus[\xi_2]\oplus [\xi_3]\oplus [\eta_1]\oplus[\eta_2]\oplus 3[\mu],
\end{equation}
\begin{equation}\label{A4}
[\lambda\xi_i]+[\xi_i]=[\xi_i\lambda]\oplus [\xi_i]=[\xi_1]\oplus [\xi_2]\oplus [\xi_3]\oplus [\mu], 
\end{equation}
\begin{equation}\label{A5}
[\lambda\eta_i]=[\eta_i\lambda]=[\mu]. 
\end{equation}
  
By associativity, we get 
\begin{equation}\label{A6}
[\pi^2]\oplus [\mu\pi]=[\id]\oplus [\lambda]\oplus [\pi]\oplus[\xi_1]\oplus[\xi_2]\oplus [\xi_3]\oplus [\eta_1]\oplus[\eta_2]\oplus 3[\mu],
\end{equation}
\begin{equation}\label{A7}
[\pi\mu]\oplus[\mu^2]=[\id]\oplus 4[\lambda]\oplus 3[\pi]\oplus 4[\xi_1]\oplus 4[\xi_2]\oplus 4[\xi_3]
\oplus 2[\eta_1]\oplus 2[\eta_2]\oplus 12[\mu],
\end{equation}
\begin{equation}\label{A8}
[\pi\xi_i]\oplus[\mu\xi_i]=[\xi_i]\oplus [\lambda]\oplus [\pi]\oplus[\xi_1]\oplus[\xi_2]\oplus [\xi_3]\oplus 
[\eta_1]\oplus[\eta_2]\oplus 4[\mu].
\end{equation}
\begin{equation}\label{A9}
[\eta_i]\oplus [\pi\eta_i]\oplus [\mu\eta_i]=[\lambda]\oplus [\pi]\oplus[\xi_1]\oplus[\xi_2]\oplus [\xi_3]\oplus 
[\eta_1]\oplus[\eta_2]\oplus 2[\mu].
\end{equation}

Eq.(\ref{A2}) shows 
$$1=\dim(\lambda\pi,\mu)=\dim(\lambda,\mu\pi).$$
Since $d(\eta_i\pi)<d(\mu)$, we have 
$$0=\dim(\eta_i\pi,\mu)=\dim (\eta_i,\mu\pi).$$
Eq.(\ref{A6}) shows that $\pi^2$ contains $\id$, $\eta_1$,$\eta_2$, and it cannot contain $\mu$ by dimension counting, 
which implies $\dim(\pi,\mu\pi)=0$ by the Frobenius reciprocity. 
Eq.(\ref{A6}) again shows that $\mu\pi$ contains $\mu$ with multiplicity 3 and $\pi^2$ contains $\pi$ with multiplicity 1. 
Thus we get 
$$[\pi^2]=[\id]\oplus [\pi]\oplus [\eta_1]\oplus [\eta_2]\oplus 5\textrm{ dim},\quad
[\mu\pi]=[3\mu]\oplus [\lambda]\oplus 10\textrm{ dim},$$
where the remainder is $\xi_1\oplus \xi_2\oplus \xi_3$. 
Therefore we may and do assume that $\pi^2$ contains $\xi_1$, and we get 
\begin{equation}
\label{A10}
[\pi^2]=[\id]\oplus [\pi]\oplus [\xi_1]\oplus [\eta_1]\oplus [\eta_2],
\end{equation}
\begin{equation}\label{A11}
[\mu\pi]=[3\mu]\oplus [\lambda]\oplus [\xi_2]\oplus[\xi_3]. 
\end{equation}
In consequence $\xi_1$ is self-conjugate. 
Taking conjugate of Eq.(\ref{A11}), we also get $[\mu\pi]=[\pi\mu]$, and Eq.(\ref{A7}) implies 
\begin{equation}\label{A12}
[\mu^2]=[\id]\oplus 3[\lambda]\oplus 3[\pi]\oplus 4[\xi_1]\oplus 3[\xi_2]\oplus 3[\xi_3]
\oplus 2[\eta_1]\oplus 2[\eta_2]\oplus 9[\mu].
\end{equation}

The Frobenius reciprocity implies 
$$[\pi\xi_1]=[\pi]\oplus 16\textrm{ dim},\quad
[\mu\xi_1]=4[\mu]\oplus [\lambda]\oplus 10\textrm{ dim},$$
and Eq.(\ref{A8}) with dimension counting implies 
$$[\pi\xi_1]=[\pi]\oplus [\eta_1]\oplus [\eta_2]\oplus 10\textrm{ dim},\quad
[\mu\xi_1]=4[\mu]\oplus [\lambda]\oplus 10\textrm{ dim},$$
where the remainder is $2[\xi_1]\oplus [\xi_2]\oplus[\xi_3].$

For $i=2,3$, Eq.(\ref{A8}) and (\ref{A12}) show that we have 
$$3=\dim(\xi_i,\mu^2)=\dim(\mu\xi_i,\mu),$$ 
and $\mu\xi_i$ contains $\mu$ with multiplicity 3, while it does not contain $\pi$ as 
$$0=\dim(\pi^2,\xi_i)=\dim(\pi,\pi\xi_i).$$ 
Thus
$$[\pi\xi_i]=[\mu]\oplus 5 \textrm{ dim},\quad
[\mu\xi_i]=[\lambda]\oplus [\pi]\oplus 3[\mu] \oplus [\eta_1]\oplus [\eta_2]\oplus 15 \textrm{ dim},$$
where the remainder is $[\xi_i]\oplus [\xi_1]\oplus [\xi_2]\oplus[\xi_3].$
If $\mu\xi_i$ contained $\xi_i$ with multiplicity 2, the Frobenius reciprocity implies that 
$\xi_i\overline{\xi_i}$ would contain $\mu$ with multiplicity 2, which is impossible. 
Thus we get 
\begin{equation}\label{A15}
[\pi\xi_i]=[\mu]\oplus [\xi_i],\quad i=2,3,
\end{equation}\label{A16}
\begin{equation}
[\mu\xi_i]=[\lambda]\oplus [\pi]\oplus [\xi_1]\oplus [\xi_2]\oplus [\xi_3]\oplus[\eta_1]\oplus [\eta_2]\oplus 3[\mu],\quad i=2,3. 
\end{equation}
Eq.(\ref{A15}) shows 
$$0=\dim(\pi\xi_i,\xi_1)=\dim(\xi_i,\pi\xi_1),\quad i=2,3.$$
Thus 
\begin{equation}\label{A13}
[\pi\xi_1]=[\pi]\oplus [\eta_1]\oplus [\eta_2]\oplus 2[\xi_1],
\end{equation}
\begin{equation}\label{A14}
[\mu\xi_1]=4[\mu]\oplus [\lambda]\oplus [\xi_1]\oplus [\xi_2]. 
\end{equation}

The Frobenius reciprocity together with the fusion rules obtained so far implies 
\begin{equation}
[\pi\eta_1]=[\pi]\oplus [\xi_1]\oplus[\eta_2],
\end{equation}
\begin{equation}
[\pi\eta_2]=[\pi]\oplus [\xi_1]\oplus[\eta_1],
\end{equation}
\begin{equation}
[\mu\eta_i]=2[\mu]\oplus [\lambda]\oplus [\xi_2]\oplus [\xi_3]. 
\end{equation}

Let $\cC$ be the fusion category generated by $\bdelta\delta$. 
Then the above computation shows that the fusion subcategory $\cC_1$ of $\cC$ generated by $\pi$ satisfies 
$$\Irr(\cC_1)=\{\id,\pi,\xi_1,\eta_1,\eta_2\}.$$

(2) Theorem \ref{intermediate} and Eq.(\ref{A1}) imply that there exists a unique intermediate subfactor $N$ between 
$M$ and $\lambda(M)$ such that if $\epsilon:N\hookrightarrow M$ is the inclusion map, we have 
$$[\epsilon\bepsilon]=[\id]\oplus [\pi].$$ 
Note that we have $d(\epsilon)=\sqrt{5}$. 
In the same way as in the proof of Lemma \ref{factorization}, there exists $\varphi\in \Aut(N)$ satisfying 
$[\lambda]=[\epsilon\varphi\bepsilon]$. 

(3) Since 
$$\dim(\pi\epsilon,\pi\epsilon)=\dim(\pi^2,\epsilon\bepsilon)=\dim(\pi^2,\id\oplus \pi)=2,$$
there exists an irreducible sector $\epsilon'$ with $[\pi\epsilon]=[\epsilon]\oplus [\epsilon']$ 
and $d(\epsilon')=3\sqrt{5}$. 
Since 
$$[\pi\epsilon\bepsilon]=[\pi(\id\oplus \pi)]=[\id]\oplus 2[\pi]\oplus [\xi_1]\oplus [\eta_1]\oplus[\eta_2],$$
we get 
$$[\epsilon'\bepsilon]=[\pi]\oplus [\xi_1]\oplus [\eta_1]\oplus [\eta_2].$$
The Frobenius reciprocity and dimension counting show $[\eta_1\epsilon]=[\eta_2\epsilon]=[\epsilon']$. 
Since $\xi_1$ is self-conjugate,
$$\dim(\xi_1\epsilon,\xi_1\epsilon)=\dim(\xi_1,\xi_1\epsilon\bepsilon)=\dim(\xi_1,\xi_1(\id\oplus \pi))=1+\dim(\xi_1,\xi_1\pi)
=1+\dim (\xi_1,\pi\xi_1),$$
and Eq.(\ref{A13}) shows $\dim(\xi_1\epsilon,\xi_1\epsilon)=3$. 
This together with the Frobenius reciprocity imply that there exist irreducible sectors $\epsilon''$ 
and $\epsilon'''$ satisfying $d(\epsilon'')=d(\epsilon'')=\sqrt{5}$,
$$[\xi_1\epsilon]=[\epsilon']\oplus [\epsilon'']\oplus [\epsilon'''],$$
and $[\epsilon''\bepsilon]=[\epsilon'''\bepsilon]=[\xi_1]$. 
The above computation shows that the dual principal graph $\cG_{M\supset N}^d$ is $\cG_{\fA_4}^{\fA_5}$, 
and the classification of finite depth subfactors of index 5 shows that $\cG_{M\supset N}$ 
is $\cG_{(\fA_5,X_5)}$ (see \cite{IMPP15}). 
Note that we have $(\fA_5,X_5)=(L(2^2),PG_1(2^2))$, and $\cG_{M\supset N}=\widetilde{\cG_{(1^3),1}}$. 
Let $\cC_2$ be the fusion category generated by $\bepsilon\epsilon$. 
As in the proof of Theorem \ref{LM}, we can parametrize $\Irr(\cC_2)$ as 
$$\Irr(\cC_2)=\{\id,\alpha',{\alpha'}^2,\rho',\sigma,\sigma\alpha',\sigma{\alpha'}^2\},$$
with the following properties: 
$$d(\alpha')=1,\quad d(\rho')=3,\quad d(\sigma)=4,$$
$$[{\alpha'}^3]=[\id],$$
$$[\alpha'\rho]=[\rho'\alpha']=[\rho'],$$
$$[{\rho'}^2]=[\id]\oplus [\alpha']\oplus[{\alpha'}^2]+2[\rho'],$$
$$[\sigma^2]=[\id]\oplus[\rho']\oplus[\sigma]\oplus [\sigma\alpha']\oplus[\sigma{\alpha'}^2],$$
$$[\alpha'\sigma]=[\sigma{\alpha'}^2],$$
$$[\rho'\sigma]=[\sigma\rho']= [\sigma]\oplus[\sigma\alpha']\oplus[\sigma{\alpha'}^2],$$
$$[\bepsilon\epsilon]=[\id]\oplus [\sigma].$$

(4) Theorem \ref{LM} shows that there exists a unique subfactor $R\subset N$, up to inner conjugacy, such that 
$R'\cap M=\C$ and there exists an outer action $\beta$ of $\fA_5$ on $R$ satisfying 
$$M=R\rtimes_\beta \fA_5\supset N=R\rtimes_\beta \fA_4.$$ 
Let $P=R\rtimes_\beta\fA_3\subset N$, and let $\iota:P\hookrightarrow N$ and $\kappa:R\hookrightarrow P$ be the inclusion maps. 
Let $\epsilon_1=\epsilon\iota\kappa$. 
Then $\epsilon_1\overline{\epsilon_1}$ corresponds to the regular representation of $\fA_5$, and
$$[\epsilon_1\overline{\epsilon_1}]=[\id]\oplus 3[\eta_1]\oplus 3[\eta_2]\oplus 4[\pi]\oplus 5[\xi].$$
Thus since $[\bdelta\delta]=[\id]\oplus [\lambda]$,
$$\dim(\delta\epsilon_1,\delta\epsilon_1)=\dim(\bdelta\delta,\epsilon_1\overline{\epsilon_1})=1,$$
and $L\supset R$ is irreducible. 

(5) Note that we have $[L:R]=6|\fA_5|=360$. 
On the other hand, since $[\lambda]$ commutes with $[\epsilon_1\overline{\epsilon_1}]$, and 
$[(\epsilon_1\overline{\epsilon_1})^2]
=|\fA_5|[\epsilon_1\overline{\epsilon_1}]$, 
\begin{align*}
\lefteqn{\dim(\delta\epsilon_1\overline{(\delta\epsilon_1)},\delta\epsilon_1\overline{(\delta\epsilon_1)})
=\dim(\bdelta\delta\epsilon_1\overline{\epsilon_1},\epsilon_1\overline{\epsilon_1}\bdelta\delta)} \\
 &=\dim(\bdelta\delta\epsilon_1\overline{\epsilon_1},\bdelta\delta\epsilon_1\overline{\epsilon_1})
 =\dim((\bdelta\delta)^2,(\epsilon_1\overline{\epsilon_1})^2)=60\dim((\id\oplus \lambda)^2,\epsilon_1\overline{\epsilon_1})\\
 &=60\dim (2\id\oplus \pi\oplus 3\lambda\oplus \mu,\epsilon_1\overline{\epsilon_1})=360.
 \end{align*}
Therefore the inclusion $L\supset R$ is of depth 2. 
 
(6) We denote $\iota_3=\iota\kappa$. 
By Lemma \ref{inv1}, we get
$$\dim(\varphi\bepsilon\epsilon\iota_3\overline{\iota_3}\varphi^{-1},\bepsilon\epsilon\iota_3\overline{\iota_3})=|\fA_4|=12.$$
Note that $\iota_3\overline{\iota_3}$ corresponds to the regular representation of $\fA_4$, and 
$$[\iota_3\overline{\iota_3}]=[\id]\oplus [\alpha']\oplus [{\alpha'}^2]\oplus 3[\rho'].$$
Thus 
$$[\bepsilon\epsilon\iota_3\overline{\iota_3}]=[(\id\oplus \sigma)(\id\oplus \alpha'\oplus {\alpha'}^2\oplus 3\rho')]
=[\id]\oplus [\alpha']\oplus [{\alpha'}^2]\oplus 3[\rho']
\oplus 4[\sigma]\oplus 4[\sigma\alpha']\oplus 4[\sigma{\alpha'}^2].$$
Dimension counting implies 
$$\dim(\varphi(\id\oplus \alpha'\oplus {\alpha'}^2\oplus 3\rho')\varphi^{-1},\id\oplus \alpha'\oplus {\alpha'}^2\oplus 3\rho')=12,$$
and $[\varphi\iota_3\overline{\iota_3}\varphi^{-1}]=[\iota_3\overline{\iota_3}]$. 
We also have 
\begin{equation}\label{disjoint}
(\varphi(\sigma\oplus \sigma\alpha'\oplus\sigma{\alpha'}^2)\varphi^{-1},(\sigma\oplus \sigma\alpha'\oplus\sigma{\alpha'}^2) )=0.
\end{equation}

To finish the proof, we cannot apply Lemma \ref{inv2} to $\fA_4$, and we make a little detour. 
We examine the automorphism $\varphi\in \Aut(N)$ more carefully. 
We first claim $[\varphi^2]=[\id]$. 
Indeed, since $\lambda$ is self-conjugate,
$$1=\dim(\epsilon\varphi\bepsilon,\epsilon\varphi^{-1}\bepsilon)=\dim(\bepsilon\epsilon\varphi,\varphi^{-1}\bepsilon\epsilon)
=\dim(\varphi\oplus \sigma\varphi,\varphi^{-1}\oplus \varphi^{-1}\sigma),$$
and either $[\varphi^2]=[\id]$ or $[\varphi\sigma\varphi]=[\sigma]$ holds. 
Assume that the latter holds. 
Then 
$$[\lambda^2]=[\epsilon\varphi\bepsilon\epsilon\varphi\bepsilon]=[\epsilon\varphi(\id\oplus \sigma)\varphi\bepsilon]
=[\epsilon\varphi^2\bepsilon]\oplus [\epsilon\varphi\sigma\varphi\bepsilon]=[\epsilon\varphi^2\bepsilon]\oplus [\epsilon\sigma\bepsilon].$$
Since 
$$[\epsilon\bepsilon]\oplus [\epsilon\sigma\bepsilon]=[(\epsilon\bepsilon)^2]=([\id]\oplus [\pi])^2=2[\id]\oplus3 [\pi]
\oplus[\xi_1]\oplus [\eta_1]\oplus[\eta_2],$$
we get 
$$[\lambda^2]=[\epsilon\varphi^2\bepsilon]\oplus [\id]\oplus 2[\pi]\oplus[\xi_1]\oplus [\eta_1]\oplus[\eta_2],$$
which is contradiction. 
Thus the claim is shown, and we also have $[\varphi\sigma\varphi]\neq[\sigma]$.

Let $\omega=\sigma\varphi\bepsilon$. 
We show the following 3 properties of $\omega$. 
\begin{itemize} 
 \item[(i)] $\omega$ is irreducible. 
 \item[(ii)] $\dim(\rho',\omega\bar{\omega})=1$. 
 \item[(iii)] $[\varphi\omega]=[\omega]$. 
\end{itemize} 
Indeed, thanks to Eq.(\ref{disjoint}), we get 
$$\dim(\omega,\omega)=\dim(\sigma^2,\varphi\bepsilon\epsilon\varphi^{-1})=
\dim(\id\oplus \rho'\oplus \sigma\oplus \sigma\alpha'\oplus \sigma{\alpha'}^2,\varphi(\id\oplus \sigma)\varphi^{-1})=1,$$
and $\omega$ is irreducible. 
(ii) also follows from Eq.(\ref{disjoint}) as we have 
$$\dim(\rho',\omega\bar{\omega})=\dim(\rho'\omega,\omega)=\dim(\sigma\rho'\sigma,\varphi(\id\oplus \sigma)\varphi^{-1}),$$ 
and $\sigma\rho'\sigma$ contains $\id$ with multiplicity 1. 
(iii) follows from
\begin{align*}
\lefteqn{1=\dim(\lambda,\lambda^2)=\dim(\epsilon\varphi\bepsilon,\epsilon\varphi\bepsilon\epsilon\varphi\bepsilon)
=\dim(\bepsilon\epsilon\varphi\bepsilon,\varphi\bepsilon\epsilon\varphi\bepsilon) } \\
 &=\dim((\id\oplus \sigma)\varphi\bepsilon,\varphi(\id\oplus\sigma)\varphi\bepsilon)
 =\dim(\varphi\bepsilon\oplus\omega,\bepsilon\oplus \varphi\omega)\\
 &=\dim(\varphi,\bepsilon\epsilon)+\dim(\omega,\varphi\omega)=\dim(\omega,\varphi\omega). 
\end{align*}

The proof of Theorem \ref{LM} shows that there exists $\tau\in \Aut(P)$ such that $\sigma$ factorizes as 
$\sigma=\iota\tau\biota$. 
Thus we have $N\supset P\supset \omega(M)$. 
Since $[\iota\biota]=[\id]\oplus [\rho']$, Lemma \ref{inv3} shows that there exists a unitary $u\in N$ satisfying 
$\Ad u\circ \varphi(P)=P$, which means that there exists $\psi\in \Aut(P)$ satisfying $[\varphi\iota]=[\iota\psi]$. 
Now we have 
$$12=\dim(\iota\psi\kappa\bkappa\psi^{-1}\biota,\iota\kappa\bkappa\biota)=\dim(\psi\kappa\bkappa\psi^{-1},\biota\iota\kappa\bkappa\biota\iota).$$

We parametrize $P$-$P$ sectors generated by $\biota\iota$ as in the proof of Theorem \ref{GthF}. 
Then $[\biota\iota]=[\id]\oplus [\rho]$, $[\kappa\bkappa]=[\id]\oplus [\alpha]\oplus[\alpha^2]$, $d(\rho)=3$, $d(\alpha)=1$, $\alpha^3=\id$, 
and they satisfy the following fusion rules: 
$$[\alpha\rho]=[\rho\alpha]=[\rho],$$
$$[\rho^2]=[\id]\oplus[\alpha]\oplus [\alpha^2]\oplus 2[\rho].$$
Now we have  
$$[\biota\iota\kappa\bkappa\biota\iota]=4([\id]\oplus[\alpha]\oplus[\alpha^2]\oplus 3[\rho]),$$
and we get 
$$3=\dim(\psi(\id\oplus\alpha\oplus \alpha^2)\psi^{-1},\id\oplus\alpha\oplus\alpha^2\oplus 3\rho).$$
Thus $[\psi\kappa\bkappa\psi^{-1}]=[\kappa\bkappa]$. 
Lemma \ref{inv2} shows that there exists $\varphi_1\in \Aut(R)$ satisfying $[\psi\kappa]=[\kappa\varphi_1]$, and so 
$[\varphi\iota\kappa]=[\iota\kappa\varphi_1]$. 
Lemma \ref{finishing} shows that there exists a group $\Gamma$ containing $\fA_5$ such that $\gamma$ extends to 
an outer action of $\Gamma$ on $R$ such that 
$$L=R\rtimes \Gamma.$$
The graph $\cG_{L\supset M}$ shows that the $\Gamma$-action on $\Gamma/\fA_5$ is 3-transitive extension of $(\fA_5,X_5)$, 
and we conclude that $\Gamma=\fA_6$. 
\end{proof}

\begin{remark} A similar argument works for $(\fS_6,X_6)$. In this case, we can apply Lemma \ref{inv2} to 
$\fS_3=\Z_3\rtimes \Z_2$ instead of $\fA_3=\Z_3 $ at the last step. 
\end{remark}

Note that we computed the graph $\cG_{M_9}^{M_{10}}$ in Section 4, and the graph $\cG_{M_{11}>M_{10}}$ for 
the Mathieu group $M_{11}$ can be obtained by $\cG_{M_{11}>M_{10}}=\widetilde{\cG_{M_9}^{M_{10}}}$.

\begin{theorem}\label{GthM}
Let $L\supset M$ be a finite index inclusion of factors with $\cG_{L\supset M}=\cG_{M_{11}>M_{10}}$. 
Then there exists a unique subfactor $R\subset M$, up to inner conjugacy, such that $R'\cap L=\C$ and there exists 
an outer action $\gamma$ of $M_{11}$ on $R$ satisfying 
$$L=R\rtimes_\gamma M_{11}\supset M=R\rtimes_\gamma M_{10}.$$ 
\begin{figure}[H]
\centering
\begin{tikzpicture}
\draw (-4,2)--(-3,1.5)--(-2,1)--(-1,0.5)--(0,0)--(-1,-0.5)--(-2,-1)--(-3,-1.5)--(-4,-2);
\draw (-2,0.5)--(-1,0.5);
\draw (-2,-0.5)--(-1,-0.5);
\draw (0,0)--(0,1)--(-1,2);
\draw (0,1)--(1,2);
\draw (0,1)--(0,2);
\draw (0,0)--(1,0.5)--(2,0)--(1,-0.5)--(0,0);
\draw (1,0.5)--(2,1);
\draw (1,-0.5)--(2,-1);
\draw (0,0)--(0,-1)--(0,-2);
\draw(0.1,-0.5)node{2};
\draw(-4,2)node{$*$};
\draw(-3,1.5)node{$\bullet$};
\draw(-2,1)node{$\bullet$};
\draw(-1,0.5)node{$\bullet$};
\draw(0,0)node{$\bullet$};
\draw(-1,-0.5)node{$\bullet$};
\draw(-2,-1)node{$\bullet$};
\draw(-3,-1.5)node{$\bullet$};
\draw(-4,-2)node{$\bullet$};
\draw(-2,0.5)node{$\bullet$};
\draw(-2,-0.5)node{$\bullet$};
\draw(0,1)node{$\bullet$};
\draw(-1,2)node{$\bullet$};
\draw(1,2)node{$\bullet$};
\draw(0,2)node{$\bullet$};
\draw(1,0.5)node{$\bullet$};
\draw(2,1)node{$\bullet$};
\draw(1,-0.5)node{$\bullet$};
\draw(2,-1)node{$\bullet$};
\draw(2,0)node{$\bullet$};
\draw(0,-1)node{$\bullet$};
\draw(0,-2)node{$\bullet$};
\draw(-3.8,2.3)node{$\id_M$};
\draw(-2.8,1.8)node{$\delta$};
\draw(-1.8,1.3)node{$\lambda$};
\draw(-0.8,0.8)node{$\delta\pi$};
\draw(-2,0.5)node[left]{$\pi$};
\draw(-3.8,-2.3)node{$\chi$};
\draw(-2.8,-1.8)node{$\delta\chi$};
\draw(-1.8,-1.3)node{$\lambda\chi$};
\draw(-0.8,-0.8)node{$\delta\pi\chi$};
\draw(-2,-0.5)node[left]{$\pi\chi$};
\draw(0.2,0.3)node{$\mu$};
\draw(0,1)node[right]{$\delta\xi_1$};
\draw(-1,2)node[above]{$\xi_1$};
\draw(0,2)node[above]{$\xi_2$};
\draw(1,2)node[above]{$\xi_3$};
\draw(1,0.5)node[below]{$\delta\eta_1$};
\draw(1,-0.5)node[below]{$\delta\eta_2$};
\draw(0,-1)node[right]{$\delta\zeta$};
\draw(2,1)node[right]{$\eta_1$};
\draw(2,-1)node[right]{$\eta_2$};
\draw(2,0)node[right]{$\nu$};
\draw(0,-2)node[below]{$\zeta$};
\end{tikzpicture}
\caption{\label{M11-M10}$\cG_{M_{11}>M_{10}}$}
\end{figure}
\end{theorem}

\begin{proof} 
(1) Let $\delta:M\hookrightarrow L$ be the inclusion map, and let $[\bdelta\delta]=[\id]\oplus [\lambda]$ 
be the irreducible decomposition. 
We parametrize the irreducible $M$-$M$ sectors and the $L$-$M$ sectors generated by $\delta$ as in Figure \ref{M11-M10}. 
Then we have  
$$d(\chi)=1,\quad d(\pi)=9,\quad d(\lambda)=d(\xi_1)=d(\xi_2)=d(\xi_3)=d(\eta_1)=d(\eta_2)=10,$$
$$\quad d(\zeta)=16,\quad d(\nu)=20, \quad  d(\mu)=80,\quad d(\delta)=\sqrt{11}.$$
From the graph, we can see that $\lambda$, $\pi$, $\pi\chi$, $\mu$, $\nu$, and $\chi$ are self-conjugate, and 
$$\{[\chi\lambda],[\overline{\xi_1}],[\overline{\xi_2}],[\overline{\xi_3}],[\overline{\eta_1}],[\overline{\eta_2}]\}
=\{[\lambda\chi],[\xi_1],[\xi_2],[\xi_3],[\eta_1],[\eta_2]\}.$$ 
Since $\pi\chi$ is self-conjugate, we have $[\pi\chi]=[\chi\pi]$. 
By the graph symmetry, we have $[\zeta\chi]=[\zeta]$, $[\mu\chi]=[\mu]$, $[\nu\chi]=[\nu]$, and 
$$\{[\xi_1\chi],[\xi_2\chi],[\xi_3\chi]\}=\{[\xi_1],[\xi_2],[\xi_3]\},$$
$$\{[\eta_1\chi],[\eta_2\chi]\}=\{[\eta_1],[\eta_2]\}.$$

The basic fusion rules coming from the graph are 
\begin{equation}\label{M1}
[\lambda^2]=[\id]\oplus [\lambda]\oplus [\pi]\oplus [\mu],
\end{equation}
\begin{equation}\label{M2}
[\lambda\pi]=[\lambda]\oplus [\mu],
\end{equation}
\begin{equation}\label{M3}
[\lambda\mu]=[\lambda]\oplus [\lambda\chi]\oplus [\pi]\oplus[\pi\chi]\oplus [\xi_1]\oplus[\xi_2]\oplus [\xi_3]\oplus [\eta_1]\oplus[\eta_2]
\oplus 2[\zeta]\oplus 2[\nu]\oplus 8[\mu],
\end{equation}
\begin{equation}\label{M4}
[\lambda\xi_i]+[\xi_i]=[\xi_1]\oplus [\xi_2]\oplus [\xi_3]\oplus [\mu], 
\end{equation}
\begin{equation}\label{M5}
[\lambda\eta_i]=[\nu]\oplus [\mu], 
\end{equation}
\begin{equation}\label{M6}
[\lambda\zeta]=2[\mu], 
\end{equation}
\begin{equation}\label{M7}
[\lambda\nu]=[\eta_1]\oplus [\eta_2]\oplus [\nu]\oplus 2[\mu].  
\end{equation}
Since the right-hand sides of Eq.(\ref{M2}),(\ref{M3}),(\ref{M5}),(\ref{M6}) are self-conjugate, we have 
$[\lambda\pi]=[\pi\lambda]$, $[\lambda\mu]=[\mu\lambda]$, $[\lambda\eta_i]=[\overline{\eta_i}\lambda]$, and 
$[\lambda\zeta]=[\zeta\lambda]$.
Since $[\lambda^2\pi]=[\pi\lambda^2]$, we get $[\mu\pi]=[\pi\mu]$. 
  
By associativity, we get 
\begin{align}\label{M8}
[\pi^2]\oplus [\mu\pi]&=[\id]\oplus [\lambda]\oplus [\lambda\chi]\oplus [\pi]\oplus [\pi\chi] \\
 &\oplus[\xi_1]\oplus[\xi_2]\oplus [\xi_3]
 \oplus [\eta_1]\oplus[\eta_2]\oplus 2[\zeta]\oplus 2[\nu]\oplus 8[\mu], \nonumber
\end{align}
\begin{align}\label{M9}
[\pi\mu]\oplus[\mu^2]&=[\id]\oplus [\chi]\oplus 9[\lambda]\oplus 9[\lambda\chi]\oplus 8[\pi]\oplus 8[\pi\chi] \\
 &\oplus 9[\xi_1]\oplus 9[\xi_2]\oplus 9[\xi_3]
\oplus 9[\eta_1]\oplus 9[\eta_2]\oplus 14[\zeta]\oplus 18[\nu]\oplus  72[\mu], \nonumber
\end{align}
\begin{align}\label{M10}
[\pi\xi_i]\oplus[\mu\xi_i] &= [\lambda]\oplus [\lambda\chi]\oplus [\pi]\oplus[\pi\chi] \\
 &\oplus [\xi_i]\oplus [\xi_1]\oplus[\xi_2]\oplus [\xi_3]\oplus [\eta_1]\oplus[\eta_2]
\oplus 2[\zeta]\oplus 2[\nu]\oplus 9[\mu], \nonumber
\end{align}
\begin{align}\label{M11}
[\eta_i]\oplus [\pi\eta_i]\oplus [\mu\eta_i] &= [\lambda]\oplus [\lambda\chi]\oplus [\pi]\oplus[\pi\chi]\\
 &\oplus [\xi_1]\oplus[\xi_2]\oplus [\xi_3]
\oplus 2[\eta_1]\oplus 2[\eta_2]
\oplus 2[\zeta]\oplus 2[\nu]\oplus 9[\mu], \nonumber
\end{align}
\begin{align}\label{M12}
[\pi\zeta]\oplus [\mu\zeta] &=2[\lambda]\oplus 2[\lambda\chi]\oplus 2[\pi]\oplus 2[\pi\chi] \\
 &\oplus 2[\xi_1]\oplus 2[\xi_2]\oplus 2[\xi_3]\oplus 2[\eta_1]\oplus 2[\eta_2]\oplus 3[\zeta]\oplus 4[\nu]\oplus 14[\mu],  \nonumber
\end{align}
\begin{align}\label{M13}
[\pi\nu]\oplus[\mu\nu] &=2[\lambda]\oplus 2[\lambda\chi]\oplus 2[\pi]\oplus 2[\pi\chi] \\
 &\oplus 2[\xi_1]\oplus 2[\xi_2]\oplus 2[\xi_3]\oplus 2[\eta_1]\oplus 2[\eta_2]\oplus 4[\zeta]\oplus 5[\nu]\oplus 18[\mu]. \nonumber
\end{align}

We give a criterion to separate the summations of the left-hand sides.  
For irreducible $X$ and $Y$, we have 
$$\dim (\lambda \pi X,Y)=\dim(\pi X,\lambda Y),$$
and on the other hand, 
$$\dim(\lambda \pi X,Y)=\dim((\lambda\oplus \mu)X,Y)=\dim(\lambda X,Y)+\dim (\mu X,Y).$$
Thus 
\begin{equation}\label{M14}
\dim(\pi X\oplus \mu X,Y)=\dim(\pi X,Y\oplus \lambda Y)-\dim(\lambda X,Y). 
\end{equation}

The Frobenius reciprocity implies $\dim(\pi^2,\lambda)=0$ and $\dim(\mu\pi,\lambda)=1$. 
We claim that $\pi^2$ does not contain $\mu$. 
Assume on the contrary that $\pi^2$ contains $\mu$. 
Then Eq.(\ref{M8}) implies $\dim(\mu\pi,\mu)=7$. 
Since $[\lambda]$ commutes with $[\zeta]$, Eq.(\ref{M6}) shows that $2[\mu]=[\zeta\lambda]$, and 
$$14=\dim(2\mu\pi,\mu)=\dim(\zeta\lambda\pi,\mu)=\dim(\zeta(\lambda\oplus \mu),\mu)=\dim(2\mu\oplus \zeta\mu,\mu).$$
Since $\mu$ and $\zeta$ are self-conjugate, we get $\dim(\mu\zeta,\mu)=12$. 
However, this and Eq.(\ref{M12}) show $\dim(\pi\zeta,\mu)=2$, which is impossible because $d(\pi\zeta)<2d(\mu)$. 
Therefore the claim is shown. 
The Frobenius reciprocity implies that we have $\dim(\mu\pi,\pi)=0$. 
Since $[\mu\pi\chi]=[\mu\chi\pi]=[\mu\pi]$, we get $\dim(\mu\pi,\pi\chi)=0$ too. 
Thus dimension counting shows that we may put 
$$[\pi^2]=[\id]\oplus [\pi]\oplus [\pi\chi]\oplus 2[\zeta]\oplus \bigoplus_{i=1}^3a_i[\xi_i]\oplus \bigoplus_{i=1}^2b_i[\eta_i]\oplus c[\nu],$$
where $a_i$, $b_i$, and $c$ are non-negative integers satisfying 
$$\sum_{i=1}^3a_i+\sum_{i=1}^2b_i+2c=3.$$ 
Applying Eq.(\ref{M14}) to this, we obtain $a_1+a_2+a_3=1$ and $b_i=1-c$. 
We may and do assume $a_1=1$, $a_2=a_3=0$, and 
$$[\pi^2]=[\id]\oplus [\pi]\oplus [\pi\chi]\oplus 2[\zeta]\oplus [\xi_1]\oplus (1-c)[\eta_1]\oplus (1-c)[\eta_2]\oplus c[\nu],$$
$$[\mu\pi]=8[\mu]\oplus [\lambda]\oplus [\lambda\chi]\oplus [\xi_2]\oplus [\xi_3]\oplus c[\eta_1]\oplus c[\eta_2]\oplus (2-c)[\nu].$$

Since $[\mu\pi]=[\pi\mu]$, Eq.(\ref{M9}) shows $\dim(\mu^2,\xi_i)=8$ for $i=2,3$, and the Frobenius reciprocity and 
Eq.(\ref{M10}) show that $\pi\xi_i$ contains $\mu$. 
Thus 
$$[\pi\xi_i]=[\mu]\oplus 10 \textrm{ dim},\quad i=2,3.$$ 
If $\xi_i$ were not contained in $\pi\xi_i$, Eq.(\ref{M10}) implies that $\mu\xi_i$ would contain $\xi_i$ with multiplicity 2, 
and consequently $\xi_i\overline{\xi_i}$ would contain $\mu$ with multiplicity 2, which is contradiction because 
$d(\xi_i)^2<2d(\mu).$ 
Thus we have 
\begin{equation}\label{M15}
[\pi\xi_i]=[\mu]\oplus [\xi_i],\quad i=2,3,
\end{equation}
\begin{align}\label{M16}
\lefteqn{[\mu\xi_i]= [\lambda]\oplus [\lambda\chi]\oplus [\pi]\oplus[\pi\chi] } \\
 &\oplus [\xi_1]\oplus[\xi_2]\oplus [\xi_3]\oplus [\eta_1]\oplus[\eta_2]
\oplus 2[\zeta]\oplus 2[\nu]\oplus 8[\mu],\quad i=2,3.\nonumber
\end{align}

Since our argument is already long, we state the next claim as a separate lemma. 

\end{proof}

\begin{lemma} With the above notation, we have $c=0$. 
\end{lemma}

\begin{proof}
Assume on the contrary that $c=1$.  
Since  
$$2[\pi\mu]=[\pi\lambda\zeta]=[(\mu\oplus \lambda)\zeta]=[\mu\zeta]\oplus 2[\mu],$$
we can obtain the irreducible decomposition of $\mu\zeta$ and $\pi\zeta$. 

Now Eq.(\ref{M14}), the Frobenius reciprocity, and dimension counting show the following:

\begin{equation}
[\pi^2]=[\id]\oplus [\pi]\oplus[\pi\chi]\oplus 2[\zeta]\oplus [\xi_1]\oplus [\nu],\tag{W1}
\end{equation}

\begin{equation}[\mu\pi]= [\lambda]\oplus [\lambda\chi] 
\oplus[\xi_2]\oplus [\xi_3]\oplus [\eta_1]\oplus [\eta_2]\oplus [\nu]\oplus 8[\mu], \tag{W2}
\end{equation}

\begin{align}
[\mu^2]&=[\id]\oplus [\chi]\oplus 8[\lambda]\oplus 8[\lambda\chi]\oplus 8[\pi]\oplus 8[\pi\chi] \tag{W3}\\
 &\oplus 9[\xi_1]\oplus 8[\xi_2]\oplus 8[\xi_3]
\oplus 8[\eta_1]\oplus 8[\eta_2]\oplus 14[\zeta]\oplus 17[\nu]\oplus  64[\mu], \nonumber
\end{align}

\begin{equation}
[\pi\zeta]=2[\pi]\oplus 2[\pi\chi]\oplus 2[\xi_1]\oplus 3[\zeta]\oplus 2[\nu],\tag{W4}
\end{equation}

\begin{equation}\label{W5}
[\mu\zeta]= 2[\lambda]\oplus 2[\lambda\chi]\oplus 2[\xi_2]\oplus 2[\xi_3]\oplus 2[\eta_1]\oplus 2[\eta_2]\oplus 2[\nu]\oplus 14[\mu], 
\tag{W5}
\end{equation}

\begin{equation}
[\pi\xi_1]=[\pi]\oplus [\pi\chi]\oplus 2[\zeta]\oplus 2[\xi_1]\oplus [\nu],\tag{W6}
\end{equation}

\begin{equation}
[\mu\xi_1] = [\lambda]\oplus [\lambda\chi]\oplus[\xi_2]\oplus [\xi_3]\oplus [\eta_1]\oplus[\eta_2]
\oplus [\nu]\oplus 9[\mu], \tag{W7}
\end{equation}

\begin{equation}
[\pi\nu]=[\pi]\oplus [\pi\chi]\oplus 2[\zeta]\oplus [\xi_1]\oplus 2[\nu]\oplus [\mu],\tag{W8}
\end{equation}

\begin{align}
\lefteqn{[\mu\nu] =2[\lambda]\oplus 2[\lambda\chi]\oplus [\pi]\oplus [\pi\chi]  \tag{W9}}\\
 &\oplus [\xi_1]\oplus 2[\xi_2]\oplus 2[\xi_3]\oplus 2[\eta_1]\oplus 2[\eta_2]\oplus 2[\zeta]\oplus 3[\nu]\oplus 17[\mu] \nonumber
\end{align}
Here the letter `W' stands for wrong equations. 
Since the right-hand sides are self-conjugate, we see that $[\mu]$ commutes with $[\pi]$, $[\xi_1]$, $[\zeta]$, and $[\nu]$. 

An argument similar to the case of $\pi\xi_i$ with $i=2,3$, shows $[\pi\eta_1]=[\mu]\oplus [\eta_2]$, 
and $[\pi\eta_2]=[\mu]\oplus [\eta_1]$. 
Eq.(\ref{M10}) shows 
$$2=\dim(\mu\eta_i,\zeta)=\dim(\eta_i\zeta,\mu),$$
and consequently $[\eta_i\zeta]=2[\mu]$. 
In the same way, we have $[\xi_2\zeta]=[\xi_3\zeta]=2[\mu]$, and taking conjugate, we also get 
\begin{equation}
[\xi_2\zeta]=[\xi_3\zeta]=[\eta_1\zeta]=[\eta_2\zeta]=[\zeta\xi_2]=[\zeta\xi_3]=[\zeta\eta_1]=[\zeta\eta_2]=2[\mu].\tag{W10}
\end{equation}

\begin{equation}
[\pi\eta_1]=[\mu]\oplus [\eta_2],\quad [\pi\eta_1]=[\mu]\oplus [\eta_1],\tag{W11}
\end{equation}

\begin{align}
\lefteqn{[\mu\xi_2]=[\mu\xi_3]=[\mu\eta_1]=[\mu\eta_2] \tag{W12}} \\
 &=[\lambda]\oplus [\lambda\chi]\oplus [\pi]\oplus[\pi\chi] 
\oplus [\xi_1]\oplus[\xi_2]\oplus [\xi_3]\oplus [\eta_1]\oplus[\eta_2]
\oplus 2[\zeta]\oplus 2[\nu]\oplus 8[\mu].\nonumber
\end{align}

From,
$$2[\mu\xi_2]=[\zeta\lambda\xi_1]=[\zeta(\mu\oplus \xi_2\oplus \xi_3)]=[\zeta\mu]\oplus [\zeta(\xi_1\oplus \xi_3)]
=[\zeta\mu]\oplus 2[\mu]\oplus [\zeta\xi_1],$$
we get the irreducible decomposition of $\zeta\xi_1$. 

From
$$2[\mu\eta_i]=[\zeta\lambda\eta_i]=[\zeta(\mu\oplus \nu)]=[\zeta\mu]\oplus [\zeta\nu],$$
we get the irreducible decomposition of $[\zeta\nu]$. 
The Frobenius reciprocity and dimension counting shows 
\begin{equation}
[\zeta\xi_1]=2[\pi]\oplus 2[\pi\chi]\oplus 4[\zeta]\oplus 2[\xi_1]\oplus 2[\nu],\tag{W13}
\end{equation}

\begin{equation}
[\zeta\nu]=2[\pi]\oplus 2[\pi\chi]\oplus 4[\zeta]\oplus 2[\xi_1] \oplus 2[\nu]\oplus 2[\mu],\tag{W14}
\end{equation}

\begin{equation}
[\zeta^2]=[\id]\oplus [\chi]\oplus 3[\pi]\oplus 3[\pi\chi]\oplus 5[\zeta]\oplus 4[\xi_1]\oplus 4[\nu].\tag{W15}
\end{equation}

Next we determine the left multiplications of $[\xi_1]$ and $[\nu]$ by applying associativity to $[\pi^2X]$. 
The two equations 
$$[\pi(\pi\xi_1)]=[\pi(\pi\oplus \pi\chi\oplus 2\zeta\oplus 2\xi_1\oplus \nu)],$$
$$[\pi^2\xi_1]=[(\id\oplus \pi\oplus\pi\chi\oplus 2\zeta\oplus \xi_1\oplus \nu)\xi_1],$$
show
$$[\xi_1^2]\oplus [\nu\xi_1]=[\id]\oplus[\chi]\oplus 3[\pi]\oplus 3[\pi\chi]\oplus 4[\zeta]\oplus 2[\xi_1]\oplus 4[\nu]\oplus [\mu].$$
By the Frobenius reciprocity and dimension computing, we get 
\begin{equation}
[\xi_1^2]=[\id]\oplus [\chi]\oplus 2 [\pi]\oplus 2[\pi\chi]\oplus 2[\zeta]\oplus [\xi_1]\oplus [\nu],\tag{W16}
\end{equation}
\begin{equation}
[\nu\xi_1]=[\pi]\oplus [\pi\chi]\oplus 2[\zeta]\oplus [\mu]\oplus [\xi_1]\oplus 3[\nu].\tag{W17} 
\end{equation}

The two equations 
$$[\pi(\pi\nu)]=[\pi(\pi\oplus \pi\chi\oplus 2\zeta\oplus \xi_1\oplus 2\nu\oplus \mu)],$$
$$[\pi^2\nu]=[(\id\oplus \pi\oplus\pi\chi\oplus 2\zeta\oplus \xi_1\oplus \nu)\nu],$$ 
show
\begin{align*}
\lefteqn{[\xi_1\nu]\oplus[\nu^2]=[\id]\oplus [\chi]\oplus [\lambda]\oplus [\lambda\chi]\oplus 3[\pi]\oplus 3[\pi\chi]} \\
 &\oplus 4[\zeta]\oplus 4[\xi_1]\oplus 3[\nu]\oplus 4[\mu]\oplus [\xi_2]\oplus[\xi_3]\oplus[\eta_1]\oplus [\eta_2], 
\end{align*}
and 
\begin{align}
\lefteqn{[\nu^2]=[\id]\oplus [\chi]\oplus [\lambda]\oplus [\lambda\chi]\oplus 2[\pi]\oplus 2[\pi\chi]\tag{W18}} \\
 &\oplus 2[\zeta]\oplus 3[\mu]\oplus 3[\xi_1]\oplus [\xi_2]\oplus [\xi_3]\oplus [\eta_1]\oplus [\eta_2]. \nonumber
\end{align}

The two equations  
$$[\pi(\pi\eta_1)]=[\pi\mu]\oplus [\pi\eta_2]=[\pi\mu]\oplus[\mu]\oplus[\eta_1],$$
$$[\pi^2\eta_1]=[(\id\oplus \pi\oplus\pi\chi\oplus 2\zeta\oplus \xi_1\oplus \nu)\eta_1]
=[\eta_1]\oplus [\eta_2]\oplus [\chi\eta_2]\oplus 6[\mu]\oplus [\xi_1\eta_1]\oplus [\nu\eta_1].$$
show that $[\chi\eta_2]=[\eta_1]$, and 
$$[\xi_1\eta_1]\oplus [\nu\eta_1]=[\lambda]\oplus[\lambda\chi]\oplus[\xi_2]\oplus[\xi_3]\oplus [\nu]\oplus 3[\mu]. $$
In a similar way, we have $[\chi\xi_2]=[\xi_3]$ and 
$$[\xi_1\xi_2]\oplus [\nu\xi_2]=[\lambda]\oplus [\lambda\chi]\oplus [\eta_1]\oplus [\eta_2]\oplus [\nu]\oplus 3[\mu].$$

We claim
$$\{[\overline{\xi_2},[\overline{\xi_3}],[\overline{\eta_1}],[\overline{\eta_2}]\}=\{[\xi_2],[\xi_3],[\eta_1],[\eta_2]\}.$$
Indeed, since $[\lambda\chi\xi_2]=[\lambda\xi_3]$ does not contain $\id$, we see that $\lambda\chi$ is not 
the conjugate sector of $\xi_2$. 
A similar argument applied to $\xi_3$, $\eta_1$, and $\eta_2$ shows the claim. 

Assume first that $[\overline{\xi_2}]$ is either $[\eta_1]$ or $[\eta_2]$. 
Note that in this case $[\overline{\xi_3}]=[\overline{\xi_2}\chi]$ is also either $\eta_1$ or $\eta_2$. 
Then 
$$\dim(\xi_1\eta_2,\lambda)=\dim(\lambda \xi_1,\overline{\eta_2})=1,$$  
and $\dim(\xi_1\eta_2,\lambda\chi)=1$ in the same way. 
We have 
$$\dim(\nu\xi_2,\lambda)=\dim(\lambda\nu,\overline{\xi_2})=1,$$
and $\dim(\nu\xi_2,\lambda\chi)=1$ in the same way. 
Thus \begin{equation}\label{W19}
[\xi_1\eta_1]=[\xi_1\eta_2]=[\lambda]\oplus [\lambda\chi]\oplus [\mu],\tag{W19}
\end{equation}
\begin{equation}\label{W20}
[\nu\eta_1]=[\nu\eta_2]=[\xi_2]\oplus [\xi_3]\oplus [\nu]\oplus 2[\mu],\tag{W20}
\end{equation}
\begin{equation}\label{W21}
[\xi_1\xi_2]=[\xi_1\xi_3]=[\eta_1]\oplus [\eta_2]\oplus [\mu],\tag{W21}
\end{equation}
\begin{equation}\label{W22}
[\nu\xi_2]=[\nu\xi_3]=[\lambda]\oplus [\lambda\chi]\oplus[\nu]\oplus 2[\mu].\tag{W22}
\end{equation}
Multiplying the both sides of Eq.(\ref{W20}) and (\ref{W21}) by $[\lambda]$ from left, we get
$$[\eta_1^2]\oplus[\eta_2\eta_1]=[\eta_1\eta_2]\oplus [\eta_2^2]=2[\xi_1]\oplus [\eta_1]\oplus[\eta_2]\oplus 2[\mu],$$ 
$$[\xi_2^2]\oplus[\xi_3\xi_2]=[\xi_2\xi_3]\oplus [\xi_3^2]=2[\mu]\oplus 2[\nu].$$
Taking conjugate, we get contradiction. 

Assume now that $[\overline{\xi_2}]$ is either $[\xi_2]$ or $[\xi_3]$. 
In this case $[\overline{\xi_3}]$ is either $[\xi_2]$ or $[\xi_3]$ too. 
The Frobenius reciprocity and dimension counting show 
\begin{equation}\label{W23}
[\xi_1\eta_1]=[\xi_1\eta_2]=[\mu]\oplus [\xi_2]\oplus [\xi_3],\tag{W23}
\end{equation}
\begin{equation}\label{W24}
[\nu\eta_1]=[\nu\eta_2]=[\lambda]\oplus [\lambda\chi]\oplus [\nu]\oplus 2[\mu],\tag{W24} 
\end{equation}
\begin{equation}\label{W25}
[\xi_1\xi_2]=[\xi_1\xi_3]=[\lambda]\oplus [\lambda\chi]\oplus [\mu],\tag{W25}
\end{equation}
\begin{equation}\label{W26}
[\nu\xi_2]=[\nu\xi_3]=[\eta_1]\oplus [\eta_2]\oplus [\nu]\oplus 2[\mu].\tag{W26} 
\end{equation}
Multiplying the both sides of Eq.(\ref{W23}) and (\ref{W26}) by $[\lambda]$ from left, 
we get 
$$[\xi_2\eta_1]\oplus [\xi_3\eta_1]=[\xi_2\eta_2]\oplus [\xi_3\eta_2]
=2[\xi_1]\oplus [\xi_2]\oplus [\xi_3]\oplus 2[\mu],$$ 
$$[\eta_1\xi_2]\oplus [\eta_2\xi_2]=[\eta_1\xi_3]\oplus [\eta_2\xi_3]=2[\nu]\oplus 2[\mu],$$
which is contradiction again. 
Finally we conclude that $c=0$. 
\end{proof}

\begin{proof}[Continuation of the proof of Theorem \ref{GthM}]
The above lemma and Eq.(\ref{M8}) show 
\begin{equation}\label{M17}
[\pi^2]=[\id]\oplus [\pi]\oplus [\pi\chi]\oplus 2[\zeta]\oplus [\xi_1]\oplus [\eta_1]\oplus [\eta_2],
\end{equation}
\begin{equation}\label{M18}
[\mu\pi]=8[\mu]\oplus [\lambda]\oplus [\lambda\chi]\oplus [\xi_2]\oplus [\xi_3]\oplus 2[\nu].
\end{equation}
From Eq.(\ref{M17}), we can see 
\begin{equation}\label{M19}
\{[\overline{\xi_1}],[\overline{\eta_1}],[\overline{\eta_2}]\}=\{[\xi_1],[\eta_1],[\eta_2]\},
\end{equation}
and in consequence 
\begin{equation}\label{M20}
\{[\chi\lambda],[\overline{\xi_2}],[\overline{\xi_3}]\}=\{[\lambda\chi],[\xi_2],[\xi_3]\}.
\end{equation}

Since  
$$2[\pi\mu]=[\pi\lambda\zeta]=[(\mu\oplus \lambda)\zeta]=[\mu\zeta]\oplus 2[\mu],$$
we get 
\begin{equation}\label{M21}
[\mu\zeta] =2[\lambda]\oplus 2[\lambda\chi]\oplus 2[\xi_2]\oplus 2[\xi_3]\oplus 4[\nu]\oplus 14[\mu], 
\end{equation}
and from Eq.(\ref{M12}),  
\begin{equation}\label{M22}
[\pi\zeta]=2[\pi]\oplus 2[\pi\chi]\oplus 2[\xi_1]\oplus 2[\eta_1]\oplus 2[\eta_2]\oplus 3[\zeta]. 
\end{equation} 

Eq.(\ref{M18}) shows that $\pi\nu$ contains $\mu$ with multiplicity 2. 
If $\pi\nu$ contained $\nu$ with multiplicity at most 1, Eq.(\ref{M13}) shows that $\nu^2$ would contain 
$\mu$ with multiplicity 4, which is impossible because $d(\nu^2)=4d(\mu)$ and $\nu^2$ contains $\id$. 
Thus we get 
\begin{equation}\label{M23}
[\pi\nu]=2[\mu]\oplus 2[\nu]. 
\end{equation}

Now the Frobenius reciprocity implies that neither $\pi\xi_1$, $\pi\eta_1$, nor $\pi\eta_2$ contains 
$\lambda$, $\lambda\chi$, $\xi_2$, $\xi_3$, $\nu$, and we get 
\begin{equation}\label{M25}
[\pi\xi_1]=[\pi]\oplus[\pi\chi]\oplus 2[\xi_1]\oplus [\eta_1]\oplus[\eta_2]
\oplus 2[\zeta],
\end{equation}
\begin{equation}\label{M26}
[\pi\eta_1]= [\pi]\oplus[\pi\chi]\oplus [\xi_1]\oplus [\eta_1]\oplus 2[\eta_2]
\oplus 2[\zeta],
\end{equation}
\begin{equation}\label{M27}
[\pi\eta_2]= [\pi]\oplus[\pi\chi]\oplus [\xi_1]\oplus 2[\eta_1]\oplus [\eta_2]
\oplus 2[\zeta]. 
\end{equation}

The above fusion rules show that the fusion category $\cC_1$ generated by $\pi$ satisfies 
$$\Irr(\cC_1)=\{\id,\chi,\pi,\pi\chi,\xi_1,\eta_1,\eta_2,\zeta\}.$$ 

(2) Theorem \ref{intermediate} and Eq.(\ref{M1}) imply that there exists a unique intermediate subfactor $N$ between 
$M$ and $\lambda(M)$ such that if $\epsilon:N\hookrightarrow M$ is the inclusion map, we have 
$$[\epsilon\bepsilon]=[\id]\oplus [\pi].$$ 
Note that we have $d(\epsilon)=\sqrt{10}$. 
In the same way as in the proof of Lemma \ref{factorization}, there exists $\varphi\in \Aut(N)$ satisfying 
$[\lambda]=[\epsilon\varphi\bepsilon]$. 

(3) We show the dual principal graph $\cG_{M\supset N}^d$ is $\cG_{M_9}^{M_{10}}$ computed in Section 4. 
Since 
$$\dim(\pi\epsilon,\pi\epsilon)=\dim(\pi,\pi\epsilon\bepsilon)=\dim(\pi,\pi(1\oplus \pi))=2,$$
there exists an irreducible $\epsilon_0$ satisfying $[\pi\epsilon]=[\epsilon]\oplus [\epsilon_0]$ and $d(\epsilon_0)=8\sqrt{10}$. 
Since Eq.(\ref{M19}) and 
$$[\pi\epsilon\bepsilon]=[\pi]\oplus [\pi^2]=[\id]\oplus 2[\pi]\oplus [\chi\pi]\oplus [\xi_1]\oplus [\eta_1]\oplus [\eta_2]\oplus 
2[\zeta],$$
we get 
$$[\epsilon_0\bepsilon]=[\pi]\oplus [\chi\pi]\oplus [\overline{\xi_1}]\oplus [\overline{\eta_1}]\oplus [\overline{\eta_2}]\oplus 
2[\zeta].$$ 
By the Frobenius reciprocity, 
$$[\zeta\epsilon]=2[\epsilon_0].$$
Since 
$$\dim(\overline{\xi_1}\epsilon,\overline{\xi_1}\epsilon)=\dim(\overline{\xi_1},\overline{\xi_1}(\id\oplus \pi))=3,$$
there exist two irreducibles $\epsilon_2$ and $\epsilon_3$ satisfying 
$$[\overline{\xi_1}\epsilon]=[\epsilon_0]\oplus [\epsilon_2]\oplus [\epsilon_3],$$
and $d(\epsilon_2)+d(\epsilon_3)=2\sqrt{10}$. 
By the Frobenius reciprocity, we get $d(\epsilon_2)=d(\epsilon_3)=\sqrt{10}$ and 
$$[\epsilon_2\bepsilon]=[\epsilon_2\bepsilon]=[\overline{\xi_1}].$$
In a similar way, we can show 
$$\dim(\overline{\eta_1}\epsilon,\overline{\eta_1}\epsilon)=\dim(\overline{\eta_2}\epsilon,\overline{\eta_2}\epsilon)
=\dim(\overline{\eta_1}\epsilon,\overline{\eta_2}\epsilon)=2,$$
and there exists irreducible $\epsilon_4$ satisfying 
$$[\eta_1\epsilon]=[\eta_2\epsilon]=[\epsilon_4],$$
and $d(\epsilon_4)=2\sqrt{10}$. 
The Frobenius reciprocity shows 
$$[\epsilon_4\bepsilon]=[\overline{\eta_1}]\oplus [\overline{\eta_2}].$$
Note that $\xi_1$ is self-conjugate and $\{[\overline{\eta_1}],[\overline{\eta_2}]\}=\{[\eta_1],[\eta_2]\}$.  
Thus we get $\cG_{M\supset N}^d=\cG_{M_9}^{M_{10}}$.  
\begin{figure}[H]
\centering
\begin{tikzpicture}
\draw (-3,1)--(-2,0)--(-1,1)--(0,0)--(1,1)--(2,0)--(3,1);
\draw (0,0)--(0,1);
\draw (0,0)--(4,1)--(3.5,0);
\draw (4,1)--(4.5,0);
\draw (0,0)--(5,1)--(5.5,0);
\draw(0,0)--(6,1)--(5.5,0);
\draw(-0.1,0.5)node{2};
\draw(-3,1)node{$*$};
\draw(-2,0)node{$\bullet$};
\draw(-1,1)node{$\bullet$};
\draw(0,0)node{$\bullet$};
\draw(0,1)node{$\bullet$};
\draw(1,1)node{$\bullet$};
\draw(2,0)node{$\bullet$};
\draw(3,1)node{$\bullet$};
\draw(4,1)node{$\bullet$};
\draw(3.5,0)node{$\bullet$};
\draw(4.5,0)node{$\bullet$};
\draw(5,1)node{$\bullet$};
\draw(6,1)node{$\bullet$};
\draw(5.5,0)node{$\bullet$};
\draw(-3,1)node[above]{$\id_M$};
\draw(-1,1)node[above]{$\pi$};
\draw(0,1)node[above]{$\zeta$};
\draw(1,1)node[above]{$\chi\pi$};
\draw(3,1)node[above]{$\chi$};
\draw(4,1)node[above]{$\xi_1$};
\draw(5,1)node[above]{$\eta_1$};
\draw(6,1)node[above]{$\eta_2$};
\draw(-2,0)node[below]{$\epsilon$};
\draw(0,0)node[below]{$\epsilon_0$};
\draw(2,0)node[below]{$\chi\epsilon$};
\draw(3.5,0)node[below]{$\epsilon_2$};
\draw(4.5,0)node[below]{$\epsilon_3$};
\draw(5.5,0)node[below]{$\epsilon_4$};
\end{tikzpicture}
\caption{$\cG_{M\supset N}^d$}
\end{figure}

Now Theorem \ref{dualM10-M9} implies that $\cG_{M\supset N}=\cG_{M_{10}>M_9}$. 

The rest of the proof is very much similar to that of Theorem \ref{GTA6}, 
and we make only points different from it.  

(4) Theorem \ref{LM} shows that there exists a unique subfactor $R\subset N$, up to inner conjugacy, such that 
$R'\cap M=\C$ and there exists an outer action $\beta$ of $M_{10}$ on $R$ satisfying 
$$M=R\rtimes_\beta M_{10}\supset N=R\rtimes_\beta M_9.$$ 
The inclusion $L\supset R$ is irreducible. 

(5) To prove that $L\supset R$ is of depth 2, it suffices to show that $[\lambda]$ commutes with 
$$[\id]\oplus [\chi]\oplus 9[\pi]\oplus 9[\pi\chi]\oplus 10[\xi_1]\oplus 10[\eta_1]\oplus 10[\eta_2]\oplus 16[\zeta],$$
which corresponds to the regular representation of $M_{11}$. 
Indeed, it follows from 
\begin{align*}
\lefteqn{[\lambda]([\id]\oplus [\chi]\oplus 9[\pi]\oplus 9[\pi\chi]\oplus 10[\xi_1]\oplus 10[\eta_1]\oplus 10[\eta_2]\oplus 16[\zeta])} \\
 &=[\lambda]\oplus [\lambda\chi]\oplus 9([\lambda]\oplus [\mu])\oplus 9([\lambda\chi]\oplus [\mu])\oplus 10([\mu]\oplus [\xi_2]\oplus [\xi_3]) \\
 &\oplus 10([\mu]\oplus [\nu])\oplus 10([\mu]\oplus [\nu])\oplus 32[\mu] \\
 &=10([\lambda]\oplus [\lambda\chi]\oplus [\xi_2]\oplus [\xi_3]\oplus 2[\nu]\oplus 8[\mu]),
\end{align*}
which is self-conjugate as we can take the conjugate of the both sides. 

(6) We can apply Lemma \ref{inv2} to $Q_8$ to finish the proof.  
\end{proof}

\end{document}